\documentclass{article}
 
\usepackage[latin1]{inputenc}
\usepackage[T1]{fontenc}
\usepackage[francais, english]{babel}
\usepackage{lmodern}
\usepackage{stmaryrd}
\usepackage{amsthm}
\usepackage{amsmath}
\usepackage{amssymb}
\usepackage{mathrsfs}
\usepackage{soul}
\usepackage{layout}
\usepackage{setspace}
\usepackage[top=3.5cm, bottom=3.5cm, left=3 cm, right=3 cm]{geometry}
\usepackage{pgf,tikz}
\usepackage{graphicx}
\usepackage{xcolor}
\usetikzlibrary{arrows}
\usepackage{listings}
\usepackage{enumitem}
\usepackage{hyperref}
\usepackage{cases}
\usepackage{mathabx}
\usepackage{cleveref}
\usepackage{xcolor}
\usepackage{url}
\newtheorem{th.}{Théorème}[section]
\newtheorem{lemme}[th.]{Lemme}
 
\newtheorem*{lemme*}{Lemme} 
\newtheorem*{pgd*}{Principe des grandes déviations} 
\newtheorem{prop}[th.]{Proposition}
\newtheorem*{prop*}{Proposition}
\newtheorem*{prop5.3bis*}{Proposition 5.3 bis}
\newtheorem*{def*}{Définition}
\newtheorem*{cor*}{Corollaire}
\newtheorem{cor.}[th.]{Corollaire}
\newtheorem{cor}[th.]{Corollaire}
\newtheorem*{rem*}{Remarque}
\newtheorem{rem.}[th.]{Remarque}
\newtheorem*{ex.}{Exemple}
\newtheorem*{th*}{Théorème}
\newtheorem*{th11}{Théorème 1.1}

\newtheorem*{th.*}{Théorème}

\newcommand{\T}{\mathbb{T}}
\newcommand{\R}{\mathbb{R}}
\newcommand{\Z}{\mathbb{Z}}
\newcommand{\Q}{\mathbb{Q}}
\newcommand{\N}{\mathbb{N}}

\newcommand{\PP}{\mathscr{P}}
\newcommand{\F}{\PP}
\newcommand{\Pbb}{\mathbb{P}}
\newcommand{\leb}{\text{leb}}
\newcommand{\ad}{\text{ad}}
\newcommand{\ag}{\mathfrak{a}}

\newcommand{\supp}{\text{supp }}

\title{Radon stationary measures for a random walk on  $\T^d\times \R$}
\author{Timothée Bénard}
\date{}

\begin{document}

\maketitle

\bigskip

\begin{abstract}
 We classify Radon stationary measures for a random walk on $\mathbb{T}^d \times \mathbb{R}$. This walk is realised by a random action of $SL_{d}(\mathbb{Z})$ on the $\mathbb{T}^d$ component, coupled with a translation on the $\mathbb{R}$ component. We  show, under assumptions of irreducibility and recurrence, the rigidity and homogeneity of Radon ergodic stationary measures.

\end{abstract}

\selectlanguage{french}

\begin{abstract}

On classifie les mesures de Radon stationnaires pour une marche aléatoire sur $\T^d \times \R$. Cette marche est réalisée par une action aléatoire de $SL_{d}(\Z)$ sur la composante en $\T^d$ couplée à une translation de la coordonnée $\R$. Nous démontrons sous des hypothèses d'irréductibilité et de récurrence la rigidité et l'homogénéité des mesures de Radon stationnaires ergodiques.
  \end{abstract} 
\bigskip

\bigskip

\bigskip

\section*{Introduction}
\large

Ce texte aborde le problème de classification des mesures stationnaires dans le cas d'une marche aléatoire sur un espace homogène de volume infini. Plus précisément, on se donne $G$ un groupe de Lie réel, $\Lambda \subseteq G$ un sous groupe discret, et $\mu \in \mathcal{P}(G)$ une probabilité sur $G$. On peut alors réaliser une marche aléatoire sur le quotient $X=G/\Lambda$ telle que la probabilité de transition à partir d'un point $x \in X$ est  la convolution $\mu\star \delta_{x}$, i.e. la probabilité image de $\mu$ par l'application $G\rightarrow X, g\mapsto g.x$. Comprendre une telle marche participe à décrire l'action de $G$ sur l'espace homogène $X$.
 
Le comportement d'une marche aléatoire est intimement lié à ses mesures stationnaires. Les mesures considérées seront toutes des mesures de Radon, i.e. positives et finies sur les compacts. Une mesure de Radon $\nu \in \mathcal{M}^{Rad}(X)$ est dite \emph{stationnaire} si elle est invariante par itération de la marche ``en moyenne'', ou plus formellement si : $\nu = \int_{G}g_{\star}\nu \,\,d \mu(g)$. On précise qu'elle est \emph{ergodique} si elle appartient à un rayon extrémal du cône convexe des mesures de Radon stationnaires sur $X$. 
 L'ergodicité d'une mesure stationnaire $\nu$ est une information clef pour étudier une marche aléatoire, car elle garantit l'équidistribution de $\nu$-presque toute trajectoire récurrente selon la mesure $\nu$. De plus, par désintégration ergodique, toute mesure stationnaire s'exprime comme moyenne intégrale de mesures stationnaires ergodiques.

\bigskip

En $2013$, Y. Benoist et J-F. Quint font progresser le sujet de manière considérable, en explicitant toutes les probabilités stationnaires pour une large classe de marches en \emph{volume fini}.
%%%%

\begin{th*}[Benoist-Quint \cite{BQII}]
Soit $G$ un groupe de Lie réel, $\Lambda \subseteq G$ un sous groupe discret de covolume fini, $X= G/\Lambda$ et  $\mu \in \mathcal{P}(G)$ une probabilité sur $G$ dont le support est compact et engendre un sous groupe $\Gamma$ tel que $\widebar{\emph{Ad} \Gamma}^Z$ est semi-simple, Zariski connexe et sans facteur compact. Alors toute probabilité $\mu$-stationnaire ergodique $\nu$ sur $X$ est $\Gamma$-invariante et homogène.

\end{th*}

La notation Ad$\Gamma$ désigne l'image de $\Gamma$ par la représentation adjointe de $G$, et $\widebar{\text{Ad} \Gamma}^Z$ son adhérence de Zariski dans le groupe Aut$\mathfrak{g}$ des automorphismes de $\mathfrak{g}$. L'homogénéité signifie que le support de $\nu$ est une orbite de son stabilisateur $G_{\nu}:=\{g\in G, g_{\star}\nu=\nu\}$ (on peut ainsi identifier $\nu$ à une mesure $G_{\nu}$-invariante sur l'espace homogène $G_{\nu}/G_{\nu, x}$).

\bigskip

L'idée clef est le phénomène de  \emph{dérive exponentielle}, découvert  quelques années plus tôt dans \cite{BQI}. La classification des mesures stationnaires vient alors d'être établie par J. Bourgain, A. Furman, E. Lindenstrauss et S. Mozes dans le cas particulier d'une marche sur le tore $\T^d$ induite par une probabilité $\mu \in SL_{d}(\Z)$ proximale fortement irréductible (voir \cite{BFLM-CR}, \cite{BFLM}). Rappelons que la \emph{proximalité} de $\mu$ signifie que le semi-groupe $\Gamma \subseteq SL_{d}(\Z)$ engendré par son support contient un élément ayant une valeur propre simple de module strictement supérieur à celui de ses autres valeurs propres, et que l'hypothèse de \emph{forte irréductibilité} signifie que l'action $\Gamma$ sur $\R^d$ ne préserve pas de réunion finie de sous-espaces vectoriels propres non nuls. La preuve de \cite{BFLM-CR} s'appuie sur la théorie de Fourier et développe une approche quantitative permettant d'obtenir aussi des résultats d'équidistribution des probabilités de position $(\mu^n \star \delta_{x})_{x \in \T^d}$ quand $n$ tend vers l'infini. Cependant, la démarche repose de fa\c con cruciale sur l'hypothèse de proximalité. Y. Benoist et J-F. Quint proposent dans $\cite{BQI}$ une méthode complètement différente qui s'inspire des travaux de M. Ratner et ne requiert pas d'hypothèse de proximalité.  

\begin{th*}[Benoist-Quint \cite{BQI}, Bourgain-Furman-Lindenstrauss-Mozes \cite{BFLM-CR},\cite{BFLM}]
Soit $\mu \in \mathcal{P}(SL_{d}(\Z))$ une probabilité à support fini engendrant un sous semi-groupe $\Gamma \subseteq SL_{d}(\Z)$ fortement irréductible.

Alors toute probabilité $\mu$-stationnaire ergodique sur le tore $\T^d$ est soit une  équiprobabilité sur une $\Gamma$-orbite finie dans $\T^d$, soit la probabilité de Haar. 

\end{th*}

Les preuves des théorèmes  \cite{BQI} et \cite{BQII} ont été  reprises pour étudier des marches aléatoires dans des cadres similaires. Citons par exemple A. Eskin et  M. Mirzakhani qui décrivent les mesures invariantes et stationnaires pour l'action de $SL_{2}(\R)$ sur l'espace des modules des surfaces de translations (\cite{EskMirz}),  A. Brown et F. Rodriguez Hertz qui s'intéressent plus généralement à des marches par difféomorphismes sur des variétés compactes (\cite{BrRod}), ou encore O. Sargent et U. Shapira qui considèrent une marche sur un espace homogène défini par un stabilisateur non discret (\cite{SarShap}).

\bigskip
Ce texte aborde le cas où l'espace homogène supportant la marche est de mesure \emph{infinie}. Peu de résultats sont connus dans cette voie. 
%Une difficulté est que le theorème ergodique en mesure infinie est beaucoup moins maniable et ne permet pas en général de contrôler les temps de retour dans des parties récurrentes. 
On se concentre sur un premier cas concret : classer les mesures de Radon stationnaires pour une marche sur $\T^d \times \R$. On peut en effet réaliser $\T^d \times \R$ comme un espace homogène en posant $G := (SL_{d}(\Z)\ltimes \R^d) \,\oplus\, \R$, $\Lambda := (SL_{d}(\Z)\ltimes \Z^d)\, \oplus\, \{0\}$. On vérifie que $G/\Lambda \equiv \T^d \times \R$ et que via cette identification,  l'action d'un élément $(g,r,s) \in G$ est donnée pour $(x,t) \in \T^d \times \R$ par la formule $(g,r,s).(x,t)=(gx + r, s+t)$. Pour définir notre  marche aléatoire, on fixe une probabilité $\mu \in \mathcal{P}(SL_{d}(\Z))$, on note $\Gamma := \langle \text{supp}\,\mu \rangle \subseteq SL_{d}(\Z)$ le sous semi-groupe engendré par son support,  et on se donne $\chi : \Gamma  \rightarrow \R$ un morphisme de semi-groupes. On peut voir  $\Gamma$ comme un sous semi-groupe de $G$ via le plongement $i : \Gamma \hookrightarrow G, g \mapsto (g,0, \chi(g))$, et l'action de $\Gamma$ sur $G/\Lambda =\T^d\times \R$ est alors donnée par $g.(x,t)=(gx, t+ \chi(g))$. On va démontrer le théorème suivant :

\begin{th11}
Supposons  la probabilité $\mu \in \mathcal{P}(SL_{d}(\Z))$  à support fini engendrant un sous semi-groupe $\Gamma \subseteq SL_{d}(\Z)$ fortement irréductible, et supposons la probabilité image $\chi_{\star} \mu \in \mathcal{P}(\R)$ d'espérance nulle. 

Alors toute mesure de Radon $\mu$-stationnaire ergodique sur $\T^d \times \R$ est $\Gamma$-invariante et homogène. 
\end{th11}

Le \cref{TH0}  décrit aussi les probabilités stationnaires pour certaines marches sur $\T^d$ qui n'entrent pas dans le cadre du théorème de Benoist-Quint car induites par une probabilité sur $SL_{d}(\Z)$ qui n'a pas de moment d'ordre $1$. En effet, considérons le cas où $\chi(\Gamma)= \Z$. On peut restreindre notre marche sur $\T^d \times \R$ en une marche sur $\T^d \times \Z$, puis considérer la marche induite sur le bloc $\T^d \times \{0\} \equiv \T^d$. Cette marche est donnée par une probabilité $\mu_{\tau} \in \mathcal{P}(SL_{d}(\Z))$ sans  moment d'ordre $1$. Le \cref{TH0} implique que toute probabilité $\mu_{\tau}$-stationnaire ergodique sur $\T^d$ est soit atomique, soit la probabilité de Haar. 

Notre résultat peut  être adapté pour décrire les mesures de Radon stationnaires sur des produits plus généraux, de la forme $G/\Lambda \times \R^k$ où $G$ est un groupe algébrique réel simple, $\Lambda$ un réseau  dans $G$ et $k\in\{1,2\}$. On doit alors supposer que le semi-groupe $\Gamma$ est Zariski-dense dans $G$. Si $k\geq 3$ il apparaît  des difficultés venant du fait qu'une marche aléatoire, même centrée, sur $\R^3$ n'est pas nécessairement récurrente.

%La méthode employée s'inspirera  de celle de Benoist-Quint \cite{BQII} mais nécessitera des adaptations pour gérer le volume infini de $G/\Lambda=\T^d \times \R$ (principalement pour traiter le cas des mesures atomiques (section $3$), la non dégénérescence des mesures limites (section $5$) et le contrôle de la fenêtre (section $7$)).

\bigskip
La démonstration du \cref{TH0} s'inspirera  de  \cite{BQII} mais nécessitera des adaptations pour gérer le volume infini de $G/\Lambda=\T^d \times \R$.   Le texte s'organise ainsi :

\bigskip

\emph{Section 1} : On présente une version plus détaillée du \cref{TH0}, et on explique  l'intérêt d'avoir supposé le paramètre de translation réelle $\chi$ de $\mu$-moyenne nulle. 

\emph{Section 2} : On prouve qu'une mesure de Radon stationnaire sur $\T^d\times \R$ admet une décomposition en mesures limites. Cela signifie qu'elle est une moyenne intégrale de mesures indexées par les trajectoires de la marche et vérifiant une certaine propriété d'équivariance. La section 2  étend des méthodes d'habitude employées pour l'étude de probabilités stationnaires.

\emph{Section 3} : On démontre le \cref{TH0} pour des mesures de Radon stationnaires ergodiques $\nu$ sur $\T^d\times \R$ dont le projeté sur $\T^d$ est atomique. Nous aurons besoin d'exclure ce cas par la suite. La preuve repose sur les propriétés de récurrence de la marche.

\bigskip

\emph{A partir de la section $4$, on fixe une mesure $\nu\in \mathcal{M}^{Rad}(\T^d\times \R)$ stationnaire ergodique dont le projeté sur $\T^d$ est sans atome.} 

\bigskip

\emph{Section 4}: Après quelques rappels sur les groupes algébriques, on explique la stratégie mise en oeuvre pour décrire la mesure $\nu$. L'idée est de montrer que les mesures limites décomposant $\nu$ sont chacune invariante par translation sur $\T^d$ suivant une certaine  direction limite. Pour cela, on définit un système dynamique fibré $(B^+, T^+, \beta^+)$ (de mesure infinie) et un flot $(\phi_{t})_{t\in \R^r}$ sur $B^+$ tels que l'invariance recherchée équivaut à l'invariance de la mesure $\beta^+$ sous l'action du flot $(\phi_{t})_{t\in\R^r}$. Après désintégration de $\beta^+$ le long de $(\phi_{t})_{t\in\R^r}$, on se ramène à un énoncé général sur $(B^+, T^+, \beta^+)$, appelé \emph{théorème de dérive exponentielle} dont la démonstration occupe le reste du texte. La preuve repose sur l'étude des fibres de l'opérateur  $T^+$, plus précisément de leur intersection avec une fenêtre $W\subseteq B^+$ de mesure finie fixée au préalable. 

\emph{Section 5}: On montre que les mesures limites de $\nu$ sur $\T^d\times \R$ sont non dégénérées, i.e. que leurs projections sur $\T^d$ sont  sans atomes. Cela permettra plus tard de choisir deux points de $B^+$ dont les fibres s'écartent à vitesse exponentielle. La preuve distingue les cas où le semi-groupe de translation $\chi(\Gamma) \subseteq \R$ est discret ou dense. Dans le premier cas, on se ramène à montrer la récurrence hors de $\{0\}$ pour une marche sur le tore. La difficulté est que cette marche n'a pas de moment d'ordre $1$. Dans le second cas, on raisonne par l'absurde et on utilise la récurrence de la marche et un argument de connexité.

\emph{Section 6} : On explicite les fibres de l'opérateur $T^+$ puis on donne une formule générale décrivant leur distribution au sein d'une fenêtre $W\subseteq B^+$ de mesure finie, du point de vue de la mesure restreinte $\beta^+_{|W}$. C'est \emph{l'équirépartition des morceaux de fibres}. 

\emph{Section 7} : On prouve un théorème local limite pour le cocycle $(\sigma, \chi)$ où $\sigma$ désigne le cocycle d'Iwasawa et $\chi$ le paramètre de translation réelle de la marche.  Grâce à la section $6$, cela permet d'estimer  selon quelle proportion une fibre du système dynamique $(B^+, T^+, \beta^+)$ rencontre la fenêtre $W \subseteq B^+$.

\emph{Section 8} : On démontre la \emph{loi des angles} qui permet de contrôler l'écart (en termes de norme et de direction) entre deux morceaux de  fibres du système dynamique $(B^+, T^+, \beta^+)$.

\emph{Section 9} : A l'aide des sections 5,6,7,8, on prouve le théorème de dérive exponentielle.

\emph{Section 10} : On explique comment le théorème de dérive exponentielle permet de conclure la preuve du \cref{TH0}.

\bigskip

\newpage

\tableofcontents

%%%%%%%%%%%%%%%%%%%%%%%%%%%%%%%%%%%%%%%%%%%%%%%%%%%%%%%%%%%%%%%%%%%%%%%%%%%%%%%%%%%%%%%%%%%%%%%%%%%%%%%%%%%%%%%%%%%%%%%%%%%%%%%%%%%%%%%%%%%%%%%%%%%%%%%%%%%%%%%%%%%%%%%%%%%%%%%%%%%%%%%%%%%%%%%%%%%%%%%%%%%%%%%%%%%%%%%%%%%%%%%%%%%%%%%%%%%%%%%%%%%%%%%%%%%%%%%%%%%%%%%%%%%%%%%%%%%%%%%%%%%%%%%%%%%%%%%%%%%%%%%%%%%%%%%%%%%%%%%%%%%%%%%%%%%%%%%%%%%%%%%%%%%%%%%%%%%%%%%%%%%%%%%%%%%%%%%%%%%%%%%%%%%%%%%%%%%%%%%%%%%%%%%%%%%%%%%%%%%%%%%%%%%%%%%%%%%%%%%%%%%%%%%%%%%%%%%%%%%%%%%%%%%%%%%%%%%%%%%%%%%%%%%%%%%%%%%%%%%%%%%%%%%%%%%%%%%%%%%%%%%%%%%%%%%%%%%%%%%%%%%%%%%%%%%%%%%%%%%%%%%%%%%%%%%%%%%%%%%%%%%%%%%%%%%%%%%%%%%%%%%%%%%%%%%%%%
%%%%%%%%%%%%%%%%%%%%%%%%%%%%%%%%%%%%%%%%%%%%%%%%%%%%%%%%%%%%%%%%%%%%%%%%%%%%%%%%%%%%%%%%%%%%%%%%%%%%%%%%%%%%%%%%%%%%%%%%%%%%%%%%%%%%%%%%%%%%%%%%%%%%%%%%%%%%%%%%%%%%%%%%%%%%%%%%%%%%%%%%%%%%%%%%%%%%%%%%%%%%%%%%%%%%%%%%%%%%%%%%%%%%%%%%%%%%%%%%%%%%%%%%%%%%%%%%%%%%%%%%%%%%%%%%%%%%%%%%%%%%%%%%%%%%%%%%%%%%%%%%%%%

\newpage

 \section{Premières considérations}
 On présente une version détaillée du \cref{TH0}, et on explique  l'intérêt d'avoir supposé le paramètre de translation réelle $\chi$ de $\mu$-moyenne nulle. 

\bigskip

{\bf Notations}. On se donne $\mu \in \mathcal{P}(SL_{d}(\Z))$  une probabilité  sur $SL_{d}(\Z)$, on note $\Gamma := \langle \text{supp}\,\mu \rangle$  le sous semi-groupe de $SL_{d}(\Z)$ engendré par son support,  et on fixe $\chi : \Gamma \rightarrow \R$ un morphisme de semi-groupes. Ces données induisent une action de $\Gamma$ sur $\T^d \times \R$ via la formule $g.(x,t)=(gx, t + \chi(g))$, puis une marche aléatoire de probabilités de transitions données par $(\mu \star \delta_{(x,t)})_{(x,t) \in \T^d\times \R}$.  
 
 Notons $\mathcal{M}^{Rad}(\T^d\times \R)$ l'ensemble des mesures de Radon sur $\T^d\times \R$.  On dit qu'une mesure $\nu \in \mathcal{M}^{Rad}(\T^d\times \R)$ est  \emph{$\mu$-stationnaire} si $\nu = \int_{\Gamma} g_{\star} \nu \,d\mu(g)$.  On précise alors que $\nu$ est \emph{ergodique} si elle appartient à un rayon extrémal du cône des mesures de Radon $\mu$-stationnaires sur $\T^d\times \R$.

\bigskip

L'objectif de ce texte est de classer les mesures de Radon stationnaires pour une marche sur $\T^d \times \R$ :

\begin{th.}\label{TH0}
Soit  $\mu \in \mathcal{P}(SL_{d}(\Z))$ une probabilité à support fini engendrant un sous semi-groupe $\Gamma \subseteq SL_{d}(\Z)$ fortement irréductible, et soit $\chi : \Gamma \rightarrow \R$ un morphisme de semi-groupes tel que la probabilité $\chi_{\star}\mu \in \mathcal{P}(\R)$ est centrée. Alors toute mesure de Radon $\mu$-stationnaire ergodique sur $\T^d \times \R$ est $\Gamma$-invariante et homogène. 
\end{th.}

La \emph{forte irréductibilité} de $\Gamma$ signifie que l'action de $\Gamma$ sur $\R^d$ ne préserve pas de réunion finie de sous-espaces vectoriels propres non nuls. L'hypothèse $\chi_{\star}\mu$ \emph{centrée} signifie que ($\chi_{\star}\mu$ est a un moment d'ordre $1$ et) $\int_{\R}t\,\,d\chi_{\star}\mu(t)=0$. Enfin, l'\emph{homogénéité} d'une mesure de Radon $\nu$ sur $\T^d\times \R$  signifie que le support de $\nu$ est une orbite de son stabilisateur $G_{\nu}:=\{g\in (SL_{d}(\Z)\ltimes \R^d) \oplus\, \R, \,\,\,g_{\star}\nu=\nu\}$ (on peut ainsi identifier $\nu$ à une mesure $G_{\nu}$-invariante sur l'espace homogène $G_{\nu}/G_{\nu, x}$).

\bigskip

On constate que ce théorème décrit les mesures de Radon stationaires ergodiques, mais plus généralement toutes les mesures de Radon stationnaires sur $\T^d \times \R$ car elles s'écrivent comme une moyenne de mesures ergodiques (voir \cite{Bony}). Il en découle notamment la rigidité des mesures stationnaires sur $\T^d\times \R$ :

\begin{cor}
Dans le cadre du \cref{TH0}, toute mesure de Radon $\mu$-stationnaire sur $\T^d\times \R$ est $\Gamma$-invariante. 
\end{cor}

\bigskip

La preuve du \cref{TH0} fournira une classification plus explicite :

\begin{th.}\label{TH}
Soit  $\mu \in \mathcal{P}(SL_{d}(\Z))$ une probabilité à support fini engendrant un sous semi-groupe $\Gamma \subseteq SL_{d}(\Z)$ fortement irréductible et soit $\chi : \Gamma \rightarrow \R$ un morphisme de semi-groupes tel que la probabilité $\chi_{\star}\mu \in \mathcal{P}(\R)$ est centrée. On se donne une mesure de Radon $\mu$-stationnaire ergodique $\nu$ sur $\T^d \times \R$. Il y a plusieurs  cas:

\begin{itemize}
\item[$1$er] cas : il  existe un point $x \in \T^d$ pour lequel $\nu(x \times \R) >0$. Alors ce point est rationnel, on note $\omega := \Gamma x \subseteq \T^d$ sa $\Gamma$-orbite (finie) dans $\T^d$. Si $\chi(\Gamma) \subseteq \R$ est discret, alors $\nu$ est une mesure uniforme sur une $\Gamma$-orbite incluse dans $\omega \times \chi(\Gamma)$ ou un translaté. Si $\chi(\Gamma) \subseteq \R$ est dense, alors $\nu = \nu_{\omega} \otimes \text{leb}$ où $\nu_{\omega} \in \mathcal{M}^f(\T^d)$ est une mesure uniforme sur $\omega$.

\item[$2$ème]cas : la mesure projetée $\nu(. \times \R) \in \mathcal{M}^{Rad}(\T^d)$ est sans atome. La mesure $\nu$ est alors une mesure de Haar sur $\T^d \times \widebar{\chi(\Gamma)}$ ou un translaté. 
\end{itemize}
\end{th.}

\bigskip

Nous allons maintenant mettre en lumière quelques remarques essentielles pour la suite du texte, à savoir que la $\mu$-marche que l'on considère est récurrente presque sûre, et que les mesures de Radon $\mu$-stationnaires  sur $\T^d\times \R$ projetées sur l'axe réel sont $\chi(\Gamma)$-invariantes. Le lemme général derrière ce phénomène est le suivant :

\begin{lemme} \label{ref0}
Soit $m \in \mathcal{P}(\R)$ une probabilité centrée sur $\R$, notons $\Gamma_{m} \subseteq \R$ le semi-groupe engendré par son support. Alors :
\begin{itemize}
\item Pour tout $t_{0} \in \R$, presque toute $m$-trajectoire issue de $t_{0}$ revient arbitrairement proche de $t_{0}$. 
\item Toute mesure de Radon $m$-stationnaire sur $\R$  est $\Gamma_{m}$-invariante (pour l'action par translation). 
\end{itemize}
 En particulier, si $\Gamma_{m}$ est discret dans $\R$, alors $\Gamma_{m}$ est un sous groupe de $\R$. Si $\Gamma_{m}$ n'est pas discret, alors $\Gamma_{m}$ est dense dans $\R$ et la mesure de Lebesgue est la seule mesure de Radon $m$-stationnaire sur $\R$ (à scalaire près). 

\end{lemme}

\begin{proof}
Pour prouver le premier point, donnons nous $(\Omega, \mathcal{F}, \mathbb{P})$ un espace probabilisé et  $X_n : \Omega \rightarrow \R$ des variables aléatoires indépendantes de loi $m$. La loi des grands nombres affirme que presque pour presque tout $\omega \in \Omega$, on a   $X_{1}(\omega)+ \dots + X_{n}(\omega) = o(n)$. Le lemme $3.18$ de \cite{BQRW} sur la divergence des sommes de Birkhoff affirme que pour presque tout $\omega \in \Omega$, $(X_{1}(\omega)+ \dots + X_{n}(\omega))_{n \geq 0}$ admet une sous suite tendant vers $0$, ce qui donne le résultat.

Le deuxième point est un corollaire du théorème de Choquet-Deny (\cite{ChoDen}). On en donne une démonstration dans ce cas particulier. Soit $\nu \in \mathcal{M}^{Rad}(\R)$ $m$-stationnaire. Notons $B = \R^{\N^\star}$, $\beta= m^{\N^\star}$. Le théorème de convergence presque sûre des martingales positives permet de définir $\beta$-presque sûrement ses mesures limites $\nu_{b}= \lim (b_{n}\dots b_{1})_{\star}\nu \in \mathcal{M}^{Rad}(\R)$ (voir \cref{meslim}). D'après le premier point, nous avons $\nu_{b}= \nu$ $\beta$-ps. Par ailleurs, les $\nu_{b}$ vérifient la relation d'équivariance  : $b_{1\star}\nu_{b_{2}b_{3}\dots} = \nu_{b}$ $\beta$-ps. On en déduit que $\nu$ est invariante par $\mu$-presque tout élement de $\R$ d'où le résultat.  

Dans le cas où $\Gamma_{m}$ est discret, le premier point donne que $\Gamma_{m}$ est stable par passage à l'opposé, donc est un groupe. Si le semi-groupe $\Gamma_{m}$ n'est pas discret, le premier point assure qu'il est dense dans $\R$. Dans ce cas, soit $\lambda$ une mesure de Radon $\Gamma_{m}$ invariante sur $\R$. Comme le stabilisateur de $\lambda$ dans $\R$  est fermé, on a $\lambda$ $\R$-invariante,  donc  multiple de la mesure de Lebesgue.
\end{proof}

On en déduit les propriétés annoncées plus haut.

\begin{cor}\label{premcor}
Dans le cadre du \cref{TH0} :  Soit $\nu \in \mathcal{M}^{Rad}(\T^d\times \R)$ une mesure de Radon $\mu$-stationnaire sur $\T^d\times \R$. Alors  
\begin{itemize}
\item $\nu$ est conservative : pour $\nu$-presque tout $(x_{0},t_{0})\in \T^d\times \R$, presque toute $\mu$-trajectoire issue de $(x_{0},t_{0})$ revient arbitrairement proche de $(x_{0},t_{0})$. 
\item  La projection de $\nu$ sur $\R$, notée par la suite $\nu(\T^d\times .)\in \mathcal{M}^{Rad}(\R)$, est $\chi(\Gamma)$-invariante (et en particulier un multiple de la mesure de Lebesgue si $\chi(\Gamma)$ est dense dans $\R$).
\end{itemize}  
\end{cor}

\begin{proof}
La conservativité de $\nu$ découle du premier point dans le \cref{ref0} appliqué à la probabilité centrée $\chi_{\star}\mu$. En effet, ce dernier garantit que pour tout ouvert borné de $\T^d\times \R$ de la forme $U:=\T^d\times]-r,r[$ avec $r >0$, presque toute $\mu$-trajectoire issue de $U$ revient dans $U$. Cela permet de considérer la marche induite sur $U$, qui est nécessairement récurrente par le théorème de récurrence de Poincaré. 

Par ailleurs, la mesure projetée $\nu(\T^d\times .) \in \mathcal{M}^{Rad}(\R)$ est $\chi_{\star}\mu$-stationnaire donc $\chi(\Gamma)$-invariante par le \cref{ref0}.

\end{proof}

Enfin, notons comme corollaire de la conservativité le résultat suivant :

\begin{cor}\label{nucons}
Soit $\nu, \nu' \in \mathcal{M}^{Rad}(\T^d\times \R)$ des mesures de Radon $\mu$-stationnaires sur $\T^d\times\R$. Supposons $\nu$ ergodique et $\nu' <<\nu$. Alors il existe une constante $c\geq0$ telle que $$\nu'=c\,\nu$$ 
\end{cor}

\begin{rem.} En mesure infinie, cette propriété ne découle pas de l'ergodicité. Par exemple, la $m$-marche sur $\Z$ donnée par $m:=\frac{1}{3}\delta_{-1}+\frac{2}{3}\delta_{1} \in \mathcal{P}(\Z)$ admet deux mesures de Radons stationnaires ergodiques équivalentes non proportionnelles $$\nu_{0}:=\sum_{k\in \Z}\delta_{k}\,\,\,\text{et}\,\,\, \nu:=\sum_{k\in \Z} 2^k\delta_{k}$$.

\end{rem.}

\begin{proof}[Preuve du  \cref{nucons}]
Il suffit de montrer que pour tout $r> 0$, il existe $c\geq 0$ tel que  $$\nu'_{|\T^d\times [-r, r]} =c \,\nu_{|\T^d\times [-r, r]}$$
Or les mesures restreintes $\nu'_{|\T^d\times [-r, r]}$, $\nu'_{|\T^d\times [-r, r]}$ sont finies et stationnaires pour la marche de premier-retour sur $\T^d\times [-r, r]$, avec $\nu_{|\T^d\times [-r, r]}$ ergodique. Par la théorie ergodique des marches aléatoires en mesure finie on en déduit l'égalité annoncée. Pour plus de détails, on pourra consulter \cite{AarET} (section 1.5) et \cite{BQRW} (proposition 2.9). 

\end{proof}

\begin{rem.}\label{equivergo}
 Cette preuve s'adapte pour montrer qu'une mesure $\nu$ de Radon $\mu$-stationnaire sur $\T^d\times \R$ est ergodique si et seulement pour toute partie mesurable $A\subseteq \T^d\times \R$ telle que $\mu\star 1_{A}=1_{A}$ $\nu$-pp, on a $\nu(A)=0$ ou $\nu(\T^d\times \R-A)=0$. On fait ainsi le lien avec une autre définition courante de l'ergodicité.
\end{rem.}

\newpage

\section{Mesures limites}

La notion de mesure limite constitue un outil fondamental pour étudier les mesures stationnaires. Il s'agit d'écrire une mesure stationnaire comme moyenne intégrale de mesures indexées par les trajectoires de la marche et qui vérifient une propriété d'équivariance. Ce point de vue a été introduit par Furstenberg pour étudier des probabilités stationnaires. La section qui suit étend la notion de mesure limite au cas d'une mesure stationnaire de Radon (infinie) sur $\T^d\times \R$. Elle s'appuie sur l'annexe A qui traite  des mesures limites en général, pour une mesure de Radon stationnaire sur un espace localement compact à base dénombrable.

 La preuve du \cref{TH0}  consistera (essentiellement) à montrer que ces mesures limites sont invariantes par translation de la coordonnée en $\T^d$ selon des directions denses dans le tore. 

\bigskip

On reprend les notations de la section $1$. En particulier $\mu$ est une probabilité sur $\Gamma \subseteq SL_{d}(\Z)$ dont le poussé en avant $\chi_{\star}\mu \in \mathcal{P}(\R)$ est centré. On introduit $B := \Gamma^{\N^\star}$, $\beta:= \mu^{\otimes \N^\star} \in \mathcal{P}(B)$, $T : B \rightarrow B, \,b=(b_{i})_{i \geq 1} \mapsto (b_{i+1})_{i\geq 1}$ le shift unilatère. Le lemme suivant définit les mesures limites annoncées.

\begin{lemme}\label{torenu_{b}}
Soit $\nu \in \mathcal{M}^{Rad}(\T^d\times\R)$ une mesure de Radon $\mu$-stationnaire sur $\T^d\times \R$. Il existe une application mesurable $B \rightarrow \mathcal{M}^{Rad}(\T^d\times\R), b \mapsto \nu_{b}$ telle que pour $\beta$-presque tout $b \in B$, on a la convergence $(b_{1}\dots b_{n})_{\star} \nu \rightarrow \nu_{b}$ pour la topologie faible-$\star$. Cette application est unique en dehors d'un ensemble $\beta$-négligeable. De plus, 
\begin{itemize}
\item Pour $\beta$-presque tout $b\in B$, on a $(b_{1})_{\star}\nu_{Tb} = \nu_{b}$
\item $\nu = \int_{B} \nu_{b} \,\,d \beta(b)$
\item Pour tout intervalle $I \subseteq \R$ et $\beta$-presque tout $b \in B$, on a $\nu(\T^d \times I)= \nu_{b}(\T^d \times I)$
\end{itemize}
Ces mesures $(\nu_{b})_{b \in B}$ sont appelées les \emph{mesures limites associées à $\nu$}.
\end{lemme}

\bigskip

\begin{proof}
La convergence qui définit les $\nu_{b}$, la relation d'équivariance et l'inégalité  $\nu \geq \int_{B} \nu_{b} \,\,d \beta(b)$ sont démontrées dans le \cref{meslim} de l'annexe. Il reste à vérifier le troisième point. Fixons $I \subseteq \R$ un intervalle borné. 

\emph{1er cas: $\chi(\Gamma)\subseteq \R$ discret}. D'après le \cref{ref0}, $\chi(\Gamma)$ est un groupe, on peut donc supposer $\chi(\Gamma)=\Z$ puis $\nu$ portée sur $\T^d \times \Z$ par ergodicité. D'après l'inégalité $\nu \geq \int_{B} \nu_{b} \,\,d \beta(b)$, le résultat est vrai si $I$ ne rencontre pas $\Z$.  On peut donc supposer $I =]k-1/2, k+1/2[$ où $k \in \Z$. Soit $b \in B$ tel que $(b_{1}\dots b_{n})_{\star}\nu\rightarrow \nu_{b}$, $\nu_{b}(\T^d \times \partial I)=0$ et tel qu'il existe une extraction $\sigma : \N \rightarrow \N$ vérifiant $\chi(b_{1}\dots b_{\sigma(n)}) = 0$. On a  alors $\nu_{b}(\T^d\times I)= \lim \nu((b_{1}\dots b_{\sigma(n)})^{-1}  \T^d \times I)= \nu(\T^d\times I)$. 

\emph{2ème cas: $\chi(\Gamma)\subseteq \R$ dense.}  La mesure $\nu(\T^d\times .) \in \mathcal{M}^{Rad}(\R)$ est un multiple de la mesure de lebesgue (\cref{premcor}). 
En particulier, $\nu(\T^d \times \partial I)=0$. On se donne $b \in B$ tel que $(b_{1}\dots b_{n})_{\star}\nu\rightarrow \nu_{b}$, $\,\nu_{b}(\T^d \times \partial I)=0$ et tel qu'il existe une extraction $\sigma : \N \rightarrow \N$ vérifiant $\chi(b_{1}\dots b_{\sigma(n)}) \rightarrow 0$. Alors $\nu_{b}(\T^d\times I)= \lim \nu((b_{1}\dots b_{\sigma(n)})^{-1}  \T^d \times I)=\lim \nu(\T^d\times I-\chi(b_{1}\dots b_{\sigma(n)}) =\nu(\T^d \times I)$.

\end{proof}
On donne un premier corollaire de cette définition, qui autorise de modifier la probabilité $\mu$ induisant la marche :

\emph{Stationnarité et temps d'arrêt.}
Une mesure $\mu$-stationnaire est aussi $\mu_{\tau}$-stationnaire où $\mu_{\tau}$ désigne la mesure $\mu$ conditionnée par un temps d'arrêt $\tau$. Plus précisément, soit $\mathcal{B}_{n} \subseteq \mathcal{B}$ la sous tribu des $n$ premières coordonnées,   $\tau : B \rightarrow \N\cup\{\infty\}$ un temps d'arrêt pour la filtration $(\mathcal{B}_{n})_{n \geq 0}$. On suppose $\tau$ fini $\beta$-ps. On note $\mu_{\tau} \in \mathcal{P}(\Gamma)$ la loi de $b_{1}\dots b_{\tau(b)}$ quand $b$ varie selon $\beta$.
\begin{cor}\label{LIMcor}
La mesure $\nu$ est $\mu_{\tau}$-stationnaire.
\end{cor}

\begin{proof}
Cela découle de l'égalité $\nu = \int_{B}\nu_{b}\,d\beta(b)$ et du  \cref{cond} de l'annexe A. 
\end{proof}

Dans le cas où $\chi(\Gamma) \subseteq \R$ est discret, nous serons souvent amenés à conditionner $\mu$ par rapport au temps d'arrêt  $\tau : B \rightarrow \N \cup \{\infty\}, b \mapsto \inf\{n \geq 1, \chi(b_{1}\dots b_{n})=0\}$ qui donne le temps de premier retour d'une trajectoire à son bloc de départ ($\beta$-presque sûrement fini d'après le \cref{premcor}). On note $\Gamma_{0} := \{g\in \Gamma, \chi(g)=0\}$ le semi-groupe engendré par le support de $\mu_{\tau}$. Le lemme suivant nous sera utile.

\begin{lemme} \label{0irr}
Le groupe $\Gamma_{0}$ est fortement irréductible sur $\R^d$. 
\end{lemme}

\begin{proof}

  Soit $g, h \in \Gamma$. Comme $\chi(\Gamma)$ est un groupe (\cref{ref0}), il existe $g',h' \in \Gamma$ tels que $\chi(g')= -\chi(g)$, $\chi(h')= -\chi(h)$. Ainsi le commutateur $[g,h]= (ghg'h')(hgg'h')^{-1}\in \Gamma_{0}\Gamma^{-1}_{0} \subseteq \widebar{\Gamma}^Z_{0}$. On en déduit que $ \widebar{\Gamma}^Z_{0} \supseteq \widebar{[\Gamma, \Gamma]}^Z=[G,G] \supseteq G_{c}$ où $G=\widebar{\Gamma}^Z$ est l'adhérence de Zariski de $\Gamma$ (semi-simple). Comme $\Gamma$ est fortement irréductible par hypothèse, c'est aussi le cas de $G$, puis de $\widebar{\Gamma}^Z_{0}$ et $\Gamma_{0}$.

\end{proof}

%%%%%%%%%%%%%%%%%%%%%%%%%%%%%%%%%%%%%%%%%%%%%%%%%%%%%%%%%%%%%%%%%%%%%%%%%%%%%%%%%%%%%%%%%%%%%%%%%%%%%%%%%%%%%%%%%%%%%%%%%%%%%%%%%%%%%%%%%%%%%%%%%%%%%%%%%%%%%%%%%%%%%%%%%%%%%%%%%%%%%%%%%%%%%%%%%%%%%%%%%%%%%%%%%%%%%%%%%%%%%%%%%%%%%%%%%%%%%%%%%%%%%%%%%%%%%%%%%%%%%%%%%%%%%%%%%%%%%%%%%%%%%%%%%%%%%%%%%%%%%%%%%%%
%%%%%%%%%%%%%%%%%%%%%%%%%%%%%%%%%%%%%%%%%%%%%%%%%%%%%%%%%%%%%%%%%%%%%%%%%%%%%%%%%%%%%%%%%%%%%%%%%%%%%%%%%%%%%%%%%%%%%%%%%%%%%%%%%%%%%%%%%%%%%%%%%%%%%%%%%%%%%%%%%%%%%%%%%%%%%%%%%%%%%%%%%%%%%%%%%%%%%%%%%%%%%%%%%%%%%%%%%%%%%%%%%%%%%%%%%%%%%%%%%%%%%%%%%%%%%%%%%%%%%%%%%%%%%%%%%%%%%%%%%%%%%%%%%%%%%%%%%%%%%%%%%%%
\newpage

\section{Cas atomique}

Dans cette section, on prouve le \cref{TH}  dans le cas où  la mesure projetée $\nu(. \times \R) \in \mathcal{M}(\T^d)$ admet un atome. Les notations sont celles des sections précédentes. 

\begin{prop}[cas atomique]\label{ThmAT}

Soit  $\mu \in \mathcal{P}(SL_{d}(\Z))$ une probabilité  à support fini engendrant un sous semi-groupe $\Gamma \subseteq SL_{d}(\Z)$ fortement irréductible,  et soit $\chi : \Gamma \rightarrow \R$ un morphisme de semi-groupes tel que la probabilité image $\chi_{\star}\mu \in \mathcal{P}(\R)$ est centrée. On se donne une mesure de Radon $\mu$-stationnaire ergodique $\nu$ sur $\T^d \times \R$ tel qu'il existe $x \in \T^d$ pour lequel $\nu(\{x \}\times \R) \in ]0, +\infty]$. Alors le point $x$ est rationnel. On note $\omega := \Gamma x \subseteq \T^d$ sa $\Gamma$-orbite. Il y a deux cas:

\begin{enumerate}
\item Si $\chi(\Gamma) \subseteq \R$ est discret, alors $\nu$ est une mesure uniforme sur une $\Gamma$-orbite incluse dans $\omega \times \chi(\Gamma)$ ou un translaté.

\item Si $\chi(\Gamma) \subseteq \R$ est dense, alors $\nu = \nu_{\omega} \otimes \text{leb}$ où $\nu_{\omega} \in \mathcal{M}^f(\T^d)$ est une mesure uniforme sur $\omega$. \end{enumerate}

\end{prop}

\bigskip

\begin{rem*}
Dans le premier cas de la proposition, le sous ensemble $\omega \times \chi(\Gamma)\subseteq \T^d \times \R$ est stable par $\Gamma$ mais peut contenir plusieurs $\Gamma$-orbites. Par exemple, considérons les matrices $g_{0}:= \begin{pmatrix}1 & 2\\ 0& 1\\ \end{pmatrix}$, $g_{1}:= \begin{pmatrix}1 & 0\\ 2& 1\\ \end{pmatrix}$ et $\mu:= \frac{1}{4}(\delta_{g_{0}} +\delta_{g^{-1}_{0}} + \delta_{g_{1}} + \delta_{g^{-1}_{1}})$. Le semi-groupe $\Gamma$ engendré par le support de $\mu$ est un sous groupe de $SL_{2}(\Z)$, librement engendré par $g_{0}, g_{1}$, et Zariski dense dans $SL_{2}(\R)$. On note $\chi : \Gamma \rightarrow \Z$ le morphisme défini par $\chi(g_{0})=0$, $\chi(g_{1})=1$. Le point $x = (\frac{1}{4}, 0 ) \in \T^2$ est tel que $g_{0}.x=x$, $g_{1}.x \neq x$, $g_{0}g_{1}x=g_{1}x$, $g^2_{1} .x = x$. Ainsi $\Gamma x= \{x, g_{1}x\}$ dans $\T^2$.  Finalement,  dans $\T^2 \times \R$,  l'ensemble $\Gamma . (x, 0):= \{(y,k) \in \T^2 \times \Z, \,y= x \,\text{si $k$ pair}, \,\,y = g_{1}x\, \text{sinon}\}$ est une $\Gamma$-orbite strictement incluse dans $\Gamma x \times \chi(\Gamma)$.
\end{rem*}

\bigskip
La suite de la section est consacrée à la preuve de la proposition. 

\bigskip

\emph{$1)$ Cas où $\chi(\Gamma) \subseteq \R$ est discret}. 

\bigskip

Rappelons que $\chi(\Gamma)$ est alors un sous groupe de $\R$ (via le \cref{ref0}). Le cas où $\chi(\Gamma)=\{0\}$ est celui traité par Benoist-Quint. On peut supposer $\chi(\Gamma)=\Z$ et $\nu$ portée par $\T^d \times \Z$. Il existe alors $k \in \Z$ tel que $\nu(x \times \{k\})>0$. On note $D_{k}= \T^d \times \{k\}$ le $k$-ième bloc de $\T^d \times \Z$, et $\tau : B \rightarrow \llbracket 1, +\infty \rrbracket, b \mapsto \inf\{n \geq 1, \chi(b_{1}\dots b_{n})=0\}$ le temps que met une trajectoire pour revenir au bloc duquel elle est partie. C'est un temps d'arrêt fini $\beta$-ps. On note $\mu_{\tau} \in \mathcal{P}(\Gamma)$ la loi de $b_{\tau(b)}\dots b_{1}$ (qui est aussi celle de $b_{1}\dots b_{\tau(b)}$). La loi $\mu_{\tau}$ régit la chaîne de Markov induite sur $D_{k}$ donc $\nu_{|D_{k}}$ est $\mu_{\tau}$-stationnaire ergodique. C'est aussi le cas de sa partie atomique de poids maximal. On en déduit que $\nu_{|D_{k}}$ est une mesure uniforme sur une $\Gamma_{0}$-orbite finie, où  $\Gamma_{0} = \{g\in \Gamma, \chi(g)=0 \}\subseteq \Gamma$ est le semi-groupe engendré par le support de $\mu_{\tau}$.

Le \cref{0irr} assure que l'action de $\Gamma_{0}$ sur $\R^d$  est irréductible donc  que les atomes de $\nu_{|D_{k}}$ sont des points rationnels dans $\T^d$.
Notons $\omega=\Gamma.x \subseteq \mathbb{Q}^d/ \Z^d$ la $\Gamma$-orbite de $x$ dans $\T^d$. La mesure $\nu'= \sum_{y\in \omega} \delta_{y} \otimes \sum_{k \in \Z} \delta_{k} \in \mathcal{M}^{Rad}(\T^d\times \R)$ est $\mu$-stationnaire et $\nu << \nu'$. Le support de $\nu$ étant dénombrable, les mesures $\nu'_{|\text{supp} \nu}$ et $\nu$ sont donc équivalentes puis proportionnelles (par le \cref{nucons} et l'ergodicité de $\nu$), ce qui conclut.

\bigskip

\bigskip

\emph{$2)$ Cas où $\chi(\Gamma) \subseteq \R$ est dense}. 

\bigskip

\begin{lemme}
La mesure $\nu$ est $\Gamma$-invariante. 
\end{lemme}

\begin{proof}
Il suffit de montrer que pour $\beta$-presque tout $b \in B$, on a $\nu_{b}= \nu$. Le résultat découle alors de la relation d'équivariance sur les $(\nu_{b})_{b \in B}$.

Soit $x \in \T^d$ tel que $\nu(x\times \R) \in ]0, +\infty]$, et $t \in \R$ tel que $(x,t) \in \text{supp}\, \nu_{|x \times \R}$. Si $I \subseteq \R$ est un intervalle ouvert contenant $t$, on a $\nu(x \times I)>0$, donc la conservativité de $\nu$ assure que presque toute $\mu$-trajectoire issue de $\{x\}\times I$ revient infiniment souvent dans $\{x\}\times I$. En considérant une base dénombrable de voisinages de $t$, on obtient que pour $\beta$-presque tout $b \in B$, il existe une extraction $\sigma : \N \rightarrow \N$ telle que pour tout $n \geq 0$, on a $b_{1}\dots b_{\sigma(n)} .x=x$, et $\chi(b_{1}\dots b_{\sigma(n)}) \rightarrow 0$. 

On en déduit que pour $\beta$-presque tout $b \in B$, 
$$\nu_{b}= \lim (b_{1}\dots b_{\sigma(n)})_{\star}\nu \geq \lim (b_{1}\dots b_{\sigma(n)})_{\star}(\nu_{|x \times \R}) = \nu_{|x \times \R}$$
puis $\nu_{b |x \times \R} \geq \nu_{|x \times \R}$. Comme $\nu = \int_{B}  \nu_{b} \,\,d\beta(b)$, on a donc finalement l'égalité : 
$$\forall^\beta b \in B, \,\nu_{b|x \times \R}= \nu_{|x \times \R}$$
Par ailleurs, l'ergodicité de $\nu$ implique que $\nu$ et les $\nu_{b}$ sont concentrées sur $\{y \in \Gamma.x, \nu(y\times \R)>0\} \times \R$. Le résulat du paragraphe précédent étant valable pour chacun de ces points $y$, et l'orbite $\Gamma.x \subseteq \T^d$ étant dénombrable, on en déduit que :
$$\forall^\beta b \in B, \,\nu_{b}= \nu$$
ce qui conclut.
\end{proof}
 
Déduisons du lemme précédent que $x$ est rationnel. On note $[\Gamma, \Gamma] \subseteq SL_{d}(\Z)$ le ``semi-groupe dérivé'', i.e. le semi-groupe engendré par les commutateurs $[g,h]$ pour $g, h \in \Gamma$. La $\Gamma$-invariance de $\nu$ implique que $\nu$ est invariante pour l'action de $[\Gamma, \Gamma]$ sur la coordonnée $\T^d$. Soit $I \subseteq \R$ un intervalle borné tel que $\nu(x \times I)>0$. Alors $\nu([\Gamma, \Gamma].x \times I) = \sharp ([\Gamma, \Gamma].x)\,\, \nu(x \times I)$. Le caractère Radon de $\nu$ implique donc que l'orbite de $x$ sous  $[\Gamma, \Gamma]$ est finie. Or le semi-groupe dérivé est irréductible sur $\R^d$ (car l'adhérence de Zariski $G= \widebar{\Gamma}^Z$ de $\Gamma$ dans $SL_{d}(\R)$ est semi-simple, fortement irréductible, et $\widebar{[\Gamma, \Gamma]}^Z$ contient sa composante neutre $G_{c}$). Cela entraine que $x$ est rationnel.

Notons $\omega := \Gamma.x$ la $\Gamma$-orbite de $x$ dans $\T^d$ (finie d'après le paragraphe précédent), et introduisons la mesure $\nu':=\sum_{y\in \omega} \delta_{y}\otimes \leb \in \mathcal{M}(\T^d\times \R)$. C'est une mesure de Radon sur $\T^d \times \R$ qui est  $\mu$-stationnaire (car $\Gamma$-invariante). Le \cref{premcor} affirme qu'elle est donc conservative. Pour montrer que $\nu$ et $\nu'$ sont proportionnelles, il suffit donc de  vérifier que $\nu << \nu'$ et que $\nu'$ est $\mu$-ergodique (cf. \cref{nucons}).

\bigskip

Vérifions que $\nu << \nu'$ : L'ergodicité de $\nu$ implique que $\nu$ est concentrée sur $\omega \times \R$. Il suffit donc de voir que $\nu_{|y \times \R} << \delta_{y}\otimes \leb$ pour tout $y \in \omega$, et cela est fourni par le deuxième point du \cref{premcor}.

Vérifions que $\nu'$ est $\mu$-ergodique :  Par la \cref{equivergo} et la $\Gamma$-invariance de $\nu'$, on est ramené à montrer l'énoncé suivant : Soit $A \subseteq \T^d \times \R$  une partie mesurable telle que $\nu'(A)>0$ et $g.A = A$ pour tout $g \in \Gamma$,  alors $\nu'(\T^d \times \R - A)=0$. Soit $y \in \omega$. La $\Gamma$-invariance de $\nu'$ entraine que $\nu'(A\cap (\{y\}\times \R))>0$. Introduisons $\widetilde{\Gamma} \subseteq \text{Diff}(\T^d \times \R)$ le groupe des difféomorphismes de $\T^d \times \R$ engendré par l'action de $\Gamma$, et $\widetilde{\Gamma}_{y}$ le stabilisateur de $\{y\} \times \R$ dans $\widetilde{\Gamma}$. Le sous groupe $\widetilde{\Gamma}_{y}$ agit sur $\{y\}\times \R$ par translation de la coordonnée réelle et en préservant $A \cap (\{y\}\times \R)$. Comme $\widetilde{\Gamma}_{y}$ est d'indice fini dans $\widetilde{\Gamma}$, l'ensemble des translations possibles est dense dans $\R$.  On en déduit que $A\cap (\{y\}\times \R) = \{y\}\times \R$ $\delta_{y}\otimes \leb$ presque-partout. Comme cela vaut pour tout $y \in \omega$, on conclut que $A= \T^d \times \R$ $\nu'$-pp.

\bigskip

On a finalement montré que l'atome $x$ est rationnel et que les mesures $\nu$ et $\nu'=\sum_{y\in \Gamma.x} \delta_{y}\otimes \leb$ sont proportionnelles. La preuve est donc terminée.

%%%%%%%%%%%%%%%%%%%%%%%%%%%%%%%%%%%%%%%%%%%%%%%%%%%%%%%%%%%%%%%%%%%%%%%%%%%%%%%%%%%%%%%%%%%%%%%%%%%%%%%%%%%%%%%%%%%%%%%%%%%%%%%%%%%%%%%%%%%%%%%%%%%%%%%%%%%%%%%%%%%%%%%%%%%%%%%%%%%%%%%%%%%%%%%%%%%%%%%%%%%%%%%%%%%%%%%%%%%%%%%%%%%%%%%%%%%%%%%%%%%%%%%%%%%%%%%%%%%%%%%%%%%%%%%%%%%%%%%%%%%%%%%%%%%%%%%%%%%%%%%%%%%
%%%%%%%%%%%%%%%%%%%%%%%%%%%%%%%%%%%%%%%%%%%%%%%%%%%%%%%%%%%%%%%%%%%%%%%%%%%%%%%%%%%%%%%%%%%%%%%%%%%%%%%%%%%%%%%%%%%%%%%%%%%%%%%%%%%%%%%%%%%%%%%%%%%%%%%%%%%%%%%%%%%%%%%%%%%%%%%%%%%%%%%%%%%%%%%%%%%%%%%%%%%%%%%%%%%%%%%%%%%%%%%%%%%%%%%%%%%%%%%%%%%%%%%%%%%%%%%%%%%%%%%%%%%%%%%%%%%%%%%%%%%%%%%%%%%%%%%%%%%%%%%%%%%
\newpage

\section{Réduction du cas non atomique}

D'après la conclusion de la section $3$, nous avons réduit le problème au cas où la mesure projetée $\nu(.\times \R) \in \mathcal{M}(\T^d)$ est sans atome. Le reste du texte est consacré à la preuve du résultat suivant :

\begin{th.}\label{THnonat}

Soit  $\mu \in \mathcal{P}(SL_{d}(\Z))$ une probabilité à support fini engendrant un sous semi-groupe $\Gamma \subseteq SL_{d}(\Z)$ fortement irréductible, et soit $\chi : \Gamma \rightarrow \R$ un morphisme de semi-groupes tel que la probabilité $\chi_{\star}\mu \in \mathcal{P}(\R)$ est centrée. Alors toute mesure de Radon $\mu$-stationnaire ergodique $\nu$ sur $\T^d \times \R$ dont la projection $\nu(. \times \R) \in \mathcal{M}(\T^d)$ est sans atome est (à translation près) une mesure de Haar sur $\T^d \times \widebar{\chi(\Gamma)}$.

\end{th.}

Toute mesure $\mu$-stationnaire (ergodique) sur $\T^d\times \R$ étant $\frac{1}{2}(\mu +\delta_{e})$-stationnaire (ergodique), et le cas où $\chi\equiv 0$ ayant été traité dans \cite{BQII}, on pourra supposer:
$$\mu(e)>0 \,\,\, \,\,\,et\,\,\,\,\,\,\chi \nequiv 0$$

\bigskip

La méthode que nous allons employer s'inspire de celle de Benoist-Quint (\cite{BQII}). L'approche consiste à prouver que les mesures limites $\nu_{b} \in \mathcal{M}^{Rad}(\T^d\times \R)$ sont invariantes par translations de la coordonnée  $\T^d$ suivant des directions denses dans le tore. On commence par des rappels sur les groupes algébriques réels semi-simples. Remarquons que l'adhérence de Zariski $G := \widebar{\Gamma}^Z$ de $\Gamma$ dans $SL_{d}(\R)$ est bien semi-simple (\cite{BQI}, lemme 8.5).

\subsection{Rappels sur les groupes algébriques} \label{REDrap}

\bigskip

Cette sous-section rappelle les théorèmes de structure des groupes algébriques réels semi-simples, les notions de projection de Cartan et de cocycle d'Iwasawa, leur interprétation en terme de représentations, une description de la variété drapeau à partir de représentations proximales, et enfin les propriétés d'une marche aléatoire sur la variété drapeau. On conseille pour une première lecture de supposer le groupe $G$ en question Zariski-connexe, ce qui allège les notations. Pour plus de détails, on peut consulter la référence \cite{BQRW}, notamment les sections $6.7-6.10$, $8.1-8.5$. Les groupes algébriques considérés seront tous \emph{linéaires}. 

\bigskip

Soit $G$ un groupe algébrique réel semi-simple, $K$ un sous groupe compact maximal de $G$, $\mathfrak{g}$ et $\mathfrak{k}$ leurs algèbres de Lie.  On note $\mathfrak{s}:= \mathfrak{k}^\perp$  l'orthogonal de $\mathfrak{k}$ pour la forme de Killing sur $\mathfrak{g}$. Comme elle est non-dégénérée sur $\mathfrak{k}$, on a $\mathfrak{k} \oplus \mathfrak{s}= \mathfrak{g}$. Soit $\mathfrak{a} \subseteq \mathfrak{s}$ une sous-algèbre de Lie abélienne maximale. C'est un sous-espace de Cartan i.e. une sous algèbre de $\mathfrak{g}$ abélienne ad-diagonalisable sur $\mathfrak{g}$ et maximale pour ces propriétés (mais elle n'est pas forcément abélienne maximale dans $\mathfrak{g}$). On note $\mathfrak{g} := \mathfrak{z} \oplus (\oplus_{\alpha \in \Sigma} \mathfrak{g}_{\alpha})$ la décomposition de $\mathfrak{g}$ en espaces de racines associée, où $\Sigma \subseteq \mathfrak{a}^\star-\{0\}$ désigne l'ensemble des racines de $\mathfrak{a}$ dans $\mathfrak{g}$, et $\mathfrak{g}_{\alpha}:= \{y \in \mathfrak{g}, \,\,\forall x \in \mathfrak{a}, \,\ad x(y)=\alpha(x)y \}$, $\mathfrak{z}:=  \{y \in \mathfrak{g},\,\, \forall x \in \mathfrak{a},\, \ad x(y)=0 \}$. On choisit ensuite $\Pi \subseteq \Sigma$ une base de $\Sigma$ et on note $\Sigma^+ \subseteq \Sigma$ l'ensemble des racines positives de $\Sigma$ (i.e. s'écrivant comme somme d'éléments de $\Pi$). Posons $\mathfrak{u}:= \oplus_{\alpha \in \Sigma^+}  \mathfrak{g}_{\alpha}$ sous-algèbre nilpotente de $\mathfrak{g}$, $ \mathfrak{a}^+ := \{x \in  \mathfrak{a}, \,\forall \alpha \in \Sigma^+, \alpha(x)\geq0\}$.   

On revient maintenant au groupe $G$. On note $G_{c}$ la composante Zariski-connexe de $G$ contenant l'élément neutre, $K_{c}= K \cap G_{c}$ sous groupe compact maximal de $G_{c}$. On appelle $A$ et $U  \subseteq G_{c}$ les sous groupes de $G$ Zariski-fermés Zariski-connexes d'algèbres de Lie respectives $ \mathfrak{a}$ et $ \mathfrak{u}$ (ils existent bien). On note $P:=N_{G}( \mathfrak{z}\oplus  \mathfrak{u})$ le sous-groupe parabolique minimal associé à $ \Sigma^+$. On peut écrire $P=M\exp(\ag)U$ où $M:=P \cap K$  est un sous groupe qui rencontre chaque composante connexe de $G$, qui normalise $A$ (donc $\exp(\ag)$), et tel que $M_{c}:= M \cap G_{c}$ centralise $A$. 

Introduisons maintenant quelques décompositions utiles de $G$. 

\begin{lemme}[Décomposition de Cartan] \label{rap1}
$$G= K \exp( \mathfrak{a}^+) K_{c}$$
Plus précisément, pour tout $g \in G$, il existe $k \in K, l_{c}\in K_{c}$ et un unique élément $x \in  \mathfrak{a}^+$ tel que $g=k \exp(x) l_{c}$. On note $ \kappa(g) =x$, appelé projection de Cartan de $g$.
\end{lemme}

%\begin{rem.}
%On pourrait aussi considérer la décomposition de Cartan $G= K_{c} \exp( \mathfrak{a}^+) K$. Ces deux décompositions fournissent la même notion de projection de Cartan. En effet, soit $g\in G$ décomposé en $g=k \exp(x) l_{c}$ avec $k \in K, x \in \ag^+, l_{c} \in K_{c}$. On écrit $k = k_{c} m$ où $m \in M=P\cap K$. Comme $m$ commute avec $A$, on a $g= k \exp(x) m l_{c}$ donc $x$ est aussi la projection de Cartan associée à la décomposition $G= K_{c} \exp( \mathfrak{a}^+) K$.  
%\end{rem.}

\begin{lemme}[Décomposition d'Iwasawa] \label{rap2}
$$G= K \exp( \mathfrak{a}) U$$ 
Plus précisément, l'application $K \times \exp( \mathfrak{a}) \times U \rightarrow G,\, (k, a, u) \mapsto kau$ est un homémorphisme.  
\end{lemme}

On définit la \emph{variété drapeau} $\mathscr{P}:=G/P_{c}$ où $P_{c}:= P\cap G_{c}$. Elle jouera un rôle clef dans toute la suite du texte. Le groupe $G$ agit sur  $\mathscr{P}$ par translation à gauche et l'action du sous groupe $K$ est transitive. Notons $\xi_{0}:=P_{c}/P_{c}$ le drapeau standard. On définit sur $\PP$ un cocycle continu $\sigma : G \times \PP \rightarrow \mathfrak{a}$  appelé \emph{cocycle d'Iwasawa} : soit $g \in G$, on note $\sigma(g, \xi_{0})$ l'unique élément de $\mathfrak{a}$ tel que $g\in K \exp(\sigma(g, \xi_{0})) U$; si $\xi \in \PP$ est quelconque, on écrit $\xi = kP_{c}$ pour un $k \in K$ et on pose $\sigma(g, \xi)= \sigma(gk, \xi_{0})$. 

\bigskip
Le groupe des composantes connexes $F:= G/G_{c}$ a plusieurs actions naturelles. Il agit à droite sur la variété drapeau $\PP:= G/P_{c}$ à travers l'identification $F\equiv P/P_{c}$. Il agit aussi  (à gauche) sur l'espace de Cartan $\ag$ de la fa\c con suivante : si $f \in F$ est représenté par $m \in M$, alors l'action de $f$ sur $\ag$ est donnée par la représentation adjointe $\text{Ad} m$ restreinte à $\ag$. Cette action est bien définie d'après les propriétés de $M$ et $M_{c}$ énoncées plus haut. Par ailleurs, le cocycle d'Iwasawa $\sigma : G \times \PP \rightarrow \ag$ est équivariant à droite pour ces actions au sens où pour $g \in G, \xi \in \PP, f \in F$, on a $\sigma(g, \xi f)= f^{-1}\sigma(g,\xi)$. 

\bigskip

Rappelons quelques éléments de théorie du plus haut poids pour une représentation de $G_{c}$ . Soit $V$ un $\R$-espace vectoriel de dimension finie, $\rho : G_{c} \rightarrow SL(V)$ une représentation algébrique irréductible de $G_{c}$. Alors l'action de l'algèbre de Lie $\mathfrak{a}$ induite sur $V$ est diagonalisable. On note $V = \oplus_{i \in I} V_{i}$ une décomposition de $V$ en sous-espaces propres. Chaque $V_{i}$ correspond à une forme linéaire $\alpha_{i} \in \mathfrak{a}^\star$. Cet ensemble de formes linéaires admet un (unique) plus grand élément $\omega  \in \mathfrak{a}^\star$  au sens où pour tout $i \in I$, $\omega - \alpha_{i}$ est somme de racines dans la base $\Pi$. On appelle $\omega$ le \emph{plus haut poids} de la représentation $\rho$. Le sous-espace propre associé est $V^U:=\{v \in V, \,\, \rho(U)v=v\}$. Sa dimension est appelée \emph{dimension proximale} de $\rho(G_{c})$. On dit que la représentation $\rho$ est \emph{proximale} si sa dimension proximale est égale à $1$.

Une représentation $ (V, \rho)$ de $G_{c}$ peut se ``prolonger'' en une représentation de $G$ appelée \emph{représentation induite}. Elle est construite ainsi : on note 
$$W := \{\varphi : G \rightarrow V, \,\,\varphi(\,. \;g_{c})= \rho(g_{c})^{-1}\varphi \}$$
$\R$-espace vectoriel de dimension $|G/G_{c}| .\dim V $, et pour $g \in G$, $\varphi \in W$, on pose $\rho_{W}(g) \varphi := \varphi(g^{-1}. )$. On peut décomposer $W$ en copies de $V$ indexées par l'ensemble des composantes connexes $F:= G/G_{c}$ de $G$. Pour cela on pose pour $f \in F$, $V^f:= \{\varphi \in W, \,\varphi=0 \text{ en dehors de $f$}\}$ et on remarque que $W = \bigoplus_{f \in F} V^f$. Notons $G \rightarrow F, g \mapsto f_{g}$ le morphisme naturel. On identifie $V^{f_{e}}\equiv V$ via l'isomorphisme $\varphi \mapsto \varphi(e)$. Pour tout $g \in G$,  $\rho_{W}(g) V = V^{f_{g}}$, et si $g \in G_{c}$, alors $\rho_{W}(g) : V \rightarrow V$ coïncide avec $\rho(g)$.

Le lemme suivant permet de voir toute représentation de $G$ à travers une représentation induite par $G_{c}$.
\begin{lemme} \label{redind}
Soit $(V,\rho)$ une représentation de $G$, $(W, \rho_{W})= \emph{Ind}^G_{G_{c}}(V, \rho)$ la représentation induite par $\rho_{|G_{c}}$. Il existe une (unique) application linéaire $G$-équivariante $\Psi : W \rightarrow V$  telle que $\Psi_{|V}= Id_{V}$. 
\end{lemme}

Nous allons maintenant voir comment la projection de Cartan et le cocycle d'Iwasawa s'interprètent en termes de représentations. Donnons nous $(V, \rho)$ une représentation algébrique irréductible de $G_{c}$, on note $\omega \in \ag^\star$ son plus haut poids, $r \in \N$ la dimension proximale de $\rho(G_{c})$  et $(W,\rho_{W})= \text{Ind}^G_{G_{c}}(V, \rho)$ la représentation induite à $G$. On munit $W$ d'un \emph{``bon'' produit scalaire} $\langle..\rangle$ où ``bon'' signifie que $\rho_{W}(K)$ préserve  $\langle..\rangle$ (noté plus tard $\rho_{W}(K)\subseteq O(W)$), que $\rho_{W}(A)$ est constitué d'éléments auto-adjoints (noté plus tard $\rho_{W}(A)\subseteq \text{Sym}(W)$), et que les sous-espaces $(V^f)_{f \in F}$ sont orthogonaux entre eux. On introduit aussi une application $G$-équivariante de $\PP$ vers la variété Grassmanienne $\amalg_{f \in F} \mathbb{G}_{r}(V^f)$ donnée par :
$$\PP \rightarrow \amalg_{f \in F} \mathbb{G}_{r}(V^f), \xi= g\xi_{0} \mapsto V_{\xi} :=\rho_{W}(g)V^U$$
On a alors le dictionnaire suivant :

\begin{lemme} \label{rep}
Soit $g \in G$, $\xi \in \PP$
\begin{itemize}
\item $\omega(\kappa(g)) = \log (||\rho_{W}(g)_{|V}||)$
\item $\omega(\sigma(g, \xi)) = \log \frac{||\rho_{W}(g)w ||}{||w||}$ pour $w \in V_{\xi}$
\end{itemize}
\end{lemme}

La combinaison du \cref{redind} et  du \cref{rep} donnent une interprétation de la projection de Cartan et du cocycle d'Iwasawa en termes de représentations algébriques fortement irréductibles de $G$. Plus explicitement, soit $(V, \rho)$ une telle représentation, $\omega \in \ag^\star$ et $r\geq 0$ comme précédemment. On peut munir $V$ d'un produit scalaire $\rho(K)$-invariant tel que $\rho(A)\subseteq \text{Sym}(V)$ (cf.  \cite{BQRW}, lemma $8.9$).  Pour $\xi=g\xi_{0} \in \PP$, on pose $V_{\xi}:= gV^U$. Alors l'application $\PP\rightarrow \mathbb{G}_{r}(V), \,\xi \mapsto V_{\xi}$ et le plus haut poids $\omega \in \ag^\star$ sont $F$-invariants et   
$\omega(\kappa(g)) = \log (||\rho(g)||)$ et  $\omega(\sigma(g, \xi)) = \log \frac{||\rho(g)v ||}{||v||}$ pour $v \in V_{\xi}$.

\bigskip

Les paragraphes précédents montrent comment les objets algébriques tels que $\PP$, $\ag$, $\kappa$, $\sigma$ fournissent des informations intrinsèques qui décrivent simultanément toutes les représentations fortement irréductibles de $G$. Réciproquement, ces objets peuvent être caractérisés par un nombre fini de réprésentations irréductibles proximales de $G_{c}$. 

\begin{lemme}
Il existe une famille finie $(\rho_{\alpha}, V_{\alpha})_{\alpha \in \Pi}$ de représentations algébriques  proximales irréductibles de $G_{c}$ telle que :

\begin{itemize}
\item l'application produit
$$\PP \rightarrow \amalg_{f \in F} \Pi_{\alpha \in \Pi} \mathbb{P}(V^f_{\alpha}), \,\,\xi \mapsto (V_{\alpha, \xi})_{\alpha \in \Pi} $$
est un plongement $G$-équivariant. 
\item les plus haut poids $(\omega_{\alpha})_{\alpha \in \Pi}$ de ces représentations forment une base de $\ag^\star$
\end{itemize}
\end{lemme}

Dans la suite on fixe une telle famille $(\rho_{\alpha}, V_{\alpha})_{\alpha \in \Pi}$. On munit chaque représentation induite $W_{\alpha}$ d'un ``bon'' produit scalaire. On a alors une distance $d^f_{\alpha}$ sur $\Pbb(V^f_{\alpha})$ donnée par $d^f_{\alpha}(\R v, \R w):= \frac{||v \wedge w||}{||v|| ||w||}$ pour tout $v, w \in V^f_{\alpha}-\{0\}$. Etant donné $f \in F$, on munit $\Pi_{\alpha \in \Pi} \mathbb{P}(V^f_{\alpha})$ de la distance produit $d((\R v_{\alpha}), (\R w_{\alpha})) := \sum_{\alpha \in \Pi} d^f_{\alpha}(\R v_{\alpha}, \R w_{\alpha})$. Cela induit une distance sur  $\PP$ (telle que deux points appartenant à des composantes connexes différentes sont par convention à distance infinie).

\bigskip

Terminons les rappels par la dynamique des marches aléatoires sur la variété drapeau (détails dans \cite{BQRW}, $10.1$). Soit  $\mu \in \mathcal{P}(G)$  une probabilité sur $G$ dont le support engendre un sous semi-groupe Zariski-dense de $G$. Notons $\mu_{G_{c}} \in \mathcal{P}(G_{c})$ la loi du premier retour à $G_{c}$ pour la marche à droite sur $G$. Plus formellement, on pose $B= G^{\N^\star}$, $\beta := \mu^{\otimes \N^\star}$, $\tau_{G_{c}} : B\rightarrow \N\cup \{\infty\}, b \mapsto \inf\{n \geq 1, \, b_{1}\dots b_{n}\in G_{c}\}$ le temps de retour à $G_{c}$ (fini $\beta$-ps), et on définit $\mu_{G_{c}}$ comme étant la loi de $ b_{1}\dots b_{\tau_{G_{c}}(b)}$ quand $b$ varie suivant $\beta$.   

La composante connexe $\PP_{c}:= G_{c}/P_{c}$ de la variété drapeau admet une unique probabilité $\mu_{G_{c}}$-stationnaire $\nu_{\PP_{c}} \in \mathcal{P}(\PP_{c})$. Notons $\beta_{G_{c}}:= \mu_{G_{c}}^{\otimes \N^\star} \in \mathcal{P}(B)$. La probabilité $\nu_{\PP_{c}}$ est proximale au sens où ses mesures limites $(\nu_{\PP_{c},b})_{b \in B}$ sont $\beta_{G_{c}}$-presque sûrement des masses de Dirac. Comme l'application de retour $s :  (B, \beta) \rightarrow (B, \beta_{G_{c}}), b \mapsto (b_{1}\dots b_{\tau_{G_{c}}(b)},  b_{\tau_{G_{c}}(b)+1}\dots b_{\tau_{G_{c}}(T^{\tau_{G_{c}}(b)}b)}, \dots) $ respecte les mesures, on dispose d'une famille mesurable de drapeaux $(\xi_{b})_{b \in B} \in \PP_{c}^B$ $\beta$-ps bien définie  telle que pour $\beta$-presque tout $b \in B$,  $\nu_{\PP_{c},s(b)}= \delta_{\xi_{b}}$.

La variété drapeau $\PP$ admet une unique probabilité $\mu$-stationnaire que l'on désignera par $\nu_{\PP} \in \mathcal{P}(\PP)$. Via l'identification $\PP\equiv \PP_{c}\times F$ donnée par l'action de $F$ sur $\PP$, cette probabilité s'écrit $\nu_{\PP}=  \nu_{\PP_{c}} \otimes df$ où $df$ est la probabilité uniforme sur $F$. En particulier, $\nu_{\PP}$ est $F$-invariante à droite,  et pour $\beta$-presque tout $b \in B$, la mesure limite $\nu_{\PP, b}= \frac{1}{|F|} \sum_{f \in F} \delta_{\xi_{b},f}$ où $\xi_{b,f}:= \xi_{b}.f$. On dit que $\nu_{\PP}$ est $\mu$-proximale au dessus de $F$. 

On note $\sigma_{\mu} := \int_{G \times \PP}\sigma(g,\xi) \,\,d\mu(g)d\nu_{\PP}(\xi) \in \ag$ le \emph{vecteur de Lyapunov} associé. Il est $F$-invariant. Etant donné représentation algébrique fortement irréductible $\rho : G \rightarrow GL(V)$, de plus haut poids  $\omega \in \ag^\star$ (par rapport à $G_{c}$), on note $\lambda_{\mu}:= \omega(\sigma_{\mu}) \in \R$ le \emph{premier exposant de Lyapunov} de la $\mu$-marche sur $V$. Si l'image $\rho(G) \subseteq GL(V)$ n'est pas bornée, alors on a $\lambda_{\mu}>0$ (\cite{BQRW}, $4.7$)

%\begin{rem.}
%Si le groupe $G$ n'est pas compact, alors les représentation $\rho_{\alpha}(G)$ ne le sont pas non plus. En effet, s'il existe $\alpha \in \Pi$ tel que $\rho_{\alpha}(G) \subseteq GL(V_{\alpha})$ est compact, la proximalité de la représentation impose $\dim V_{\alpha}=1$, puis comme $G= [G,G]$, on a $\rho_{\alpha}(G)=\{Id\}$, donc $\omega_{\alpha} =0$, contredisant le deuxième point du lemme.

%\end{rem.}

\bigskip

\subsection{Réduction au théorème de dérive exponentielle} \label{REDred}

On réduit la preuve du \cref{THnonat} à un théorème de dérive exponentielle sur un système dynamique fibré. On explique aussi l'intuition derrière une telle démarche.
\bigskip

Revenons au cadre du \cref{THnonat}. Le groupe $G:= \widebar{\Gamma}^Z \subseteq SL_{d}(\R)$ est fortement irréductible et contient $\Gamma\subseteq SL_{d}(\Z)$ comme semi-groupe Zariski-dense. Il est donc  semi-simple (cf \cite{BQI}, lemme $8.5$). On peut ainsi appliquer la sous-section précédente à $G$ et se donner un sous groupe compact maximal $K$, un sous espace de Cartan $\mathfrak{a}$, un système de racines $\Sigma \subseteq \mathfrak{a}^\star$, etc. fixés une fois pour toute dans la suite. Par ailleurs, $G$ est inclus dans $SL_{d}(\R)$ donc admet une représentation canonique sur $V=\R^d$ (donnée par l'identité). Le plus haut poids de $G_{c}$ sera noté $\omega \in \ag^\star$, sa dimension proximale sera notée $r \in \llbracket 1, d-1 \rrbracket$ ($r< d$ car $G$ n'est pas borné).  On munit $V$ d'un produit scalaire $\langle.,. \rangle$ tel que $K \subseteq O(V)$, $A \subseteq \text{Sym}(V)$.

On reprend les notations $\PP, \PP_{c}, (\xi_{b})_{b \in B} \in \PP_{c}^B$ ainsi que l'application ($F$-invariante) $\PP \rightarrow  \mathbb{G}_{r}(V), \xi= g\xi_{0} \mapsto V_{\xi} :=gV^U$ de la sous-section précédente et on pose $V_{b}:= V_{\xi_{b}}$, appelé \emph{$r$-plan limite} pour la trajectoire $b \in B$. Il est prouvé dans $\cite{BQRW}$ (section $10.2$) que pour $\beta$-presque tout $b \in B$, toute valeur d'adhérence de la suite de matrices $(\frac{b_{1}\dots b_{n}}{||b_{1}\dots b_{n}||})_{n \geq 1} \in M_{d}(\R)^{\N^\star}$ a pour image $V_{b}$.

\bigskip

L'idée fondamentale motivant la preuve du \cref{THnonat}  est de montrer que les mesures limites $\nu_{b}$ de $\nu \in \mathcal{M}^{Rad}(\T^d \times \R)$ sont invariantes par translation de la coordonnée  $\T^d$ suivant les $r$-plans limites $V_{b}$. Un tel résultat permettrait de conclure. En effet, la forte irréductibilité de $G$ assure que les projections de ces $r$-plans dans $\T^d$ sont presque sûrement denses, et ainsi que les mesures $\nu_{b}$, puis $\nu$, sont invariantes par translation de la coordonnée $\T^d$. On termine alors la preuve en appliquant le deuxième point \cref{premcor}.

On cherche donc à prouver l'invariance des $\nu_{b}$ selon les $r$-plans  $V_{b}$. On note  $X := \T^d \times \R$ et on introduit le système dynamique fibré $B^X := B \times X$, $T^X : B^X \rightarrow B^X, (b,x) \mapsto (Tb, b_{1}^{-1}x)$ muni de la mesure de Radon invariante $\beta^X:= \int_{B} \delta_{b} \otimes \nu_{b} \,d\beta(b)$. Donnons nous de plus une famille mesurable $(e_{b})_{b \in B} \in (\R^d)^r$ telle que pour $\beta$-presque tout $b \in B$, $e_{b}= (e_{b,1}, \dots, e_{b, r})$ est une base orthonormée de $V_{b}$.  On peut définir sur $B^ X$ un flot à $r$-paramètres $(\phi_{t})_{t \in \R^r}$ via la formule $\phi_{t} (b,x) := (b, x+\sum_{i=1}^r t_{i}e_{b,i})$, la translation concernant la coordonnée sur $\T^d$. Montrer l'invariance des $\nu_{b}$ suivant les $r$-plans $V_{b}$ revient alors à montrer que la mesure $\beta^X$ est $\phi_{t}$ invariante. Une première difficulté est que le flot $\phi_{t}$ ne commute pas avec la transformation $T^X$. On rajoute alors des paramètres au système pour contourner cette obstruction.

Notons $\sigma : G \times \PP \rightarrow \mathfrak{a}$ le cocycle d'Iwasawa, $\ag^F := \{x \in \ag, \,F.x=x\}$ le sous-espace $F$-invariant de $\ag$, $\sigma_{F}:= \frac{1}{|F|} \Sigma_{f \in F} f.\sigma$ le cocycle projeté de $\sigma$ sur $\ag^F$,  et  $\theta : B \rightarrow \mathfrak{a}^F, b \mapsto \sigma_{F}(b_{1}, \xi_{Tb})$ (où $T$ désigne le shift sur $B = \Gamma^{\N^\star})$. Pour $\beta$-presque tout $b \in B$, on a la relation d'équivariance $b_{1}V_{Tb} = V_{b}$, et quelque soit $v \in V_{Tb}$, l'égalité $||b_{1}v|| = e^{\omega(\theta(b))} ||v||$ (cf. \cref{rep}). Il existe donc une matrice orthogonale $O(b)\in O_{r}(\R)$ tel que
 \begin{align*}
 b_{1} e_{Tb}= e^{\omega(\theta(b))} O(b)e_{b} \tag{$\ast$}
\end{align*}
On définit un système dynamique fibré au dessus du shift en posant :
 \begin{itemize}
 
 \item $B^+ = B \times \ag^F \times O_{r}(\R) \times X$
 \item $T^+ : B^+ \rightarrow B^+, (b,z,O, x) \mapsto (Tb, z- \theta(b), OO(b)^{-1} , b^{-1}_{1}x)$.
 \item $\beta^+ := \int_{B} \delta_{(b,z,O)} \otimes \nu_{b} \,\, d\beta(b)d\text{leb}_{\ag^F}(z) dh(O)\,\,\,\,$ 
\end{itemize}
où $\text{leb}_{\ag^F}$  et $h$ sont des mesures de Haar  sur $\ag$ et $O_{r}(\R)$, fixées une fois pour toute. 

\bigskip

 On  munit $B^+$ d'un flot $(\phi_{t})_{t \in \R^r}$ respectant les fibres défini par  $\phi_{t} : B^+ \rightarrow B^+, (b,z, O, x) \mapsto (b,z,O, x + \langle t,  e^{\omega(z)} O e_{b}\rangle')$, où $\langle.,.\rangle' : \R^r \times (\R^d)^r \rightarrow \R^d, (t,v) \mapsto \sum_{i=1}^rt_{i}v_{i}$ et la translation concerne la coordonnée sur $\T^d$ de $X= \T^d \times \R$.

\bigskip

\begin{lemme}\label{ref5}
\begin{itemize}
\item $T^+$ préserve la mesure $\beta^+$.
\item Pour $\beta$-presque $b \in B$, pour tout $(z,O,x,t) \in \ag^F\times O_{r}(\R) \times X\times \R$, $$T^+\circ \phi_{t}(b,z,O,x)= \phi_{t}\circ T^+ (b,z,O,x)$$

\end{itemize}
\end{lemme}

\bigskip

\begin{proof}
Vérifions le premier point. Soit $\varphi : B^+\rightarrow[0,+\infty]$ mesurable. 
\begin{align*}
T^+_{\star}\beta^+(\varphi)&=\int_{B^+}\varphi(Tb, z- \theta(b), OO(b)^{-1} , b^{-1}_{1}x)  \,\,d\nu_{b}(x)d\beta(b)d\text{leb}_{\ag^F}(z) dh(O)\\
&=\int_{B^+}\varphi(Tb, z, O , b^{-1}_{1}x)  \,\,d\nu_{b}(x)d\beta(b)d\text{leb}_{\ag^F}(z) dh(O)\\
&=\int_{B^+}\varphi(b, z, O , x)  \,\,d\nu_{b}(x)d\beta(b)d\text{leb}_{\ag^F}(z) dh(O)\\
&=\beta^+(\varphi)
\end{align*}
en utilisant l'invariance par translation à droite des mesures de Haar, et l'équivariance des $\nu_{b}$. Ainsi, $T^+_{\star}\beta^+=\beta^+$.

Vérifions le second point. On se donne $b \in B$ tel que   $b_{1} e_{Tb}= e^{\omega(\theta(b))} O(b)e_{b}$ (vérifié $\beta$-ps.).  Soit $(z,O,x,t) \in \ag^F\times O_{r}(\R) \times X\times \R$. Alors 
\begin{align*}
&T^+\circ \phi_{t}(b,z,O,x)= (Tb, z-\theta(b), OO(b)^{-1},b^{-1}_{1}(x + \langle t,  e^{\omega(z)} O e_{b}\rangle'))\\
& \phi_{t}\circ T^+(b,z,O,x) =(Tb, z-\theta(b), OO(b)^{-1},b^{-1}_{1}x + \langle t,  e^{\omega(z-\theta(b))} OO(b)^{-1} e_{Tb}\rangle')
\end{align*}
Il suffit donc de vérifier que $ b^{-1}_{1}e^{\omega(z)} O e_{b}=e^{\omega(z-\theta(b))} OO(b)^{-1} e_{Tb}$ ce qui est une réécriture de $(\ast)$ ci dessus. 
\end{proof}

\bigskip

L'invariance des $\nu_{b}$ suivant les $r$-plans $V_{b}$ revient à montrer que $\beta^+$ est invariante sous le flot $(\phi_{t})_{t \in \R^r}$. L'idée est alors de désintégrer $\beta^+$ le long des orbites de $\phi_{t}$. Dans  le cas où l'espace des orbites muni de la tribu quotient est un espace borélien standard, cette désintégration est une désintégration au sens classique (par rapport à l'application de projection $B^+ \rightarrow \phi_{t}\backslash B^+$) et attribue à chaque $\phi_{t}$-orbite une mesure de Radon. Notons que pour tout $c \in B^+$, on peut alors relever la mesure sur son orbite $\{\phi_{t}(c), t \in \R^r\}$ via l'homéomorphisme local : $\R^r \rightarrow \{\phi_{t}(c), t \in \R^r\}, t \mapsto \phi_{t}(c)$. On obtient finalement une famille de mesures  $\sigma : B^+ \rightarrow \mathcal{M}^{Rad}(\R^r)$ qui vérifie la propriété d'équivariance : $\sigma (\phi_{t}c)= \delta_{t}^\star \sigma(c)$  où $\delta_{t}$ désigne la translation de pas $t $ sur $\R^r$.

\bigskip

Le problème de cette approche est que l'espace des orbites n'est pas standard en général. Notons $\mathcal{M}_{1}^{Rad}(\R^r) = \mathcal{M}^{Rad}(\R^r)/\sim$ l'ensemble des mesures de Radon (non nulles) sur $\R^r$ quotienté par la relation d'équivalence $m \sim m'$ s'il existe $\lambda >0$ tel que $m= \lambda m'$. Un élément de $\mathcal{M}_{1}^{Rad}(\R^r)$ sera appelé $\emph{classe  projective de mesures}$. L'action par translation de $\R^r$ sur $\mathcal{M}^{Rad}(\R^r)$ respecte la relation d'équivalence $\sim$, donc induit une action sur $\mathcal{M}_{1}^{Rad}(\R^r)$. 
La théorie générale donne un sens à une désintégration de $\beta^+$ le long du flot  $(\phi_{t})_{t \in \R^r}$ (voir \cite{BQI} ou \cite{Duf}), et conduit à une application mesurable $\sigma : B^+ \rightarrow \mathcal{M}_{1}^{Rad}(\R^r)$ ainsi qu'à une partie  $E \subseteq B^+$  de mesure pleine telle que pour tout $c \in E$, $t \in \R^r$ tel que $\phi_{t}c  \in E$, on a $\sigma(\phi_{t}c)= \delta_{t}^{\star} \sigma(c)$. On a de plus que 

\begin{lemme}\label{ref6}
 La mesure de Radon $\beta^+$ est invariante par le flot $(\phi_{t})_{t\in \R^r}$ si et seulement si  l'application $\sigma$ est $\beta^+$-pp  égale à la classe projective $\R_{>0}\leb$.
\end{lemme}

\begin{proof}
On réfère au chapitre $2$ de la thèse de L. Dufloux \cite{Duf} pour une introduction détaillée aux mesures conditionnelles le long d'une action de groupe à stabilisateurs discrets. Ce dernier lemme est une application du théorème $2.1.2.14$.
\end{proof}

On est donc ramené à montrer que $\sigma : B^+ \rightarrow \mathcal{M}_{1}^{Rad}(\R^r)$ est $\beta^+$-pp  égale à la classe projective de Lebesgue. Le point clef est de montrer que pour presque tout point $c \in B^+$, la classe projective $\sigma(c) \in \mathcal{M}^{Rad}_{1}(\R^r)$ est invariante par translation suivant une droite de $\R^r$ (attention, il s'agit de l'invariance d'un ensemble  de mesures, mais pas forcément des mesures individuellement). Dans le cas où $r=1$, on peut alors conclure directement (en utilisant (\cite{BQI}, proposition $4.3$) et en remarquant que le flot $\phi_{t}$ stabilise une suite exhaustive de compacts de $B^+$ pour se ramener à une mesure finie). Le cas général nécessitera quelques efforts supplémentaires (voir la dernière section).

Pour $n \geq 0$, on note $\mathcal{Q}^+_{n} = (T^+)^{-n}(\mathcal{B}^+)$ la tribu des $n$-fibres de $T^+$. C'est une famille décroissante de sous tribus de $\mathcal{B}^f$. Leur intersection $\mathcal{Q}^+_{\infty} :=\cap_{n \geq 0} \mathcal{Q}^+_{n}$ est appelée la \emph{tribu queue} du système dynamique $(B^+, T^+)$. La tribu queue d'un système dynamique fournit beaucoup d'informations sur celui ci, par exemple sa trivialité par rapport à une mesure invariante implique que cette mesure est ergodique. 

\begin{lemme}\label{ref7}
Quitte à modifier $\sigma$ sur un ensemble $\beta^+$-négligeable, on peut supposer $\sigma \circ T^+= \sigma$ (en tout point).  L'application  $\sigma$  est alors $\mathcal{Q}^+_{\infty}$-mesurable.
\end{lemme}

\begin{proof}
On commence par montrer que $\sigma \circ T^+= \sigma$ $\beta^+$-pp. Cela provient essentiellement du fait que $T^+$ commute avec le flot $(\phi_{t})_{t\in \R^r}$.

Fixons $(b , z, O) \in B \times \ag^F \times O_{r}(\R)$ un point $\beta\otimes\leb_{\ag^F}\otimes h$-typique. On entend par là que ce qui suit est valable pour presque tout choix de $(b,z,O)$.  On note $E:= \{(b,z,O)\} \times X$, $E' := \{(Tb,z-\theta(b),OO(b)^{-1})\} \times X$ munis naturellement des mesures $\nu_{E}:=\delta_{(b,z,O)}\otimes \nu_{b}$ et $\nu_{E'}:=\delta_{(Tb,z-\theta(b),OO(b)^{-1})} \otimes \nu_{Tb}$. Remarquons que $E$ et $E'$ sont tous deux stables par le flot $\phi_{t}$ et que d'après le \cref{ref5} l'application $T^+_{|E}: E \rightarrow E'$ est une bijection bimesurable qui respecte les mesures et commute à l'action du flot. Soit $\sigma_{E}: E \rightarrow \mathcal{M}^{Rad}_{1}(\R^r)$, $\sigma_{E'}: E' \rightarrow \mathcal{M}^{Rad}_{1}(\R^r)$ les mesures conditionnelles associées aux flots sur $E$ et $E'$. On a alors que $\sigma_{E'} \circ T_{|E}^+= \sigma_{E}$ $\nu_{E}$-pp. Par ailleurs, on a que $\sigma_{E}= \sigma_{|E}$ $\nu_{E}$-pp et $\sigma_{E'}= \sigma_{|E'}$ $\nu_{E'}$-pp. Finalement, pour presque tout $(b, z, O)$ et $\nu_{b}$-presque tout $x \in X$, on a $\sigma \circ T^+(b,z, O, x)= \sigma (b, z, O, x)$ d'où l'égalité $\beta^+$-presque partout.

Pour conclure, on se donne $D \subseteq B^+$ plein tel que $\sigma \circ T^+_{|D}= \sigma_{|D}$. On pose $D'= \cup_{n \geq 0}\cap_{p \geq n} (T^+)^{-p}D$. C'est un sous ensemble plein de $B^+$ tel que $(T^+)^{-1}D'= D'$. On se donne une apllication $\sigma' : B^+ \rightarrow \mathcal{M}_{1}^{Rad}(\R)$ telle que $\sigma'$ est constante sur le complémentaire de $D'$ et $\sigma'(c) = \lim_{n} \sigma \circ (T^+)^n(c)$ pour $c \in D'$ (cette limite est définie car la suite est stationnaire). L'application $\sigma'$ coïncide avec $\sigma$ sur $D$ donc presque partout et est $T^+$-invariante.
\end{proof}

\bigskip

Le théorème de dérive exponentielle déduit du lemme précédent l'invariance de presque toute classe projective $\sigma(c) \in  \mathcal{M}_{1}^{Rad}(\R^r)$ suivant une droite de $\R^r$. Il est du à Benoist-Quint (\cite{BQII}) dans le cas des espaces homogènes de volume fini. Sa démonstration occupera toute la prochaine section. 

\begin{th*}[Dérive exponentielle]
Soit $Y$ un espace mesurable standard, $\sigma : B^+ \rightarrow Y$ une application $\mathcal{Q}^+_{\infty}$-mesurable, $E \subseteq B^+$ une partie $\beta^+$-pleine. 
Alors pour presque tout $c \in E$, pour tout $\varepsilon>0$, il existe $c', c'' \in E$, $t \in ]-\varepsilon, \varepsilon[^r$ non nul,  tels que $c''= \phi_{t}(c')$ et $\sigma(c)= \sigma(c')=\sigma(c'')$. 
\end{th*}

En appliquant la dérive à la mesure conditionnelle $\sigma : B^+ \rightarrow \mathcal{M}_{1}^{Rad}(\R^r)$, on obtient : 
\begin{cor.} \label{corderiv}
Pour $\beta$-presque tout $c \in B^+$, l'ensemble $$\{t \in \R^r, \,\, \text{$\sigma(c) \in \mathcal{M}_{1}^{Rad}(\R^r)$ est $t$-invariante}\}$$ contient une droite.
\end{cor.}

\begin{proof}[Démonstration du corollaire]
D'après le dernier lemme, on a peut appliquer le théorème de dérive à la mesure conditionnelle $\sigma : B^+ \rightarrow \mathcal{M}_{1}^{Rad}(\R)$ et à une partie $E \subseteq B^+$ telle que pour tout $c \in E$, $t \in \R$ tel que $\phi_{t}(c)  \in E$ on a $\sigma(\phi_{t}(c))= \delta_{t}^{\star} \sigma(c)$.  Notons $E' \subseteq E$ l'ensemble des $c \in E$ tels que la conclusion du théorème de dérive est vérifiée.  Pour $c \in E'$, l'ensemble $\{t \in \R^r, \,\, \text{$\sigma(c) \in \mathcal{M}_{1}^{Rad}(\R^r)$ est $t$-invariante}\}$ est un sous-groupe fermé de $\R^r$ qui n'est pas discret. Il contient donc forcément une droite, d'où le corollaire.
\end{proof}

On en déduira le \cref{THnonat}  dans la dernière section.

\bigskip

\subsection{Stratégie pour la dérive exponentielle}

Les prochaines sections sont consacrées à la preuve de la dérive exponentielle énoncée en fin de section $4$. 

La \emph{stratégie de la preuve} est la suivante. \label{structurederive}
On se donne $W \subseteq B^+$ une partie de mesure finie, vue comme une \emph{fenêtre}. On montre le résultat pour presque tout $c \in W \cap E$. On fixe $K \subseteq W \cap E$ une partie compacte de mesure presque pleine dans $W$ et sur laquelle l'application $\sigma$ est continue (possible via le théorème de Lusin), ainsi qu'un point $c=(b, z, O, x) \in K$ ($\beta^+$-typique). Via le théorème de non-dégénérescence des mesures limites (section $5$), on peut construire une suite d'approximations $(c_{p})_{p\geq 0} \in K^\N$ du point $c$ de la forme $c_{p}=(b,z, O, x+u_{p})$ où le terme de dérive $u_{p} \in \R^d$, agit sur la coordonnée en $\T^d$, est bien positionné par rapport à $x$, et tend vers $0$ quand $p \rightarrow \infty$. Appelons $n$-fibre passant par $c$ l'ensemble $F_{n,c}:=\{c' \in B^+, (T^+)^n(c')=(T^+)^n(c)\}$.  Pour tout $p \geq 0$, on peut choisir des entiers $n_{p} \geq 0$ et des éléments $h_{p, c} \in F_{n_{p}, c}$, $h_{p,c_{p}} \in F_{n_{p}, c_{p}}$  dans les $n_{p}$-fibres passant par $c$ et $c_{p}$ tels que (quitte à extraire):
\begin{itemize}
\item $h_{p,c}, h_{p,c_{p}}  \in K$ (section $6.2$, équirépartition des morceaux fibres)
\item $h_{p,c_{p}} = h_{p,c} + D_{p}$ où le terme de dérive $D_{p} \in \R^d$ agit sur la coordonnée en $\T^d$ et se comporte bien du point de vue de la norme et de la direction (section $8.6$, contrôle de la dérive).
\end{itemize}

On peut de plus supposer que les suites $(h_{p,c})$ et $(h_{p,c_{p}})$ convergent dans $K$ vers des points limites $c', c'' \in K$. 

Alors $\sigma(c)=\sigma(c')=\sigma(c'')$. En effet, la continuité de $\sigma$ sur $K$ entraine que $\sigma(c')= \lim_{p \to \infty} \sigma(h_{p,c})$. Comme l'application $\sigma$ est $\mathcal{Q}^+_{\infty}$-mesurable, elle est constante sur les $n$-fibres passant par un point, donc $\sigma(h_{p,c})=\sigma(c)$ puis $\sigma(c')= \sigma(c)$. De même avec $c''$. 

De plus $c''= \phi_{t}c'$ où le paramètre $t$ est arbitrairement petit (quitte à bien choisir les $h_{p,c}$, $h_{p,c_{p}}$), grâce au bon comportement du terme de dérive $D_{p}$, ce qui conclut.

%%%%%%%%%%%%%%%%%%%%%%%%%%%%%%%%%%%%%%%%%%%%%%%%%%%%%%%%%%%%%%%%%%%%%%%%%%%%%%%%%%%%%%%%%%%%%%%%%%%%%%%%%%%%%%%%%%%%%%%%%%%%%%%%%%%%%%%%%%%%%%%%%%%%%%%%%%%%%%%%%%%%%%%%%%%%%%%%%%%%%%%%%%%%%%%%%%%%%%%%%%%%%%%%%%%%%%%%%%%%%%%%%%%%%%%%%%%%%%%%%%%%%%%%%%%%%%%%%%%%%%%%%%%%%%%%%%%%%%%%%%%%%%%%%%%%%%%%%%%%%%%%%%%
%%%%%%%%%%%%%%%%%%%%%%%%%%%%%%%%%%%%%%%%%%%%%%%%%%%%%%%%%%%%%%%%%%%%%%%%%%%%%%%%%%%%%%%%%%%%%%%%%%%%%%%%%%%%%%%%%%%%%%%%%%%%%%%%%%%%%%%%%%%%%%%%%%%%%%%%%%%%%%%%%%%%%%%%%%%%%%%%%%%%%%%%%%%%%%%%%%%%%%%%%%%%%%%%%%%%%%%%%%%%%%%%%%%%%%%%%%%%%%%%%%%%%%%%%%%%%%%%%%%%%%%%%%%%%%%%%%%%%%%%%%%%%%%%%%%%%%%%%%%%%%%%%%%
\newpage

\section{Non dégénérescence des mesures limites}

Les notations sont celles de la section $4$.  On rappelle que $\mu$ est une probabilité sur $SL_{d}(\Z)$ à support fini engendrant un semi-groupe $\Gamma \subseteq SL_{d}(\Z)$ fortement irréductible, que $\chi : \Gamma \rightarrow \R$ est un morphisme de semi-groupes tel que la probabilité image $\chi_{\star} \mu$ est centrée. Ces données induisent une action de $\Gamma$ sur $X:=\T^d \times \R$ via la formule $g.(x,t)=(gx, t + \chi(g))$, puis une marche aléatoire de probabilités de transitions données par $(\mu \star \delta_{x})_{x \in X}$. On s'est donné une mesure de Radon stationnaire ergodique $\nu \in \mathcal{M}^{Rad}(X)$ \emph{dont la projection sur $\T^d$ est sans atome}. On note par ailleurs $B:= \Gamma^{\N^\star}$, $\beta:= \mu^{\otimes \N^\star}$ et pour $\beta$-presque tout $b \in B$, $\nu_{b}:= \lim_{n \to \infty} (b_{1}\dots b_{n})_{\star} \nu \in \mathcal{M}^{Rad}(X)$ (cf. section $2$ sur les mesures limites). 

Le but de cette section est de montrer que les mesures limites $(\nu_{b})_{b \in B}$ se projettent sur $\T^d$ en des mesures sans atome. Autrement dit : 

\begin{th.}[Non dégénérescence]\label{THnondeg}
Pour $\beta$-presque tout $b \in B$,  on a $$\forall x\in \T^d,\,\,\,\,\,\, \nu_{b}(\{x\} \times \R)=0$$
\end{th.}

Signalons tout de suite le corollaire utilisé dans la preuve de la dérive exponentielle. Pour $b \in B$, $x \in \T^d$, on note :
$$W_{b}(x):= \{y \in \T^d, \,\lim_{n \to \infty}d(b^{-1}_{n}\dots b^{-1}_{1}x, b^{-1}_{n}\dots b^{-1}_{1}y) =0\}$$
où la distance considérée est la distance euclidienne standard sur le tore.

\begin{cor}\label{cornondeg}
Pour $\beta$-presque tout $b \in B$, pour tout $x \in \T^d$, on a  $$\nu_{b}(W_{b}(x) \times \R) =0$$.
\end{cor}

\begin{proof}[Preuve du corollaire]
Soit $x, y \in \T^d$, $x \neq y$. Pour $\beta$-presque tout $b \in B$, la suite $(b^{-1}_{n}\dots b^{-1}_{1}(x-y))_{n \geq 0}$ ne tend pas vers $0$ dans $\T^d$, i.e. $y \notin W_{b}(x)$. Soit $I \subseteq \R$ un intervalle borné. D'après le \cref{THnondeg}, pour $\beta$-presque tout $b \in B$, la mesure $\nu_{b,I}:=\nu_{b}(.\times I) \in \mathcal{M}^f(\T^d)$ est sans atome, donc  $\nu_{b,I}^{\otimes 2}$ ne charge pas la diagonale $\Delta_{\T^d} \subseteq \T^d \times \T^d$. On en déduit via le théorème de Fubini que pour $\beta$-presque tout $b \in B$, $$\nu_{b,I}^{\otimes 2}((x,y) \in \T^d\times \T^d, \, y \in W_{b}(x) )= 0$$
autrement dit que 
$$\int_{B} \nu_{b,I} (W_{b}(x))  d\nu_{b,I}(x) =0$$
Ainsi, pour $\nu_{b,I}$-presque tout $x \in \T^d$, on a $\nu_{b,I} (W_{b}(x)) =0$. C'est en fait vrai pour tout $x$. Sinon, il existerait $x \in \T^d$ tel que $\nu_{b,I} (W_{b}(x))>0$. On aurait alors pour tout $y \in W_{b}(x)$ que $\nu_{b,I} (W_{b}(y))>0$, ce qui contredit la première assertion. 
On a finalement montré que pour tout intervalle $I \subseteq \R$ borné, pour $\beta$-presque tout $b \in B$, pour tout $x \in \T^d$, on a  $\nu_{b}(W_{b}(x) \times I)=0$. En considérant des intervalles $I$ arbitrairement grands, on obtient le résultat.
\end{proof}

\bigskip

\subsection{Cas où $\chi(\Gamma) \subseteq \R$ est discret.}

L'image $\chi(\Gamma)$ est alors un sous-groupe de $\R$ (cf. \cref{ref0}). Le cas où $\chi(\Gamma)=\{0\}$ provient des travaux de Benoist-Quint $\cite{BQI}$. On peut donc supposer $\chi(\Gamma)$ discret non trivial, et même $\chi(\Gamma)=\Z$. Comme la mesure $\nu \in \mathcal{M}^{Rad}(X)$ est $\mu$-stationnaire ergodique sur $X= \T^d \times \R$, son support est inclus dans un sous ensemble de la forme $\T^d \times (r+\Z)$ où $r\in \R$. Quitte à considérer un translaté de $\nu$, on peut supposer que $\text{supp}\,\nu \subseteq \T^d \times \Z$. On est ainsi ramené à considérer une marche sur $\T^d \times \Z$. L'objectif est de montrer que les mesures limites $\nu_{b} \in \mathcal{M}^{Rad}(\T^d \times \Z)$ n'ont pas d'atome. Il suffit de montrer que les mesures limites restreintes au bloc de base $\nu_{b |\T^d \times \{0\}} \in \mathcal{M}^{Rad}(\T^d)$ n'ont pas d'atome (le cas des autres blocs s'en déduisant par translation). La stratégie est alors de voir les $\nu_{b |\T^d \times \{0\}} $ comme des mesures limites pour une marche induite sur le tore $\T^d$. Le point clef est de montrer que cette marche est récurrente hors de zéro. Un argument tiré de Benoist-Quint $\cite{BQI}$ implique alors la non dégénérescence cherchée. 

Construisons la marche induite. On définit un ``temps de retour au bloc de départ'' $\tau : B \rightarrow \N\cup \{\infty\} : b \mapsto \inf\{n \geq 1, \,\,\chi(b_{n} \dots b_{1})=0\}$ et on note $\mu_{\tau} \in \mathcal{P}(\Gamma)$ la loi de $b_{\tau(b)}\dots b_{1}$ quand $b$ varie suivant $\beta$. Notons que c'est aussi la loi de $b_{1}\dots b_{\tau(b)}$, donc que la mesure $\nu_{|\T^d\times \{0\}}$ est $\mu_{\tau}$-stationnaire d'après le \cref{LIMcor}. On pose $\beta_{\tau} :=\mu_{\tau}^{\otimes \N^\star} \in \mathcal{P}(B)$. 

La proposition suivante affirme que le point $0 \in \T^d$ est positivement instable pour la marche induite par $\mu_{\tau}$ sur $\T^d$ et permet de conclure. On note $P_{\mu_{\tau}}$ l'opérateur de Markov associé à $\mu_{\tau}$, défini sur les fonctions mesurables non négatives $f : \T^d \rightarrow [0, \infty]$ par $P_{\mu_{\tau}} f(x) = \int_{G} f(gx)\,d \mu_{\tau}(g)$. On dit qu'une fonction $u : \T^d-\{0\} \rightarrow [0, +\infty[$ est \emph{propre} si elle est continue et si l'image inverse de tout compact est  compacte. 

\begin{prop}\label{pref1}
Il existe une fonction propre $u : \T^d-\{0\} \rightarrow [0, +\infty[$ et des constantes $a \in [0,1[$, $C \in [0, +\infty[  $ telles que $$ P_{\mu_{\tau}}u \leq a u +C$$
\end{prop}

\bigskip

Voyons comment en déduire le théorème de non dégénérescence.
\begin{proof}[\Cref{pref1} $\implies$ \cref{THnondeg} :]

Notons $\nu_{0}:= \nu_{|\T^d \times\{0\}} \in \mathcal{M}^f(\T^d)$. C'est une mesure finie $\mu_{\tau}$-stationnaire sur $\T^d$, on peut donc s'en donner une décomposition en mesures limites $(\nu_{0,a})_{a \in B}$ (par rapport à la marche $\mu_{\tau}$). Vérifions qu'elles coïncident avec les $\nu_{b|\T^d\times \{0\}}$. On introduit  les différents temps de retour à $\chi=0$ en posant

\begin{equation*}
\left\{
\begin{aligned}
&\tau^0=0 \\
&\tau^{k+1} : B \rightarrow \N \cup \{\infty\}, b \mapsto \inf\{n > \tau^k(b), \,\chi(b_{n}\dots b_{1})=0\} \,\,\,  \end{aligned} \right.
\end{equation*}

 En particulier $\tau^1=\tau$. On pose alors  $\rho : B \rightarrow B, b \mapsto  (b_{\tau^k(b)+1}\dots b_{\tau^{k+1}(b)})_{k \geq 0}$. L'application $\rho$ vérifie $\rho_{\star} \beta = \beta_{\tau}$ et pour $\beta$-presque tout $b \in B$, on a $\nu_{b |\T^d \times \{0\}}= [\lim (b_{1}\dots b_{n})_{\star} \nu]_{|\T^d \times \{0\}}=[\lim (\rho(b)_{1}\dots \rho(b)_{n})_{\star} \nu]_{|\T^d \times \{0\}}= \lim (\rho(b)_{1}\dots \rho(b)_{n})_{\star} [\nu_{|\T^d \times \{0\}}]= \nu_{0, \rho(b)}$. On est donc ramené à montrer que les mesures limites de $\nu_{0}$ sont sans atome.    

Soit $u, a, C$ comme dans la proposition précédente et posons $$v : \T^d \times \T^d -\Delta_{\T^d} \rightarrow \R_{+}, (x,y) \mapsto u(x-y)$$ L'application $v$ est propre et vérifie $P_{\mu_{\tau}} v \leq a v +C$ (ce qui s'interprète comme un phénomène de récurrence hors de la diagonale pour la marche induite par $\mu_{\tau}$ sur $\T^d\times \T^d$). Le lemme $3.11$ de $\cite{BQI}$ affirme alors que pour $\beta_{\tau}$-presque tout $a \in B$, la probabilité $\nu_{0, a}$ n'a pas d'atome, ce qui conclut. 
\end{proof}

\bigskip

On va maintenant prouver la \cref{pref1}. La difficulté est que la probabilité $\mu_{\tau}$ n'a pas de moment d'ordre $1$. On ne peut donc pas appliquer les résultats de Benoist-Quint qui concernent les marches à moment exponentiel. 

\begin{lemme}\label{ref8}
La probabilité $\mu_{\tau}$ n'a pas de moment d'ordre $1$.
\end{lemme}

\begin{proof}
Il s'agit de prouver que $\int_{G} |\log(||g||)| d\mu_{\tau}(g)= +\infty$. On a la minoration 
$$\int_{G} |\log(||g||)| d\mu_{\tau}(g) \geq  \int_{\{\tau \geq k\}} |\log(||b_{\tau(b)}\dots b_{1}||) | d\beta(b)$$ 
 Or $\beta \{\tau \geq k\} \geq c k^{-1/2}$   pour un certain $c >0$ (voir par exemple \cite{LLRW},$(5.1.1)$). Par ailleurs,  d'après le principe des grandes déviations (\cite{BQRW}, $(13.6)$), il  existe $\alpha >0$ tel que pour $k$ assez grand,  on a $$\beta \{b \in B,  \forall n \geq k, \,||b_{n}\dots b_{1}|| \geq e^{\frac{n}{2}\lambda_{\mu} } \} \geq 1- e^{-\alpha k} $$ (où $\lambda_{ \mu} >0$ désigne l'exposant de Lyapunov associé à $\mu$, voir ($4.1$)). On en déduit que $\int_{G} |\log(||g||)| d\mu_{\tau}(g) \geq (ck^{-1/2}- e^{-\alpha k})\frac{k}{2}\lambda_{\mu}  \rightarrow  +\infty$ quand $k \to \infty$. D'où le résultat. 
\end{proof}

\bigskip

Cette difficulté est surmontée en considérant la $\mu_{\tau}$-marche comme une sous marche de la $\mu$-marche sur $\T^d$, qui elle est à moment exponentiel fini donc plus facile à contrôler (notamment grâce au principe des grandes déviations).

\bigskip

\subsection*{Preuve de la \cref{pref1} }

Dans toute cette preuve, on munira $\R^d$ de la distance euclidienne standard, et $\T^d=\R^d/\Z^d$ de la distance quotient. Plus précisément, en notant $\pi : \R^d \rightarrow \T^d$ la projection canonique, on pose pour $x,y \in \T^d$, $$d(x,y):= \inf\{d(\widetilde{x}, \widetilde{y}),\,\widetilde{x}\in \pi^{-1}(x), \widetilde{y}\in \pi^{-1}(y))\}$$

On va montrer le résultat suivant. 

\begin{prop5.3bis*}\label{pref1bis}
Fixons $\delta >0$ et notons $u$ la fonction $u : \mathbb{T}^d-\{0\} \rightarrow ]0,+\infty[, x \mapsto d(x, 0)^{-\delta}$. 
On peut choisir le paramètre $\delta$ pour qu'il existe des constantes $a \in [0,1[$, $C \in [0, +\infty[$, $k \geq 1$  telles que 
$$P^k_{\mu_{\tau}}u \leq a u + C $$

\end{prop5.3bis*}

\begin{rem*}
$1)$ La \cref{pref1} s'en déduit en considérant la fonction propre $\sum_{i=0}^{k-1} \alpha^i P^i_{\mu}u$ où $\alpha \in ]1, (\frac{1}{a})^{\frac{1}{k-1}}[$.

$2)$ $\mu^{\star k}_{\tau} = \mu_{\tau^k}$ où $\tau^k : B \rightarrow \N \cup \{\infty\}$ est le temps de $k$-ième atteinte de $\chi=0$. Ainsi, $P^k_{\mu_{\tau}}$ est  l'opérateur de Markov associé à la marche $\mu_{\tau_{k}}$ sur $\T^d$.

\end{rem*}

\begin{proof}
On fixe des paramètres $\delta>0$, $k\geq 1$ que l'on ajustera à la fin. On partitionne le tore $\T^d$ en deux régions : une petite boule ouverte centrée en $0$ notée $D := B(0, \varepsilon)$ et son complémentaire noté $R:= X - D$. L'hypothèse que la boule est petite signifiera $\varepsilon < \frac{1}{4}e^{-2M}$ où $M := \text{log}\sup \{||g||,||g^1||, g \in \text{supp}\,\mu\}$. 

Fixons $x \in \T^d-\{0\}$. On note $T^x_{1} := \inf\{n >0, \,\,b_{n}\dots b_{1}x \in R\}$,  $T^x_{2} := \inf\{n > T^x_{1}, \,\,b_{n}\dots b_{1}x \in R\}$, etc. les temps de passage dans $R$, et $T^x_{-} := \max\{T^x_{i}, \,\, T^x_{i} \leq \tau^k\}$ le dernier temps de passage dans $R$ avant de rencontrer $\chi =0$ pour la $k$-ième fois.

On écrit 
\begin{align*}
(P^k_{\mu_{\tau}} u)(x) &= \int_{B }u(b_{\tau^k(b)}\dots b_{1}x) \,\,d\beta(b)\\
				&= \int_{\{T^x_{-} =0\}}u(b_{\tau^k(b)}\dots b_{1}x) \,\,d\beta(b)\, +\, \int_{\{T^x_{-} \geq 1\}} u(b_{\tau^k(b)}\dots b_{1}x) \,\,d\beta(b)
\end{align*}			

\bigskip

Cette décomposition distingue les cas où la trajectoire $(b_{n}\dots b_{1}x)_{n \geq 0}$ sort ou non du disque $D$ avant l'instant $\tau^k(b)$. On va majorer chacune des intégrales en montrant :

\bigskip

$1)$ Si $\delta>0$ est assez petit, alors il existe une constante $C \in [0, +\infty[$ tel que pour tout $x \in \T^d-\{0\}$, on a $\int_{\{T^x_{-} \geq 1\}} u(b_{\tau^k(b)}\dots b_{1}x) \,\,d\beta(b)  \leq C$.

\bigskip

$2)$ Si $\delta>0$ est assez petit, et $k\geq1$ assez grand, alors il existe $a \in [0,1[$ tel que pour tout $x \in \T^d-\{0\}$, on a $\int_{\{T^x_{-} =0\}}u(b_{\tau^k(b)}\dots b_{1}x) \,\,d\beta(b)\, \leq au(x)$. 
	
\bigskip
La proposition s'en déduit immédiatement. Comme la preuve à venir est un peu technique, on explique d'abord la démarche mise en jeu :

\bigskip
\emph{Idée pour  $1)$ }: Dans ce cas, la $\mu$-marche sur $\T^d$ issue de $x$ passe par $R$ avant la $k$-ième annulation de $\chi$. Or le principe des grandes déviations implique que le temps de retour à $R$ de la $\mu$-marche sur $\T^d-\{0\}$ a un moment exponentiel. On s'attend donc à ce que la trajectoire issue de $x$ revienne souvent dans $R$ puis que l'instant de $k$-ième annulation de $\chi$ se fasse à proximité de $R$. La difficulté est que l'on ne sait pas quelle excursion hors de $R$ verra $\chi$ s'annuler pour la $k$-ième fois, autrement dit pour quelle valeur de $i \geq 1$, on a $T^x_{i}\leq \tau^k < T^x_{i+1}$. Cela oblige à sommer sur tous les cas possibles. Pour obtenir une majoration, on distingue alors suivant la valeur de $\chi$ à l'instant $T^x_{i}$.    

\emph{Idée pour  $2)$} : Dans ce cas, la $\mu$-marche sur $\T^d$ issue de $x$ reste dans la boule $D$ jusqu'à l'instant $\tau^k$. On peut donc voir cette marche comme une marche sur $\R^d-\{0\}$ sans perte d'information. On applique alors le principe des grandes déviations. 
\bigskip

\subsubsection*{Preuve du point $1)$}	

Le paramètre $k\geq 1$ ne sera utile pour le point $2)$. On conseille donc pour une première lecture de supposer $k=1$. 
	
En observant qu'un $\mu$-pas rapproche au plus de $0$ par un facteur $e^{-M}$, on obtient que :
$$\int_{\{T^x_{-} \geq 1\}} u(b_{\tau^k(b)}\dots b_{1}x) \,\,d\beta(b)  \leq \int_{\{T^x_{-}\geq 1\}} \varepsilon^{-\delta} e^{\delta M(\tau^k- T^x_{-})}\,\, d\beta$$

Il suffit donc de montrer que la variable $\tau^k- T^x_{-}$ a un moment exponentiel uniformément borné en $x \in X-\{0\}$ au sens où  :

\begin{lemme}\label{ref9}
Il existe des constantes $C_{1}>0$,  $\alpha_{1}>0$ telles que pour tout $x \in \T^d-\{0\}$, $n \geq 1$ on a :
$$\beta(T^x_{-}\geq 1, \,\,\tau^k -T^x_{-} \geq n) \leq C_{1}e^{-\alpha_{1}n}  $$
\end{lemme}

Dans cette preuve, on sera amené à considérer les variables aléatoires $S_{n} := B \rightarrow \Z, b \mapsto \chi(b_{n}\dots b_{1})$. Elles constituent une chaîne de Markov sur $\Z$ de probabilités de transition induites par $\mu_{\Z}:= \chi_{\star}\mu \in \mathcal{P}(\Z)$.  Rappelons que $\chi_{\star}\mu$ est  à support borné et d'espérance nulle, et que $S_{n}$ est en particulier récurrente. On se donne $l \geq 1$ tel que $\text{supp} \chi_{\star} \mu \subseteq [-l,l]$.

Soit $n \geq 1$, en notant $p_{j} = \lfloor{\text{log}^2j}\rfloor$ on a :
\begin{align*}
\beta(T^x_{-}\geq 1, \,\,\tau^k -T^x_{-} \geq n)&= \sum_{i \geq 1} \beta(T^x_{-} = T^x_{i}, \tau^k-T^x_{i} \geq n)\\
&= \sum_{i \geq 1} \beta(\tau^k > T^x_{i}, \,\, T^x_{i+1}-T^x_{i} > \tau^k -T^x_{i}  \geq n)\\
&= \sum_{i,j \geq 1} \beta(T^x_{i}=j, \,\tau^k >j, |S_{j}| >p_{j}, \,\, T^x_{i+1}-T^x_{i} > \tau^k -T^x_{i}  \geq n) \\&\,\,\,\,\,\,\,\,+\sum_{i,j \geq 1} \beta(T^x_{i}=j, \,\tau^k >j, |S_{j}| \leq p_{j}, \,\, T^x_{i+1}-T^x_{i} > \tau^k -T^x_{i}  \geq n) \\
%&\leq \Sigma_{1}+\Sigma_{2}\\
&\leq \sum_{i,j \geq 1} \beta(T^x_{i}=j, \,\tau^k >j, |S_{j}| >p_{j}, T^x_{i+1}- T^x_{i} > \frac{\frac{p_{j}}{l}+n}{2} ) \\&\,\,\,\,\,\,\,\,+\sum_{i,j \geq 1} \beta(T^x_{i}=j, \,\tau^k >j, |S_{j}| \leq p_{j}, T^x_{i+1}- T^x_{i} \geq n) \\
&\leq \sum_{i,j \geq 1} \beta(T^x_{i}=j,\, T^x_{i+1}- T^x_{i} > \frac{\frac{p_{j}}{l}+n}{2} ) \\&\,\,\,\,\,\,\,\,+\sum_{i,j \geq 1} \beta(T^x_{i}=j, \,\tau^k >j, |S_{j}| \leq p_{j}, T^x_{i+1}- T^x_{i} \geq n) 
\end{align*}

où la première des deux inégalités s'obtient en remarquant que si $S_{T^x_{-}} > p>0$, alors $\tau^k- T^x_{-} > p/l$.

Nous allons maintenant majorer ces deux sommes, que l'on note respectivement $\Sigma_{1}$ et $\Sigma_{2}$ :
\begin{align*}
&\Sigma_{1}:=\sum_{i,j \geq 1} \beta(T^x_{i}=j,\, T^x_{i+1}- T^x_{i} > \frac{\frac{p_{j}}{l}+n}{2} )\\
&\Sigma_{2}:=\sum_{i,j \geq 1} \beta(T^x_{i}=j, \,\tau^k >j, |S_{j}| \leq p_{j}, T^x_{i+1}- T^x_{i} \geq n) 
\end{align*}

Nous aurons besoin du lemme suivant qui découle directement du principe des grandes déviations :

\begin{lemme}\label{ref10}
Il existe des constantes $C_{2}>0$, $\alpha_{2}>0$ telles que pour tout $y \in R$, $n \geq 1$, on a 
$$\beta(T^y_{1} \geq n) \leq C_{2} e^{-\alpha_{2}n} $$
\end{lemme}

 Admettons ce lemme provisoirement.  
 
\bigskip

\emph{Majoration de $\Sigma_{1}$}
\bigskip

Pour $i,j \geq 1$, $y\in R$, on a la majoration de probabilité conditionnelle suivante :
\begin{align*}
\beta(T^x_{i+1} -T^x_{i} > \frac{\frac{p_{j}}{l} +n}{2}\,\, |\,\,T^x_{i}=j, \, b_{j}\dots b_{1}x=y )
&= \beta(T^x_{i+1} -T^x_{i} > \frac{\frac{p_{j}}{l} +n}{2}\,\, |\,\,  b_{T^x_{i}}\dots b_{1}x=y )\\
&=\beta(T^y_{1} >  \frac{\frac{p_{j}}{l} +n}{2})\\
&\leq C_{2} e^{-\alpha_{2}\frac{{p_{j}} +ln}{2l}}
\end{align*}

Aini, on peut majorer $\Sigma_{1}$ via :
\begin{align*}
\Sigma_{1} &=\sum_{i,j \geq 1}\sum_{y \in R\cap \Gamma.x} \beta(T^x_{i}=j,\, T^x_{i+1}- T^x_{i} > \frac{\frac{p_{j}}{l}+n}{2},\, b_{j}\dots b_{1}x=y )\\
&\leq \sum_{i,j \geq 1}\sum_{y \in R\cap \Gamma.x} \beta(T^x_{i}=j, \, b_{j}\dots b_{1}x=y ) C_{2}e^{-\alpha_{2}\frac{{p_{j}} +ln}{2l}}\\
&= \sum_{i,j \geq 1} \beta(T^x_{i}=j) C_{2}e^{-\alpha_{2}\frac{{p_{j}} +ln}{2l}}\\ 
&\leq (C_{2} \sum_{j \geq 1}e^{-\alpha_{2}\frac{p_{j}}{2l}} )\,e^{-\alpha_{2}\frac{n}{2} }
\end{align*}
 où la somme $\sum_{j \geq 1}e^{-\alpha_{2}\frac{p_{j}}{2l}}$ est finie car  car $e^{-\alpha_{2}\frac{p_{j}}{2l}} = o(\frac{1}{j^2})$. Notons que la dépendance en $x$ a disparu.

\bigskip

\newpage

\emph{Majoration de $\Sigma_{2}$}

\bigskip

On raisonne de la même fa\c con que pour $\Sigma_{1}$ afin d'obtenir : 
\begin{align*}
\Sigma_{2} &\leq  \sum_{i,j \geq 1} \beta(T^x_{i}=j, \,\tau^k >j, |S_{j}|\leq p_{j}) C_{2}e^{-\alpha_{2}n}\\ 
&\leq C_{2} \sum_{j \geq 1} \beta( \tau^k >j, |S_{j}|\leq p_{j})\,\, e^{-\alpha_{2}n}
\end{align*}

Il reste à montrer que la somme $ \sum_{j \geq 1} \beta( \tau^k >j, |S_{j}|\leq p_{j})$ est finie. Cela découle des propriétés générales des marches aléatoires sur $\Z$ induites par une probabilité de transition centrée à support borné (voir plus loin). Notons par ailleurs que la dépendance en $x$ a disparu. 

\bigskip

Pour conclure, on obtient la majoration : 

$$\beta(T^x_{-}\geq 1, \,\,\tau^k -T^x_{-} \geq n) \leq (C_{2} \sum_{j \geq 1}e^{-\alpha_{2}\frac{p_{j}}{2l}} )\,e^{-\alpha_{2}\frac{n}{2} } + C_{2} \sum_{j \geq 1} \beta( \tau^k >j, |S_{j}|\leq p_{j})\,\, e^{-\alpha_{2}n}$$

ce qui prouve le \cref{ref9}, donc le point $1)$. 	
	
\end{proof}

\subsubsection*{Preuve du point $2)$}	

Si $d(x,0) > \varepsilon\, e^M$, alors le résultat est immédiat car alors $\{T^x_{-}=0\}$ est de mesure nulle. On peut donc supposer $d(x,0) < \varepsilon\, e^M < \frac{1}{4}e^{-M}$ où la deuxième inégalité provient du choix de $\varepsilon$. Procédons à la majoration :

\begin{align*}
\int_{\{T^x_{-} =0\}}u(b_{\tau^k(b)}\dots b_{1}x) \,\,d\beta(b) &= \sum_{n \geq k } \int_{\{T^x_{-} =0, \tau^k=n\}} u(b_{n}\dots b_{1}x) \,\,d\beta(b)\\
&\leq \sum_{n \geq k} \int_{B} ||b_{n}\dots b_{1}\widetilde{x}||^{-\delta} \,d\beta(b)
\end{align*} 

où $||.||$ dénote la norme euclidienne standard sur $\R^d$, où $\widetilde{x}\in \R^d-\{0\}$ désigne l'unique relevé de $x$ de norme $||\widetilde{x}||<\frac{1}{4}e^{-M}$, et où la deuxième inégalité provient du fait que si $T^x_{-}(b)=0$ alors $u(b_{\tau^k(b)}\dots b_{1}x )=||b_{\tau^k(b)}\dots b_{1}\widetilde{x}||^{-\delta}$. On poursuit la majoration grâce au principe des grandes déviations (\cite{BQRW}, $(13.6)$) qui implique le résultat suivant :

\begin{lemme}\label{ref11}
Pour $v \in \R^d-\{0\}$, $n\geq 1$, notons $E^n_{v} := \{b \in B, ||b_{n}\dots b_{1} v|| \geq e^{\frac{\lambda_{\mu}}{2}n}||v|| \}$. Alors il existe des constantes $C_{3} >0$, $\alpha_{3}>0$ telles que pour tout $v \in \R^d-\{0\}, n \geq 1$, on a: 
$$\beta(E^n_{v}) \geq 1- C_{3}e^{-\alpha_{3}n} $$
\end{lemme}

On décompose alors 
\begin{align*}
\int_{B} ||b_{n}\dots b_{1}\widetilde{x}||^{-\delta} \,d\beta(b) &=  \int_{E^n_{\widetilde{x}}} ||b_{n}\dots b_{1}\widetilde{x}||^{-\delta} \,d\beta(b) +  \int_{B-E^n_{\widetilde{x}}} ||b_{n}\dots b_{1}\widetilde{x}||^{-\delta} \,d\beta(b)\\
&\leq e^{-\frac{\delta \lambda_{\mu}}{2} n}||\widetilde{x}||^{-\delta} + C_{3} e^{-\alpha_{3}n} e^{\delta M n}||\widetilde{x}||^{-\delta} 
\end{align*}

On choisit alors $\delta$ tel que $0<\delta < \frac{\alpha_{3}}{M}$.

$$\int_{B} ||b_{n}\dots b_{1}\widetilde{x}||^{-\delta} \,d\beta(b)  \leq  (\sum_{n \geq k} e^{-\frac{\delta \lambda_{\mu}}{2} n}+C_{3} e^{-(\alpha_{3}-\delta M) n})\,u(x) $$
 
Le facteur devant $u(x)$ étant le reste d'une série convergente, il est strictement inférieur à $1$ pour $k$ assez grand. Cela conclut la preuve du point $2)$.

\subsubsection*{Quelques précisions}

$1)$ \emph{Preuve du \cref{ref10}} : on peut se limiter à considérer les points $y \in R$ tels que $\beta(T^y_{1} \geq 2)>0$.  Soit un tel $y \in R$. Nécessairement, $y \in B(0, \varepsilon e^M) \subseteq B(0, \frac{1}{4} e^{-M})$. Notons $\widetilde{y} \in \R^d$ son unique de relevé de norme $||\widetilde{y}|| \leq \varepsilon e^M$. Soit $n \geq 1$. Si $b \in B$ est tel que $T^y_{1}(b) \geq n$, alors on voit par récurrence que pour tout $k =1, \dots, n$, on a $||b_{k}\dots b_{1}\widetilde{y}|| < \varepsilon < ||\widetilde{y}||$. En particulier, $b \notin E^n_{\widetilde{y}}$ (défini dans le  \cref{ref11}). On a donc pour tout $n \geq 1$ l'inégalité  $\beta(T^y_{1} \geq n)\leq \beta(B-E^n_{\widetilde{y}}) \leq C_{3}e^{-\alpha_{3} n}$ avec $C_{3}, \alpha_{3}>0$ indépendants de $y$ ce qui conclut.

\bigskip

$2)$ 
A la fin de la majoration de $\Sigma_{2}$, dans le point $1)$ nous avons affirmé que  $\sum_{j \geq 1} \beta( \tau^k >j, |S_{j}|\leq p_{j}) >\infty$. Expliquons pourquoi. Comme c'est un problème de marche aléatoire général, on préfère revenir à un cadre plus abstrait. On se donne une probabilité $m \in \mathcal{P}(\Z)$ à support fini, centrée, et on note $(S_{n})_{n \in \Z}$ la chaîne de Markov sur $\Z$ de probabilités de transition $(m\star \delta_{k})_{k \in \Z}$. On fixe un entier $k \geq 1$, et on s'intéresse au temps de  $k$-ième passage en $0$, noté $\tau^k$. Pour $l \in \Z$, on notera $\mathbb{P}_{l}$ la loi de la chaîne $(S_{n})_{n \geq 0}$ issue du point $l$. 

On note $F_{j}$ l'évènement $F_{j} := \{\tau^k >j, |S_{j}|\leq p_{j}\}$ où  $p_{j} := \lfloor \text{log}^2j \rfloor$. On veut démontrer que
\begin{equation*}
\sum_{j \geq 1} \mathbb{P}_{0}(F_{j}) <\infty 
\end{equation*}
%Notons $\tau_{\Z^-}$ le temps de premier retour au demi-axe négatif $\Z^-:= \{l \in \Z, l\leq 0\}$. Nous utiliserons le lemme suivant, démontré dans \cite{LLRW} (Proposition $5.15$) :

%\begin{lemme}\label{LL}
%Il existe une constante $c >0$ telle que pour tout $l \in \N$, pour tout $r>0$, 
%$$\mathbb{P}_{l}(\tau_{\Z^{-}} > r) \leq c\,\frac{l+1}{\sqrt{r}}$$
%\end{lemme}

L'idée est que la réalisation de l'évènement $F_{j}$ contraint le temps $\tau^k$ à appartenir à une petite fenêtre de valeurs, puis de conclure en utilisant que $\tau^k$ a un moment d'ordre $1/4$. 

On fixe $a \in \N$ tel que les $X_{i}$ sont presque sûrement à valeurs dans l'intervalle $A:=\llbracket -a, a\rrbracket$, et on note $\tau_{A}:= \inf\{n >0,\, S_{n}\in A\}$ le temps de premier retour à  $A$. Nous nous appuierons sur le lemme suivant :

\begin{lemme} \label{LL}
Il existe une constante $c >0$ telle que pour tout $l \in \Z$, pour tout $r>0$, 
$$\mathbb{P}_{l}(\tau_{A} > r) \leq c\,\frac{|l|+1}{\sqrt{r}}$$
\end{lemme}

\begin{proof}
Notons $\tau_{\Z^-}, \tau_{\Z^+}$ les temps de premier retour aux demi-axes négatif $\Z^-:= \{l \in \Z, l\leq 0\}$ et positif $\Z^+=\N$. La référence \cite{LLRW} (Proposition $5.15$) affirme qu'il existe une constante $c_{-} >0$ telle que pour tout $l \geq 0$, pour tout $r>0$, 
$$\mathbb{P}_{l}(\tau_{\Z^{-}} > r) \leq c_{-}\,\frac{l+1}{\sqrt{r}}$$
Par symétrie, il existe une constante $c_{+} >0$ telle que pour tout $l \leq 0$, pour tout $r>0$, 
$$\mathbb{P}_{l}(\tau_{\Z^{+}} > r) \leq c_{+}\,\frac{|l|+1}{\sqrt{r}}$$
Finalement, en posant $c:= \max(c_{-}, c_{+})$ et en remarquant que tout changement de demi-axe impose de rencontrer $A$, on obtient le résultat. 
\end{proof}

Pour la suite, on se donne $\varepsilon \in ]0, \frac{1}{2a+1}[$ et $N \geq 0$ tels que pour tout $l \in A$, on a $\mathbb{P}_{l}(\tau^k\leq N)\geq 1- \varepsilon$. On commence par donner une première majoration de $\mathbb{P}_{0}(F_{j})$. Remarquons l'inégalité :

 $$\mathbb{P}_{0}(j < \tau^k \leq j+ j^{1/4} \,|\, F_{j})\,\, \geq\,\, \inf_{l \in \llbracket -p_{j}, p_{j} \rrbracket}\mathbb{P}_{l}((S_{n})_{n \leq j^{1/4}} \text{ rencontre $0$ au moins $k$ fois})$$

On vérifie que ce minorant tend vers $1$ quand $j \to +\infty$.  On a pour $j \geq 0$ assez grand, pour tout $l \in \llbracket -p_{j}, p_{j} \rrbracket$, 
\begin{align*}
\mathbb{P}_{l}(\tau^k \leq j^{\frac{1}{4}}) &\geq \mathbb{P}_{l}(\tau_{A} \leq \frac{1}{2}j^{\frac{1}{4}}, \,\text{ et la suite } (S_{n})_{\tau_{A} \leq n \leq \tau_{A}+ \frac{1}{2}j^{\frac{1}{4}}} \text{ rencontre $k$ fois $0$})\\
&\geq \mathbb{P}_{l}(\tau_{A} \leq \frac{1}{2}j^{\frac{1}{4}}) (1- \varepsilon)\\
&\geq (1- 2c\frac{p_{j}+1}{j^{1/8}})(1- \varepsilon)\\
& \geq 1- \varepsilon/2
\end{align*}

Comme la valeur $\varepsilon >0$ peut être choisie arbitrairement petite, cela donne la convergence annoncée.

Ainsi, pour  $j \geq 0$ assez grand, on a  l'inégalité  :
$$\mathbb{P}_{0}(j < \tau^k \leq j+ j^{1/4} \,|\, F_{j})\geq \frac{1}{2}$$
i.e.
$$\mathbb{P}_{0}(j < \tau^k \leq j+ j^{1/4} ) \geq \frac{1}{2} \mathbb{P}_{0}(F_{j})$$

\bigskip

On est donc ramené à majorer $\sum_{j \geq 0} \mathbb{P}_{0}(j < \tau^k \leq j+ j^{1/4})$. Or on peut écrire 
\begin{align*}
\sum_{j \geq 0} \mathbb{P}_{0}(j < \tau^k \leq j+ j^{1/4}) &= \sum_{j \geq 0} \sum_{j < n \leq j+ j^{1/4} }\mathbb{P}_{0}(\tau^k =n)\\ 
&\leq \sum_{n \geq 0} n^{\frac{1}{4}}\mathbb{P}_{0}(\tau^k =n) \\
&= \mathbb{E}[(\tau^k)^{\frac{1}{4}}]
\end{align*}
 où l'inégalité provient de la condition $j < n \leq j+ j^{1/4}$ qui implique que $j \in  \N \cap [n- n^{\frac{1}{4}}, n-1 ]$, intervalle d'entiers de cardinal au plus $n^{\frac{1}{4}}$. 

Il suffit donc de montrer que $\tau^k$ admet un moment d'ordre $1/4$. Cela équivaut à montrer la convergence $$\sum_{j \geq 0 } \mathbb{P}_{0}(\tau^k > j^{4}) < \infty$$

Ecrivons $ \mathbb{P}_{0}(\tau^k > j^{4}) = u_{j}+v_{j}$ où $$u_{j}:=  \mathbb{P}_{0}(\tau^k > j^{4}, \widebar{\tau}^{p_{j}}_{A} \leq j^{4}-N) \,\,\,\,\,\,\,\,\,\,\,\,v_{j}:=  \mathbb{P}_{0}(\tau^k > j^{4}, \widebar{\tau}^{p_{j}}_{A} > j^{4}-N)$$ 

avec $ \widebar{\tau}^{p_{j}}_{A}$ défini par récurrence en posant :
$\widebar{\tau}^1_{A}= \tau_{A}$ et $\widebar{\tau}^{i+1}_{A}:=\inf\{n \geq \widebar{\tau}^i_{A} +N\, |\, S_{n}\in A \}$ (temps de retour à $A$ avec écart). 
\begin{itemize} 
\item
$\sum_{j \geq 0} u_{j}<\infty$ :

 Pour $l_{1}, \dots, l_{p_{j}} \in A$, on définit l'évènement :
 $$E(l_{1}, \dots, l_{p_{j}}) := \bigcap_{i=1\dots p_{j}}\{S_{\widebar{\tau}^i_{A}}=l_{i}, \text{et la suite $(S_{n})_{\widebar{\tau}^i_{A} \leq n < \widebar{\tau}^i_{A} +N }$ rencontre $0$ au plus $k-1$ fois  }\}$$
 
 On a $u_{j} = \sum_{l_{1}, \dots, l_{p_{j}} \in A} \mathbb{P}_{0}(\tau > j^4, \widebar{\tau}^{p_{j}}_{A} \leq j^{4}-N, E(l_{1}, \dots, l_{p_{j}})) \leq \sum_{l_{1}, \dots, l_{p_{j}} \in A} \mathbb{P}_{0}(E(l_{1}, \dots, l_{p_{j}})) $.
 Or $ \mathbb{P}_{0}(E(l_{1}, \dots, l_{p_{j}})) \leq  \varepsilon \mathbb{P}_{0}(E(l_{1}, \dots, l_{p_{j}-1} )) \leq \varepsilon ^{p_{j}}$. Finalement $u_{j}\leq ((2a+1)\varepsilon)^{p_{j}}$ et la somme converge par choix de $\varepsilon <   \frac{1}{2a+1}$. 
 
\item
$\sum_{j \geq 0} v_{j}<\infty$ : Notons $\tau^i_{A}$ le temps de $i$-ème retour à l'intervalle $A$ (sans contrainte d'écart minimal entre $\tau^i_{A}$ et $\tau^{i+1}_{A}$) avec pour convention $\tau^0_{A}=0$. Pour $j > N^{\frac{1}{4}}$, on a 
\begin{align*}
v_{j}&\leq \mathbb{P}_{0}(\widebar{\tau}^{p_{j}}_{A} > j^{4}-N)\\
 &\leq \mathbb{P}_{0}(\tau^{Np_{j}}_{A} > j^{4}-N)\\
&\leq \sum_{0\leq i \leq Np_{j}-1}\mathbb{P}_{0}(\tau^{i+1}_{A} -\tau^{i}_{A}> \frac{j^{4}-N}{Np_{j}})\\
& \leq Np_{j} c\frac{a\sqrt{Np_{j}}}{ \sqrt{j^4-N}} \tag{d'après le \cref{LL}}\\
&= O(\frac{p_{j}^{3/2}}{j^2})
\end{align*}
 . Les $(v_{j})$ sont donc de série convergente.

\end{itemize}

\subsection{Cas où $\chi(\Gamma) \subseteq \R$ est dense}
\label{Nondegnonat}
\bigskip

Introduisons quelques notations utiles. Etant donné un intervalle $I \subseteq \R$, on note  $\nu_{I}:= \nu(.\times I)$, $\nu_{b,I}:= \nu_{b}(.\times I) \in \mathcal{M}^f(\T^d)$.  Le \cref{torenu_{b}} définissant les $\nu_{b}$ implique que 
\begin{itemize}
\item $\nu_{I}= \int_{B}\nu_{b, I} \,\, d\beta(b)$
\item $\forall^\beta b \in B,\,\, \nu_{b,I}(\T^d) = \nu_{I}(\T^d)$
\item $\forall^\beta b \in B,\,\,  (b_{1})_{\star}\nu_{Tb,I}= \nu_{b, I - \chi(b_{1})}$
\end{itemize}

Par ailleurs, on rappelle le deuxième point du \cref{premcor} : il existe une constante $c>0$ tel que la mesure projetée $\nu(\T^d \times .)=c \leb$. Quitte à normaliser $\nu$, on pourra supposer que $$\nu(\T^d \times .)=\leb$$
 Le  \cref{torenu_{b}} implique alors que pour $\beta$-presque tout $b \in B$, on a $\nu_{b}(\T^d\times .)= \leb$, ou de fa\c con équivalente  $\nu_{b,I}(\T^d)= \leb(I)$ pour tout intervalle $I\subseteq \R$. 
\bigskip

\begin{proof}[Preuve du \cref{THnondeg}]

Il suffit de montrer que pour $\beta$-presque tout $b \in B$, pour tout intervalle $I \subseteq \R$ de longueur $1$, la mesure $\nu_{b,I}$ sur $\T^d$ est sans atome. On suppose par l'absurde que ce n'est pas le cas. Il existe alors $I_{0} \subseteq \R$ intervalle de longueur $1$ tel que 
$$ \beta \{b \in B, \,\,\nu_{b,I_{0}} \text{ a au moins un atome}    \}>0\,\,\,\,\,\,\,\,\,(\star)$$

La première étape est de montrer que les $\nu_{b,I}$ sont tous atomiques, avec les mêmes poids :

\begin{lemme}\label{ref12}
Il existe un ensemble dénombrable $S$, une famille de coefficients $(\alpha_{i})_{i \in S} \in \R^S_{>0}$ tels que pour $\beta$-presque tout $b \in B$, pour tout intervalle $I \subseteq \R$ de longueur $1$, il existe une famille $(x_{i})_{i \in S} \in (\T^d)^S$ de points deux à deux distincts de $\T^d$ tels que  la mesure $\nu_{b,I}$ est de la forme : $$\nu_{b,I} = \sum_{i \in S} \alpha_{i} \delta_{x_{i}}$$

\end{lemme}

\begin{proof}

Commen\c cons par prouver que les $\nu_{b,I}$ sont atomiques. L'idée est de décomposer $\nu_{b}$ comme somme d'une mesure dont la projection sur $\T^d$ est sans atome, et d'une autre dont la projection sur $\T^d$ est purement atomique, puis d'utiliser l'ergodicité de $\nu$ pour montrer que la première est nulle.

\begin{lemme}\label{ref13}
Pour $\beta$-presque tout $b \in B$, il existe une décomposition
 $$\nu_{b}= \nu^{nat}_{b} + \nu^{at}_{b}$$
  où $\nu^{nat}_{b}, \nu^{at}_{b} \in \mathcal{M}^{Rad}(\T^d\times \R)$, telle que pour tout $I \subseteq \R$, $\nu^{nat}_{b,I}$ est sans atome, $\nu^{at}_{b,I}$ est atomique. 
De plus, cette décomposition est unique et $\mu$-équivariante : $\forall^\beta b \in B, \,\, (b_{1})_{\star} \nu^{nat}_{Tb}= \nu^{nat}_{b}, \,\,\,(b_{1})_{\star} \nu^{at}_{Tb}= \nu^{at}_{b}$.
 \end{lemme} 

\begin{proof}[Preuve du \cref{ref13}]
Soit $b \in B$ tel que $\nu_{b}(\T^d\times .)= \leb$. On pose $A_{b}:= \{x \in \T^d, \,\, \nu_{b}(\{x\}\times \R)  \in ]0, +\infty])\}$, puis $\nu^{at}_{b}:= \nu_{b |A_{b}\times \R}$, $\nu^{nat}_{b}:=  \nu_{b |A^c_{b}\times \R}$. Par défintion, pour tout intervalle $I$, les mesures $\nu^{at}_{b,I}$, $\nu^{nat}_{b}$ sont respectivement atomiques et sans atome. L'unicité se vérifie directement. L'équivariance annoncée provient de l'équivariance des $\nu_{b}$, qui implique notamment que pour $\beta$-presque tout $b \in B$, on a $b_{1}A_{Tb}= A_{b}$. 
\end{proof}

Montrons maintenant que $\beta$-presque sûrement, on a $\nu_{b} = \nu^{at}_{b}$. On pose $\nu^{nat} := \int_{B} \nu^{nat}_{b} \,\,d\beta(b)$, $\nu^{at} := \int_{B} \nu^{at}_{b} \,\,d\beta(b)$. Ce sont des mesures $\mu$ stationnaires grâce à la relation d'équivariance du lemme $2$. De plus, $\nu^{at} \neq 0$ par hypothèse $(\star)$.  On a $\nu = \nu^{nat}+ \nu^{at}$ donc par ergodicité de $\nu$, il existe $c>0$ tel que $\nu = c \nu^{at}$. Or les mesures limites de $\nu^{at}$ sont les $(\nu^{at}_{b})$. On en déduit que les $\nu_{b,I}= c \nu^{at}_{b,I}$ sont atomiques, puis que $\nu^{nat}_{b}=0$. Finalement, $\nu= \nu^{at}$ et pour $\beta$-presque tout $b \in B$,  $\nu_{b}= \nu^{at}_{b}$. 

\bigskip

On prouve maintenant que les poids des atomes intervenant dans $\nu_{b,I}$, où $I \subseteq \R$ est un intervalle de longueur $1$, ne dépendent pas de $b$ et $I$. On rappelle qu'au vue de la normalisation de $\nu$, presque tous les $\nu_{b,I}$ sont de masse $1$. Soit $N \geq 0$, $k=(k_{i})_{i \geq 1} \in \{0, \dots ,2^N-1\}^{\N^{\star}}$. On définit $E_{N,k} := \Pi_{i \geq 1} ]\frac{k_{i}}{2^N}, \frac{k_{i}+1}{2^N}] \subseteq ]0,1]^{\N^\star}$.  On note $\mathcal{E}_{N,k} \subseteq B$ l'ensemble des $b \in B$ tels que les poids des atomes de $\nu_{b, [0,1]}$ classés par ordre décroissant appartiennent à $E_{N,k}$.

On a ainsi des partitions $B=\mathcal{E}_{0,0}$, et $\mathcal{E}_{N,k}= \amalg \{\mathcal{E}_{N+1, l},\,\,\, E_{N+1,l} \subseteq E_{N,k}\} $.  De plus, à $N$ fixé, la probabilité $\beta$ est concentrée sur une union finie d'ensembles $\mathcal{E}_{N,k}$ où $k$ varie dans $\{0, \dots ,2^N-1\}^{\N^{\star}}$. On  peut donc se donner une suite décroissante $(\mathcal{E}_{N, k^{(N)}})_{N \geq 0}$ (pour l'inclusion)  telle que pour tout $N \geq0$, on a $\beta(\mathcal{E}_{N, k^{(N)}})>0$. Pour $i \geq 1$, la suite $(\frac{k^{(N)}_{i}}{2^N}) \in [0,1]^{\N^\star}$ est de Cauchy, on note $\alpha_{i}:= \lim_{N \to \infty} \frac{k^{(N)}_{i}}{2^N}$. La suite des $(\alpha_{i})$ est décroissante et $\sum_{i \geq 1} \alpha_{i} = 1$ (en se rappelant que $\nu$ est normalisée, donc que les $\nu_{b,[0,1]}$ sont tous de masse $1$).

On pose $S:= \{i \geq 1, \,\, \alpha_{i}>0\}$ et on vérifie que la famille de coefficients $(\alpha_{i})_{i \in S}$ convient. Soit $b \in B$ $\beta$-typique, écrivons $\nu_{b, [0,1]} = \sum_{i \in \llbracket 1, p \rrbracket } c_{i} \delta_{x_{i}}$ où $p \in \N \cup \{\infty\}$, la suite $(c_{i})$ est décroissante strictement positive, les $x_{i}$ sont deux à deux distincts. Soit $N \geq 0$. Comme $\beta(\mathcal{E}_{N,k})>0$, il existe une extraction $\sigma : \N \rightarrow \N$ telle que pour tout $n \geq 0$, $T^{\sigma(n)}(b) \in \mathcal{E}_{N,k} $ et $\chi(b_{1}\dots b_{\sigma(n)}) \underset{n \to \infty}{\rightarrow} 0$ (par conservativité et ergodicité de la marche $\chi_{\star} \mu$ sur $\R$ pour la mesure de lebesgue). De plus, on a l'égalité :

$$ \nu_{b, [0,1]} = (b_{1}\dots b_{\sigma(n)})_{\star} \nu_{T^{\sigma(n)}b,\, [0,1]+\epsilon_{n}} $$
où $\epsilon_{n} = \chi(b_{1}\dots b_{\sigma(n)})$. 

Les poids des atomes de $\nu_{b,[0,1]}$ sont donc les mêmes que ceux de $\nu_{T^{\sigma(n)}b,\, [0,1]+\epsilon_{n}}$. Comme on a normalisé $\nu$ pour avoir $\nu(\T^d\times [0,1])=1$, on a $|\nu_{T^{\sigma(n)}b,\, [0,1]+\epsilon_{n}} - \nu_{T^{\sigma(n)}b, [0,1]} | \leq \epsilon_{n}$. Le fait que $T^{\sigma(n)}(b) \in \mathcal{E}_{N,k} $ implique alors que pour tout $i \in \llbracket 1, p \rrbracket$, pour tout $n \geq 0$ assez grand, on a $ \frac{k^{(N)}_{i}}{2^N} -\epsilon_{n} \leq c_{i} \leq \frac{k^{(N)}_{i} +1}{2^N} +\epsilon_{n}$. On fait  tendre $n \to +\infty$ pour obtenir $ \frac{k^{(N)}_{i}}{2^N} \leq c_{i} \leq \frac{k^{(N)}_{i} +1}{2^N}$, puis $N \to +\infty$ pour obtenir $c_{i}= \alpha_{i}$. On a montré que $ \llbracket 1, p \rrbracket \subseteq S$ et pour tout $i \in \llbracket 1, p \rrbracket, c_{i}= \alpha_{i}$. Comme $\sum_{i \in  \llbracket 1, p \rrbracket} c_{i} =1 = \sum_{i \in S} \alpha_{i}$, on a nécessairement que $ \llbracket 1, p \rrbracket = S$, ce qui conclut la preuve du \cref{ref12}.

\end{proof}

La suite de la preuve consiste à montrer que les mesures $\nu_{I} \in \mathcal{M}^f(\T^d)$ ne dépendent pas de l'intervalle $I$ de longueur $1$. Cela implique qu'à $I$ fixé, la mesure projetée  $\nu_{I}$ est  $\mu$-stationnaire sans atome, donc que les $\nu_{b,I}$ sont sans atomes (via Benoist-Quint). Pour prouver que $\nu_{I}$ ne dépend pas de $I$, il suffit de remarquer que c'est le cas des mesures limites $\nu_{b,I}$. 

\begin{lemme}
Pour $\beta$-presque tout $b\in B$, pour tous intervalles $I, J \subseteq \R$ de longueurs $1$ on a $\nu_{b,I}= \nu_{b,J}$. En particulier, $\nu_{I}=\nu_{J}$.
\end{lemme}

\begin{proof}
Soit $b \in B$ comme dans le \cref{ref12} et tel que $\nu_{b}(\T^d\times .)=\leb$. On énumère l'ensemble des poids $\alpha_{i}$ en une suite strictement décroissante $a_{1} := \max \alpha_{i} > a_{2} := \max_{\alpha_{i} < a_{1}} \alpha_{i}>\dots$ éventuellement finie.  On raisonne par couches en montrant d'abord que les atomes de masse $a_{1}$ de $\nu_{b,I}$ et $\nu_{b,J}$ sont les mêmes, puis que ceux de masse $a_{2}$ sont les mêmes etc. 

Soit $\epsilon  \in ]0, a_{1}-a_{2}[$. Pour tout $x \in \T^d$ tel que $\nu_{b,I}(x)=a_{1}$, on a $\nu_{b, I+ \epsilon} (x) \geq a_{1}- \epsilon > a_{2}$ (car $\nu_{b}(\T^d\times .)=\leb$). On a ainsi $\nu_{b, I+ \epsilon} (x) = a_{1}$ d'après le \cref{ref12}. Par symétrie, on a donc montré que $\{x \in \T^d, \nu_{b,I}=a_{1}\} =\{x \in \T^d, \nu_{b,I+\epsilon}=a_{1}\}$. Par connexité de $\R$, on en déduit que $\{x \in \T^d, \nu_{b,I}=a_{1}\} = \{x \in \T^d, \nu_{b,J}=a_{1}\} $.

On reprend ensuite le raisonnement précédant avec $\epsilon  \in ]0, a_{2}-a_{3}[$, puis $\epsilon  \in ]0, a_{3}-a_{4}[$, etc.  pour en déduire l'égalité des couches successives. Finalement, $\nu_{b,I}= \nu_{b,J}$.

Enfin, $\nu_{I}=\nu_{J}$ en intégrant l'égalité obtenue. 

\end{proof}

\begin{cor*}
Pour $I \subseteq \R$ intervalle de longueur $1$, la mesure $\nu_{I}$ est $\mu$-stationnaire.
\end{cor*}

\begin{proof}
Soit $A \subseteq \T^d$ mesurable. On a $\int_{G} g_{\star} \nu_{I}(A) d\mu(g) = \int_{G} \nu(g^{-1}A \times I) \, d\mu(g) =\int_{G} \nu(g^{-1}A \times I - \chi(g)) \, d\mu(g) = \int_{G} (g_{\star}\nu)(A \times I) \, d\mu(g) = \nu_{I}(A)$.
\end{proof}

\emph{Conclusion}. Soit $I_{0}$ l'intervalle mentionné au début de la preuve. Comme la mesure $\nu_{I_{0}}$ sur $\T^d$ est finie, $\mu$-stationnaire, sans atome, et comme $\mu$ est à support fini engendrant un semi-groupe $\Gamma := \langle \text{supp}\, \mu \rangle$ fortement irréductible, on a via Benoist-Quint \cite{BQI} que les mesures limites $(\nu_{I_{0},b})_{b \in B}$ sont sans atome. Or pour $\beta$-presque tout $b \in B$, on a $\nu_{b,I_{0}}= \nu_{I_{0},b}$. Cela génère une contradiction car on était parti de l'hypothèse qu'une proportion non nulle de mesures $(\nu_{b,I_{0}})_{b \in B}$ avait au moins un atome, donc conclut la preuve du \cref{THnondeg}. 
\end{proof}

%%%%%%%%%%%%%%%%%%%%%%%%%%%%%%%%%%%%%%%%%%%%%%%%%%%%%%%%%%%%%%%%%%%%%%%%%%%%%%%%%%%%%%%%%%%%%%%%%%%%%%%%%%%%%%%%%%%%%%%%%%%%%%%%%%%%%%%%%%%%%%%%%%%%%%%%%%%%%%%%%%%%%%%%%%%%%%%%%%%%%%%%%%%%%%%%%%%%%%%%%%%%%%%%%%%%%%%%%%%%%%%%%%%%%%%%%%%%%%%%%%%%%%%%%%%%%%%%%%%%%%%%%%%%%%%%%%%%%%%%%%%%%%%%%%%%%%%%%%%%%%%%%%%
%%%%%%%%%%%%%%%%%%%%%%%%%%%%%%%%%%%%%%%%%%%%%%%%%%%%%%%%%%%%%%%%%%%%%%%%%%%%%%%%%%%%%%%%%%%%%%%%%%%%%%%%%%%%%%%%%%%%%%%%%%%%%%%%%%%%%%%%%%%%%%%%%%%%%%%%%%%%%%%%%%%%%%%%%%%%%%%%%%%%%%%%%%%%%%%%%%%%%%%%%%%%%%%%%%%%%%%%%%%%%%%%%%%%%%%%%%%%%%%%%%%%%%%%%%%%%%%%%%%%%%%%%%%%%%%%%%%%%%%%%%%%%%%%%%%%%%%%%%%%%%%%%%%
\newpage

\section{Description des fibres}

L'argument de dérive exponentielle repose sur une description précise des fibres du système dynamique $(B^+, T^+)$ introduit dans la section $4$. Rappelons brièvement sa définition. On s'est fixé une probabilité  $\mu \in \mathcal{P}(SL_{d}(\Z))$ à support fini, on a noté $\Gamma := \langle \supp \mu \rangle$ le  semi-groupe engendré par son support, et on s'est donné $\chi : \Gamma \rightarrow \R$ un morphisme de semi-groupes. Cela définit une action de $\Gamma$ sur $X:= \T^d \times \R$ via $g.(x,t):=(gx, t+ \chi(g))$, puis une marche aléatoire sur $X$ de probabilités de transitions $(\mu \star \delta_{(x,t)})_{(x,t) \in X}$. On a noté  $B = \Gamma^{\N^\star}$, $\beta:= \mu^{\otimes \N^\star}$,  $T : B \rightarrow B, (b_{i})_{i\geq 1} \mapsto (b_{i+1})_{i \geq 1}$ le shift. On a introduit $G := \widebar{\Gamma}^Z \subseteq SL_{d}(\R)$ l'adhérence de Zariski de $\Gamma$ puis des objets algrébiques $K$, $\mathfrak{a}$, $\Sigma$, $\Pi$, $\mathscr{P}$, etc. permettant d'exploiter la nature algébrique de $G$, ainsi que des applications $\theta : B \rightarrow \ag^F$, $O : B \rightarrow O_{r}(\R)$. On s'est enfin donné une mesure de Radon $\nu \in \mathcal{M}^{Rad}(X)$ $\mu$-stationnaire, et une décomposition $(\nu_{b})_{b \in B}$ de $\nu$ en mesures limites.  Tout cela a été possible sous l'hypothèse que $\Gamma$ est fortement  irréductible et $\chi_{\star}\mu \in \mathcal{P}(\R)$ est centrée.  
\bigskip

\emph{Dans cette sous-section, on n'utilisera pas les définitions précises des applications $\theta$ et $O$, ni les hypothèses faites sur $\mu$.}

\bigskip

Le système dynamique $(B^+, \beta^+, T^+)$ que nous considérons est donné par :

 \begin{itemize}
 
 \item $B^+ = B \times \ag^F \times O_{r}(\R) \times X$
 \item $T^+ : B^+ \rightarrow B^+, (b,z,O, x) \mapsto (Tb, z- \theta(b), OO(b)^{-1} , b^{-1}_{1}x)$.
 \item $\beta^+ := \int_{B} \delta_{(b,z,O)} \otimes \nu_{b} \,\, d\beta(b)d\text{leb}_{\ag^F}(z) dh(O)\,\,\,\,$ 
\end{itemize}
où $\text{leb}_{\ag^F}$  et $h$ sont des mesures de Haar  sur $\ag^F$ et $O_{r}(\R)$, fixées une fois pour toute. On notera $\mathcal{B}^+$ la tribu produit sur $B^+$.

\bigskip

Dans un premier temps, on donne une description précise des fibres de $(B^+, \beta^+, T^+)$ et de la désintégration de $\beta^+$ le long de celles ci. On voit ensuite pourquoi les fibres ne peuvent pas s'accumuler dans des parties de petite mesure (ce qui est un phénomène général).

\subsection{Mesures conditionnelles le long des fibres}

Soit $c \in B^+$, $n \geq 0$. On appelle $n$-fibre passant par $c$ l'ensemble $$F_{n,c}:=\{c' \in B^+, (T^+)^n(c')=(T^+)^n(c)\}$$

Le lemme suivant affirme qu'on peut paramétrer $F_{n,c}$ par $\Gamma^n$. On note $$\theta_{n}(b):= \sum_{k=0}^{n-1} \theta(T^{k}b) \,\,\,\,\,\,\,\,\,\,\,\,\,\,\,\,\,\,O_{n}(b):= O(T^{n-1}b)O(T^{n-2}b)\dots O(b)$$

\begin{lemme} \label{parfib}
Pour $c =(b,z,O,x) \in B^{f}$, $n \geq 0$, définissons 
$$h_{n,c} : \Gamma^n \rightarrow B^+, \,\,a \mapsto (aT^nb, \, z  - \theta_{n}(b) + \theta_{n}(aT^nb), OO_{n}(b)^{-1}O_{n}(aT^nb),  \, a_{1}\dots a_{n}b^{-1}_{n}\dots b^{-1}_{1}.x )$$
où on a noté $a=(a_{1}, \dots, a_{n})$ et $aT^nb=(a_{1}, \dots, a_{n},b_{n+1}, b_{n+2}, \dots)$. 

Alors $h_{n,c}$ est une bijection entre $G^n$ et $F_{n,c}$.
\end{lemme}

\begin{proof}
C'est un calcul direct.
\end{proof}

Désintégrons maintenant $\beta^+$ le long des fibres. Pour $n \geq0$,  on note $\mathcal{Q}^+_{n}:=(T^+)^{-n}(\mathcal{B}^+)$.  C'est une sous tribu de $\mathcal{B}^+$ dont l'atome contenant un point $c \in B^+$ est exactement $F_{n,c}$. Comme la mesure restreinte $\beta^+_{|\mathcal{Q}^+_{n}}$ est $\sigma$-finie, on peut se donner une décomposition de $\beta^+$ en mesures conditionnelles par rapport à la sous tribu $\mathcal{Q}^+_{n}$, que l'on note $(\beta^+_{n,c})_{c \in B^+}$. Rappelons que la famille $(\beta^+_{n,c})_{c \in B^+}$ est une famille $\mathcal{Q}^+_{n}$-mesurable de mesures sur $B^+$ telle que $\beta^+ = \int_{B^+} \beta^+_{n,c} \, d\beta^+(c)$ et vérifiant pour $\beta^+$-presque tout $c$ que $\beta^+_{n,c}$ est concentrée sur $F_{n,c}$. De plus, une telle famille est unique (en dehors d'un ensemble $\beta^+$-négligeable) et liée à la notion d'espérance conditionnelle : si $\varphi : B^+ \rightarrow [0, +\infty]$ est une fonction positive $\mathcal{B}^+$-mesurable, alors pour $\beta^+$-presque tout $c \in B^+$, on a $\mathbb{E}(\varphi |  \mathcal{Q}^+_{n})(c)= \beta^+_{n,c}(\varphi)$. 

\bigskip

Le \cref{parfib} identifie $F_{n,c}$ à $\Gamma^n$. La mesure conditionnelle $\beta^+_{n,c} \in \mathcal{M}(F_{n,c})$ correspond alors à la probabilité $\mu^{\otimes n}$ :  

\begin{lemme} \label{mesfib}
Pour $\beta^+$-presque tout $c \in B^+$, on a $\beta^+_{n,c} = (h_{n,c})_{ \star} \mu^{\otimes n}$

\end{lemme}

\begin{proof}
Il s'agit de vérifier que les $(h_{n,c})_{ \star} \mu^{\otimes n}$ vérifient les propriétés caractérisant les mesures conditionnelles $\beta^+_{n,c}$. L'application $h_{n,c}$ ne dépend que de $(T^+)^n(c)$ donc la famille $(h_{n,c})_{ \star} \mu^{\otimes n}$ est $\mathcal{Q}^+_{n,c}$-mesurable. Le lemme précédent garantit que ces mesures sont concentrées sur les $n$-fibres. Pour conclure, on doit vérifier que  $\int_{B^+} (h_{n,c})_{ \star} \mu^{\otimes n} \,d\beta^+(c)= \beta^+$.  Soit $\varphi : B^+ \rightarrow [0, +\infty]$ une application mesurable. En notant $c=(b,z,O,x) \in B^+$,  en utilisant le théorème de Fubini, l'invariance des mesures Haar et l'équivariance des mesures limites, on obtient :
\begin{align*}
&\,\,\,\,\,\,\,\,\int_{B^+} (h_{n,c})_{ \star} \mu^{\otimes n}(\varphi) d\beta^+(c)\\
&=  \int_{B^+} \int_{\Gamma^n}
\varphi(aT^nb, \, z  - \theta_{n}(b) + \theta_{n}(aT^nb), OO_{n}(b)^{-1}O_{n}(aT^nb),  \, a_{1}\dots a_{n}b^{-1}_{n}\dots b^{-1}_{1}x) d\mu^{\otimes n}(a) d\beta^+(c)\\
& = \int_{B^+} \int_{\Gamma^n} \varphi (aT^nb, \,z ,\, O, \, a_{1}\dots a_{n}b^{-1}_{n}\dots b^{-1}_{1}x)\,\,d\mu^{\otimes n}(a) d\beta^+(c)\\
& = \int_{B \times \ag^F \times O_{r}(\R)}\int_{\Gamma^n} \int_{X}  \varphi (aT^nb, \,z ,\, O, \, a_{1}\dots a_{n}b^{-1}_{n}\dots b^{-1}_{1}x)\,\,d\nu_{b}(x) d\mu^{\otimes n}(a) d\beta(b)d \text{leb}_{\ag^F}(z)dh(O)\\
& = \int_{B \times \ag^F \times O_{r}(\R)}\int_{\Gamma^n} \int_{X}  \varphi (aT^nb, \,z ,\, O, \,x)\,\, d\nu_{aT^nb}(x) d\mu^{\otimes n}(a) d\beta(b)d \text{leb}_{\ag^F}(z)dh(O)\\
& = \int_{B \times \ag^F \times O_{r}(\R)} \int_{X}  \varphi (aT^nb, \,z ,\, O, \,x)\,\, d\nu_{b}(x) d\beta(b)d \text{leb}_{\ag^F}(z)dh(O)\\
&= \beta^+(\varphi)
\end{align*}

\end{proof}

La preuve de la dérive exponentielle se fera en restriction à une fenêtre $W \subseteq B^+$ (formellement une partie $\mathcal{B}^+$-mesurable de mesure finie non nulle). On termine ce paragraphe en  désintégrant la mesure $\beta^+_{|W}$ restreinte à la fenêtre par rapport aux morceaux de fibres $F_{n,c}  \cap W$. Pour cela, notons $\beta^+_{W}:= \beta^+_{|W}$, $\mathcal{Q}^+_{W, n}:= (\mathcal{Q}^+_{n})_{|W}$ la tribu trace de $\mathcal{Q}^+_{n}$ sur $W$, et $(\beta^+_{W,n,c})_{c \in W}$ une décomposition de $\beta^+_{W}$ en mesures conditionnelles par rapport à la sous tribu $\mathcal{Q}^+_{W, n}$.  On pose par ailleurs $Q_{n,c}:= \{a \in \Gamma^n, h_{n,c}(a) \in W\}$. 

\bigskip

Le morceau de fibre $F_{n,c} \cap W$ s'identifie à $Q_{n,c}$. Le lemme suivant affirme qu'alors la mesure conditionnelle $\beta^+_{W,n,c} \in \mathcal{M}(F_{n,c}\cap W)$ correspond à la mesure $\mu^{\otimes n}_{|Q_{n,c}}$ normalisée. 
\begin{lemme}
Soit $Q_{n,c}:= \{a \in \Gamma^n, h_{n,c}(a) \in W\}$. Pour $\beta_{W}^+$-presque tout $c \in W$, on a $\mu^{\otimes n}(Q_{n,c})>0$ et $$\beta^+_{W, n,c} = (h_{n,c})_{ \star} \frac{1}{\mu^{\otimes n}(Q_{n,c})}   \mu^{\otimes n}_{|Q_{n,c}}$$
\end{lemme}

\begin{proof}
Au vu du \cref{mesfib}, il s'agit de montrer que pour $\beta^+_{W}$-presque tout $c \in W$, on a $\beta^+_{n, c}(W)>0$ et $$\beta^+_{W, n,c} = \frac{1}{\beta^+_{n, c}(W)}   (\beta^+_{n, c})_{|W}$$

Vérifions la première assertion. On pose $E:= \{c \in B^+, \beta^+_{n, c}(W)=0\,\}$. Alors $0= \int_{E}\beta^+_{n, c}(W) d\beta^+(c)=\int_{E} \mathbb{E}_{\beta^+}(1_{W}| \mathcal{Q}^+_{n})(c) d\beta^+(c)= \int_{E}1_{W}(c) d\beta^+(c)$ car $E$ est $\mathcal{Q}^+_{n}$-mesurable. Ainsi, $\beta^+(E\cap W)=0$. 

Vérifions l'expression de $\beta^+_{W, n,c}$. Il s'agit de voir que la famille $ (\frac{1}{\beta^+_{n, c}(W)}   (\beta^+_{n, c})_{|W})_{c \in W}$ (définie pp) vérifie les propriétés caractérisant les mesures conditionnelles. Il est immédiat qu'elle est $\mathcal{Q}^+_{W, n}$-mesurable, et que pour presque tout $c \in W$, on a $\frac{1}{\beta^+_{n, c}(W)}   (\beta^+_{n, c})_{|W}$ concentré sur l'atome de $\mathcal{Q}^+_{W, n}$ contenant $c$. Par ailleurs, si $\varphi : W \rightarrow [0,+\infty]$ est une fonction positive $\mathcal{B}^+_{|W}$-mesurable, on a 
\begin{align*}
\int_{W} \frac{ (\beta^+_{n, c})_{|W}(\varphi)}{\beta^+_{n, c}(W)} d\beta^+(c) &=\int_{B^+} 1_{W}(c) \frac{ \mathbb{E}_{\beta^+}(\varphi| \mathcal{Q}^+_{n})(c)}{\mathbb{E}_{\beta^+}(1_{W}| \mathcal{Q}^+_{n})(c)} d\beta^+(c) \\
&=  \int_{B^+} \mathbb{E}_{\beta^+}(1_{W}| \mathcal{Q}^+_{n})(c) \frac{ \mathbb{E}_{\beta^+}(\varphi| \mathcal{Q}^+_{n})(c)}{\mathbb{E}_{\beta^+}(1_{W}| \mathcal{Q}^+_{n})(c)} d\beta^+(c) \\
&=  \int_{B^+} \mathbb{E}_{\beta^+}(\varphi| \mathcal{Q}^+_{n})(c) d\beta^+(c) \\
&= \beta^+(\varphi)\\
\end{align*}
On a donc l'expression annoncée pour les mesures conditionnelles restreintes. 
\end{proof}

\subsection{Equirépartition des morceaux de fibres}

Dans cette partie, on montre que les morceaux de fibres $F_{n,c} \cap W$ du paragraphe précédent ne peuvent pas s'accumuler dans des parties de $W$ de mesure trop petite. C'est en fait un résultat général que nous allons expliquer dans un cadre abstrait. Soit $(E, \mathcal{E})$ un espace mesurable, $m\in \mathcal{M}^\sigma(E, \mathcal{E})$ une mesure $\sigma$-finie sur $(E, \mathcal{E})$, et $T: E \rightarrow E$ une application mesurable telle que $T_{\star}m=m$.

%Soit $(E, \mathcal{E}, m, T)$ un système dynamique mesuré $\sigma$-fini. 

\subsubsection*{Cas probabilisé : équirépartition des fibres}

Supposons $m(E)=1$. Pour $c \in E$, $n \geq0$, on note  toujours  $F_{n,c} := \{c' \in E, \,\,T^n(c') = T^n(c)\}$ la $n$-fibre passant par $c$. 

On veut comprendre comment se répartissent les points de $F_{n,c}$ dans $E$ pour $n$ grand et $c$ fixé. Pour cela, on introduit $\mathcal{Q}_{n} := T^{-n}(\mathcal{E})$ sous tribu de $\mathcal{E}$, dont l'atome contenant $c$ est exactement $F_{n,c}$. Alors, si on se donne une partie mesurable $A \in \mathcal{E}$, l'espérance conditionnelle $\mathbb{E}_{m}(1_{A} | \mathcal{Q}_{n})(c)$ décrit la proportion selon $m$ d'éléments de $F_{n,c}$ se trouvant dans $A$. On énonce le théorème d'équirépartition des fibres de la fa\c con suivante : on note $\mathcal{Q}_{\infty} := \cap_{n \geq 0} \mathcal{Q}_{n}$. 

\begin{th.}
Soit $\psi \in L^1(E, \mathcal{E}, m)$. On a la convergence presque sûre :

$$\mathbb{E}_{m}(\psi | \mathcal{Q}_{n}) \rightarrow \mathbb{E}_{m}(\psi | \mathcal{Q}_{\infty}) $$
\end{th.}

\begin{proof}
La suite de sous-tribus $(\mathcal{Q}_{n})_{n \geq 0}$ est de plus en plus grossière. On applique un théorème de convergence des martingales (par rapport à des filtrations décroissantes). 
\end{proof}

\subsubsection*{ En restriction à une fenêtre : équirépartition des morceaux de fibres}

Nous serons amenés à étudier les fibres d'un système dynamique en mesure infinie, et plus précisément la fa\c con dont ces fibres se répartissent dans une fenêtre de mesure finie. Voyons que les résultats du cas probabilisé se transposent à ce cadre.

Reprenons les notations précédentes, mais avec $m$ seulement supposée $\sigma$-finie. On se donne $W \subseteq E$ une partie mesurable de mesure finie non nulle.

On a alors un espace mesuré restreint $(W, \mathcal{E}_{|W}, m_{|W})$ et des sous-tribus traces $\mathcal{Q}_{W,n} :={\mathcal{Q}_{n}}_{|W} $, $\mathcal{Q}_{W,\infty} := \bigcap_{n \geq 0}\mathcal{Q}_{W,n}$.  Alors, si on se donne une partie mesurable $A \in \mathcal{E}_{|W}$, l'espérance conditionnelle $\mathbb{E}_{m_{|W}}(1_{A} | \mathcal{Q}_{W,n})(c)$ décrit la proportion d'éléments de $F_{n,c}$ se trouvant dans $A$ parmi ceux qui se trouvent dans la fenêtre $W$. On énonce le théorème d'équirépartition des morceaux de fibres de la fa\c con suivante :

\begin{th.}
Soit $\psi \in L^1(W, \mathcal{E}_{|W}, m_{|W})$. On a la convergence presque sûre :

$$\mathbb{E}_{m_{|W}}(\psi | \mathcal{Q}_{W,n}) \rightarrow \mathbb{E}_{m_{|W}}(\psi | \mathcal{Q}_{W, \infty}) $$
\end{th.}

\begin{proof}
C'est la même démonstration que dans le cas des fibres.
\end{proof}

Expliquons enfin  pourquoi ce théorème affirme que les morceaux de fibres ne s'accumulent pas dans  des parties de $W$ de mesure trop petite. 

\begin{cor}\label{nonacc}
Soit $\varepsilon>0$, $N \subseteq W$ une partie mesurable telle que $m(N)< \frac{\varepsilon^2}{2}$. Alors il existe $K \subseteq W$ mesurable vérifiant $m(W-K)<\varepsilon$ et un rang $n_{0}\geq0$ tels que pour tout $c \in K$, pour tout $n \geq n_{0}$, on a 
$$\mathbb{E}_{m_{|W}}(1_{N}| \mathcal{Q}_{W,n})(c) < \varepsilon$$
\end{cor}

\begin{rem.}
L'énoncé du corollaire suppose qu'on a fixé des versions des espérances conditionnelles $\mathbb{E}(1_{N}| \mathcal{Q}_{W,n})$ dans $\mathscr{L}^1_{m_{|W}}(W, \mathcal{Q}_{W,n})$.
\end{rem.}

\begin{proof}
Notons $\psi= \mathbb{E}_{m_{|W}}(1_{N}| \mathcal{Q}_{W,\infty}) \in \mathscr{L}^1_{m_{|W}}(W, \mathcal{Q}_{W,\infty})$ une version de l'espérance conditionnelle de $1_{N}$ par rapport à $\mathcal{Q}_{W,\infty}$. C'est une fonction non négative telle que $\int_{W} \psi \,\,dm_{|W} < \frac{\varepsilon^2}{2}$. On en déduit que  $m_{|W}\{\psi \geq \frac{\varepsilon}{2}\} < \varepsilon$. Or d'après le théorème d'équirépartition des morceaux de fibres, on a la convergence presque sûre $\mathbb{E}_{m_{|W}}(1_{N}| \mathcal{Q}_{W,n}) \rightarrow \psi$. De plus, le théorème d'Egoroff donne l'existence d'une partie mesurable $K \subseteq \{\psi < \frac{\varepsilon}{2}\}$ telle que $m(W-K) <\varepsilon$ et sur laquelle la convergence est uniforme.  Il existe donc un rang $n_{0} \geq 0$ à partir duquel pour tout $c \in K$, on a $|\mathbb{E}_{m_{|W}}(1_{N}| \mathcal{Q}_{W,n})(c) - \psi(c)| < \varepsilon/2$ puis $\mathbb{E}_{m_{|W}}(1_{N}| \mathcal{Q}_{W,n})(c) < \varepsilon$.

\end{proof}

%%%%%%%%%%%%%%%%%%%%%%%%%%%%%%%%%%%%%%%%%%%%%%%%%%%%%%%%%%%%%%%%%%%%%%%%%%%%%%%%%%%%%%%%%%%%%%%%%%%%%%%%%%%%%%%%%%%%%%%%%%%%%%%%%%%%%%%%%%%%%%%%%%%%%%%%%%%%%%%%%%%%%%%%%%%%%%%%%%%%%%%%%%%%%%%%%%%%%%%%%%%%%%%%%%%%%%%%%%%%%%%%%%%%%%%%%%%%%%%%%%%%%%%%%%%%%%%%%%%%%%%%%%%%%%%%%%%%%%%%%%%%%%%%%%%%%%%%%%%%%%%%%%%
%%%%%%%%%%%%%%%%%%%%%%%%%%%%%%%%%%%%%%%%%%%%%%%%%%%%%%%%%%%%%%%%%%%%%%%%%%%%%%%%%%%%%%%%%%%%%%%%%%%%%%%%%%%%%%%%%%%%%%%%%%%%%%%%%%%%%%%%%%%%%%%%%%%%%%%%%%%%%%%%%%%%%%%%%%%%%%%%%%%%%%%%%%%%%%%%%%%%%%%%%%%%%%%%%%%%%%%%%%%%%%%%%%%%%%%%%%%%%%%%%%%%%%%%%%%%%%%%%%%%%%%%%%%%%%%%%%%%%%%%%%%%%%%%%%%%%%%%%%%%%%%%%%%

\newpage

\section{Théorème local limite}

Cette partie prépare la preuve de la loi des angles donnée dans la section suivante. On y prouve un théorème local limite qui permettra d'estimer selon quelle proportion une fibre du système dynamique $(B^+, T^+, \beta^+)$ rencontre une fenêtre $W \subseteq B^+$ fixée au préalable.

\subsection{Rappels}
On rappelle un théorème local limite avec déviations modérées démontré dans \cite{BQRW}. Les notations de cette sous-section sont indépendantes de ce qui précède. On pourra se référer à \cite{BQRW} (sections $11$, $15$, $16$) pour un exposé détaillé. Soit $X$ un espace métrique compact,  $\Gamma$ un semi-groupe localement compact à base dénombrable agissant continument sur $X$, $F$ un groupe fini, $s : \Gamma \rightarrow F$ un morphisme de groupes, et $f: X \rightarrow F,\, x \mapsto f_{x}$ une application continue $\Gamma$-équivariante. 

Plus tard, on spécifiera $X= G/P_{c}$ variété drapeau d'un groupe algébrique réel semi-simple,  $\Gamma \subseteq G$ sous groupe discret Zariski-dense agissant sur $X$ par multiplication à gauche, $F= G/G_{c}$, $s$, $f$ les projections canoniques sur  $F$. 

On se donne une probabilité $\mu \in \mathcal{P}(\Gamma)$ sur $\Gamma$, apériodique dans $F$ (autrement dit,  le support de $s_{\star} \mu$ egendre $F$ et n'est contenu dans aucun translaté de sous-groupe cocylique de $F$, hormis $F$ tout entier). On suppose de plus que l'action de $\Gamma$ sur $X$ est $\mu$-contractante au dessus de $F$ (i.e. que l'itération de la $\mu$-marche sur $X$ contracte la distance sur chaque fibre de $f$, cf. \cite{BQRW}, section $11.1$ pour une définition précise). Dans ce cas, $X$ admet une unique probabilité $\mu$-stationnaire que l'on note $\nu_{0}$ (cf. \cite{BQRW}, lemme $11.5$)

On se donne $\sigma : \Gamma\times X \rightarrow \mathbb{R}^d$ un cocycle continu tel que les fonctions $\sigma_{sup} : \Gamma \rightarrow [0, +\infty], g \mapsto \sup_{x \in X}||\sigma(g,x)||$ et $\sigma_{lip} : \Gamma \rightarrow [0, +\infty], \,g \mapsto \underset{  f_{x}=f_{x'}}{\sup} \frac{||\sigma(g,x) - \sigma(g,x')||}{d(x,x')}$ sont $\mu$-presque sûrement bornées.

 D'après (\cite{BQRW}, 3.4.1), le cocycle $\sigma$ est cohomologue à un cocycle $\sigma_{0} : \Gamma\times X \rightarrow \mathbb{R}^d$ dont la fonction $\sigma_{0,\sup}$ est $\mu$-presque sûrement bornée et tel que l'application de dérive  $X \rightarrow \R^d, x\mapsto \int_{\Gamma}\sigma_{0}(g,x)d\mu(g)$ est constante (on dit que $\sigma$ est spécial). On notera 
 $$\sigma_{\mu} := \int_{\Gamma\times X}\sigma(g,x)\,\,d\mu \otimes \nu_{0}(g,x) \in \R^d$$ 
 cette constante et 
 $$E_{\mu} :=\text{Vect}_{\R}\{\sigma_{0}(g,x) - \sigma_{\mu}, \,\,g \in \text{supp}\,\mu, \,x\in \text{supp}\, \nu_{0}\}$$

  On a par ailleurs (via \cite{BQRW}, corollary 15.10) l'existence d'un plus petit sous groupe fermé de $\R^d$, noté $\Delta_{\mu}$, tel qu'il existe une fonction continue $\varphi : \supp \nu_{0} \rightarrow \R^d/\Delta_{\mu}$  et une constante $v \in \R^d/\Delta_{\mu}$ pour lesquels, pour tout $g \in \supp \mu$, $x \in \supp \nu_{0}$ on a $$\sigma(g,x) \mod \Delta_{\mu} = v + \varphi(x) - \varphi(gx)$$. 

Remarquons qu'à priori,  $\varphi$ n'admet pas de relevé à $\R^d$. Le sous groupe $\Delta_{\mu}$ est appelé \emph{image résiduelle} du cocycle $\sigma$. On a de plus $\Delta_{\mu} \subseteq E_{\mu}$ cocompact (\cite{BQRW}, proposition $15.8$).

On énonce maintenant le théorème local limite sous les hypothèse où $E_{\mu}= \R^d$ et il existe une fonction continue $\widetilde{\varphi}_{0} : \supp \nu_{0} \rightarrow \R^d$ telle que pour tout $g \in \supp \mu$, $x \in \supp \nu_{0}$ on a $$\sigma(g,x) \mod \Delta_{\mu} = \widetilde{\varphi}_{0}(x) - \widetilde{\varphi}_{0}(gx) \mod \Delta_{\mu}$$
La première hypothèse garantit que la matrice de covariance $\Phi_{\mu} := \text{cov}_{\mu \otimes \nu_{0}}(\sigma_{0}) \in M_{d}(\R)$ est symétrique définie positive, et induit donc un produit scalaire sur $\R^d$. C'est en particulier un produit scalaire sur la composante neutre $\Delta_{\mu,0}$ de $\Delta_{\mu}$ (qui est un sous-espace vectoriel de $\R^d$). On note $\pi_{\mu}$ la mesure de Haar sur $\Delta_{\mu}$ qui donne le poids $1$ aux cubes unités de $\Phi_{\mu |\Delta_{\mu,0}}$. Dans le cas où $\Delta_{\mu}$ est discret, il s'agit de la mesure de comptage. La mesure intervenant dans le théorème local limite est une moyenne de translatés de $\pi_{\mu}$, définie pour $C \subseteq \R^d$ mesurable, par :
$$  \Pi_{\mu}(C):= \int_{X} \pi_{\mu}(C + \widetilde{\varphi}_{0}(x)) \,d\nu_{0}(x)$$

On peut enfin énoncer un théorème local limite avec déviations modérées. Une démonstration se trouve dans \cite{BQRW} (theorem $16.1$).

\begin{th.}\label{TLLBQ}
Fixons une partie convexe bornée $C \subseteq \R^d$ et  une constante $R>0$. On se donne un point $x \in \supp\nu_{0}$, une suite de vecteurs $(v_{n}) \in (\R^d)^{\N}$  telle que $||v_{n}||_{\Phi_{\mu}} \leq R\sqrt{n \log n}$ pour tout $n \geq 1$. Alors on a la convergence :

$$(2\pi n)^{\frac{d}{2}} e^{\frac{1}{2n} ||v_{n} ||^2_{\Phi_{\mu}}} \mu^{\star n}(g \in \Gamma, \,\,\sigma(g,x) - n \sigma_{\mu} \in C+ v_{n})-   \Pi_{\mu}(C+v_{n}+n \sigma_{\mu} - \widetilde{\varphi}_{0}(x)) \underset{n \to +\infty}{\rightarrow} 0$$ 

De plus, cette convergence est uniforme en le point $x \in \text{supp}\,\nu_{0}$ et en la suite $(v_{n})$ satisfaisant l'hypothèse de domination (avec $R$ fixé au préalable.) 

\end{th.}

\bigskip

\subsection{Taille des morceaux de fibres}

On revient au cadre de la marche sur $\T^d \times \R$. On  se donne  $\mu \in \mathcal{P}(SL_{d}(\Z))$  une probabilité  à support fini engendrant un sous semi-groupe $\Gamma \subseteq SL_{d}(\Z)$ dont l'adhérence de Zariski $G:= \widebar{\Gamma}^Z \subseteq SL_{d}(\R)$ est semi-simple, et $\chi : \Gamma \rightarrow \R$  un morphisme de semi-groupes tel que la probabilité image $\chi_{\star}\mu \in \mathcal{P}(\R)$ est centrée. Pour prouver la dérive exponentielle,  on pourra toujours se ramener au cas où $\mu(e)>0$ (qui implique que $\mu$ est apériodique dans $F:= G/G_{c}$) et $\chi \nequiv 0$. On fera donc ces hypothèses, qui simplifient les énoncés et les preuves dans la suite. 

\bigskip

On reprend les notations de la section $4$. Nous allons  appliquer les rappels précédents à la variété drapeau $\PP:=G/P_{c}$ munie de l'action de $\Gamma$ par translation, les morphismes $s : \Gamma  \rightarrow F$, $f : \PP \rightarrow F$ désignant les projections canoniques dans $F$. L'application $f$ est  bien continue $\Gamma$-équivariante.  

 L'action de $\Gamma$ sur $\PP$ est $\mu$-contractante au dessus de $F$ (cf. \cite{BQRW}, $13.2$), on note $\nu_{\PP} \in \mathcal{P}(\PP)$ l'unique probabilité $\mu$-stationnaire sur $\PP$.    
Soit $\ag^F:= \{x \in \ag, \,\,F.x=x\}$ le sous espace vectoriel de $\ag$ fixé par $F$,  $\sigma : G\times \PP \rightarrow \ag$ le cocycle d'Iwasawa, $\sigma_{F}:= \frac{1}{|F|}\Sigma_{f \in F}f.\sigma : G\times \PP \rightarrow \ag^F$ son projeté sur $\ag^F$.  On introduit 
$$\sigma_{F,\chi} : \Gamma\times \PP \rightarrow \ag^F \times \R, \,(g,x) \mapsto(\sigma_{F}(g,x), \chi(g))$$
 Alors $\sigma_{F,\chi}$ est un cocycle continu dont les fonctions $(\sigma_{F,\chi})_{\sup}$ et $(\sigma_{F,\chi})_{lip}$ sont $\mu$-presque sûrement bornées (car c'est le cas pour $\sigma$, cf \cite{BQRW}, $13.1$). De plus,  comme le vecteur de Lyapunov $\sigma_{\mu}:= \mu \otimes \nu_{\PP}( \sigma) \in \ag^F$ est $F$-invariant et comme la probabilité $\chi_{\star} \mu$ est centrée,  la moyenne $(\sigma_{F,\chi})_{\mu} := \mu \otimes \nu_{\PP}( \sigma_{F,\chi})$ de $\sigma_{F,\chi}$ est donnée par  $(\sigma_{F,\chi})_{\mu} = (\sigma_{\mu}, 0)$.

\bigskip

On va prouver un théorème local limite pour $\sigma_{F, \chi}$. La première étape sera de démontrer le lemme suivant :

\begin{lemme} \label{EDelta}
Pour le cocycle $\sigma_{F,\chi}: G \times \PP \rightarrow \ag^F \times \R$, nous avons :
\begin{itemize}
\item $E_{\mu}= \ag^F \times\R$.
\item Si $\widebar{\chi(\Gamma)}=\R$, alors $\Delta_{\mu}=\ag^F \times\R$. 

\item Si $\chi(\Gamma) \subseteq \R$ est discret, alors $\Delta_{\mu}=\ag^F \times k \chi(\Gamma)$ pour un certain $k \geq 1$. 

\end{itemize}
\end{lemme}

\bigskip

 Notons $l := \dim_{\R} \ag^F$, fixons une fois pour toute $\Pi_{\mu} \in \mathcal{M}^{Rad}(\ag^F \times \R)$ une mesure de Haar sur $\ag^F \times \widebar{\chi(\Gamma)}$, ainsi qu'une norme sur $\ag$,. Nous déduirons du lemme précédent  que quitte à normaliser $\Pi_{\mu}$, on a un théorème local limite avec déviations modérées pour le cocycle $\sigma_{F, \chi}$ :
 
 \bigskip
 
\begin{th.}[TLL pour $\T^d \times \R$ ] \label{TLL}

Fixons des parties convexes bornées $U \subseteq \ag^F$, $I \subseteq \R$  et une constante $R>0$. On se donne un point $x \in \text{supp}\, \nu_{\PP}$, et des suites $(u_{n}) \in (\ag^F)^\N$, $(t_{n}) \in \R^{\N}$  telles que $||a_{n}||, |t_{n}| \leq R\sqrt{n \log n}$ pour tout $n \geq 1$. Alors on a la convergence :

$$(2\pi n)^{\frac{l+1}{2}} e^{\frac{1}{2n} ||(u_{n},t_{n}) ||^2_{\Phi_{\mu}}} \mu^{\star n}(g \in \Gamma, \,\, \sigma_{F}(g,x) - n \sigma_{\mu} \in U+u_{n},\,\,\,\chi(g) \in I + t_{n}) - \Pi_{\mu}(U \times I+t_{n}) \underset{n \to +\infty}{\rightarrow} 0$$ 

De plus, cette convergence est uniforme en le point $x \in \text{supp}\, \nu_{\PP}$ et en les suites $(u_{n}), (t_{n})$ satisfaisant l'hypothèse de domination (avec $R$ fixé au préalable.) 

\end{th.}

\newpage

\subsubsection*{Preuve du \cref{EDelta}}

On montre d'abord que l'hypothèse  $\mu(e)>0$ simplifie la caractérisation de $\Delta_{\mu}$. C'est un phénomène général, on reprend donc momentanément les notations de $7.1$. 

\begin{lemme}[``$v=0$, $g \in \Gamma$'']\label{redDelta}
Supposons que $\supp \mu$ engendre $\Gamma$ comme semi-groupe, et $\mu(e)>0$. Alors $\Delta_{\mu}$ est le plus petit sous-groupe fermé $\Delta \subseteq  \R^d$ tel qu'il existe une application continue $\varphi : \supp \nu_{0} \rightarrow \R^d/\Delta$ vérifiant pour tout $g \in \Gamma$, $x \in \supp \nu_{0}$ : 
$$\sigma(g,x) \mod \Delta =  \varphi(x) - \varphi(gx)$$
 
\end{lemme}

\begin{proof}
Il est clair qu'un tel sous groupe $\Delta$ vérifie la propriété pour laquelle $\Delta_{\mu}$ est un  minorant global (définie dans le rappel). On a donc $\Delta \supseteq \Delta_{\mu}$. Par ailleurs, $\Delta_{\mu}$ lui même vérifie la propriété du lemme. En effet, par définition,  il existe une fonction continue $\varphi : \supp \nu_{0} \rightarrow \R^d/\Delta_{\mu}$  et une constante $v \in \R^d/\Delta_{\mu}$ pour lesquels, pour tout $g \in \supp \mu$, $x \in \supp \nu_{0}$ on a $$\sigma(g,x) \mod \Delta_{\mu} = v + \varphi(x) - \varphi(gx)$$
Appliqué en $g=e \in \supp \mu$, on obtient que $v=0$. On en déduit alors par la propriété de cocycle et le fait que $\supp \nu_{0}$ est $\Gamma$-invariant que $\sigma(g,x) \mod \Delta =  \varphi(x) - \varphi(gx)$ pour tout $g \in \Gamma$, $x \in \supp \nu_{0}$
\end{proof}

\bigskip

On revient au cadre  de la section $7.2$. En particulier, $\sigma$ désigne le cocycle d'Iwasawa. 
\begin{lemme}[Apériodicité d'un projeté du cocycle d'Iwasawa] \label{apsigmap}

Soit $\mathfrak{b} \subseteq \ag$ un sous-espace vectoriel et $p : \ag \rightarrow \mathfrak{b}$ un projecteur sur $\mathfrak{b}$. Alors pour le cocycle $p \circ \sigma : G \times \PP \rightarrow \mathfrak{b}$, nous avons $$\Delta_{\mu}(p \circ \sigma)= E_{\mu}(p \circ \sigma)= \mathfrak{b}$$ 

C'est en particulier le cas pour le cocycle projeté $\sigma_{F}$. 
\end{lemme}

\begin{proof}
Remarquons que $p \circ \sigma$ est bien un cocycle continu dont les fonctions $(p \circ \sigma)_{lip}$ et $(p \circ \sigma)_{\sup}$ sont $\mu$-presque sûrement bornées. Les notations  $\Delta_{\mu}(p \circ \sigma)$ et $E_{\mu}(p \circ \sigma)$ ont donc bien un sens. Le reste est une conséquence de l'apériodicité du cocycle d'Iwasawa (voir \cite{BQII}, proposition $17.1$). En effet, soit $\Delta \subseteq \mathfrak{b}$  un sous groupe fermé et $\varphi : \supp \nu_{\PP} \rightarrow \mathfrak{b}/\Delta$  une fonction continue telle que pour tout $g \in \supp \Gamma$, $x \in \supp \nu_{\PP}$ on a 
$$p \circ \sigma(g,x) \mod \Delta = \varphi(x) - \varphi(gx)$$

Alors, pour tout $g \in \supp \Gamma$, $x \in \supp \nu_{\PP}$ on a 
$$ \sigma(g,x) \mod \text{Ker }p +\Delta = \varphi(x) - \varphi(gx)$$

où on voit $\varphi$ comme une fonction à valeurs dans $\ag/(\text{Ker }p + \Delta) \equiv \mathfrak{b}/\Delta$. 
On en déduit que $\text{Ker }p + \Delta \supseteq \Delta(\sigma)= \ag$ puis que $\Delta= \mathfrak{b}$. 

\end{proof}

\bigskip

\begin{lemme} \label{Delta}
On a l'inclusion : $\ag^F\times \{0\} \subseteq \Delta_{\mu} $
\end{lemme}

\begin{proof}

L'idée est de restreindre $\sigma_{F,\chi}$ en un cocycle pour lequel $\chi=0$ puis d'utiliser l'apériodicité d'un projeté du cocycle d'Iwasawa. La difficulté est que $\Gamma$ n'est pas un groupe à priori mais seulement un semi-groupe, si bien que l'ensemble d'annulation $\{\chi=0\}$ pourrait être trivial, et ce même en sachant que $\Gamma$ est Zariski-dense dans $G$ semi-simple (considérer par exemple : $G=SL_{2}(\R)$, $\Gamma \subseteq SL_{2}(\R)$ semi-groupe Zariski-dense, librement engendré par deux éléments $a_{0},a_{1}$ et définir $\chi$ tel que $\chi(a_{0})=1$, $\chi(a_{1})=-\sqrt{2}$).

Pour contourner cette difficulté, on note $\{g_{1},\dots, g_{s}\}:=\supp\mu $, on introduit $L:= \langle a_{1}, \dots, a_{s} \rangle$ le groupe libre à $s$ générateurs $a_{1}, \dots, a_{s}$, on  pose $\mu_{L} \in \mathcal{P}(L)$ la probabilité sur $L$ donnée par $\mu_{L}(a_{i})= \mu(g_{i})$ et on fait agir $L$ sur la variété drapeau $\F$ via le morphisme de groupes  $L \rightarrow G, a_{i} \mapsto g_{i}$. Comme les marches sur $\PP$ données par $(L, \mu_{L})$ et $(\Gamma, \mu)$ coïncident, on a en particulier que $\Delta_{\mu_{L}}= \Delta_{\mu}$. Dans la suite, on notera abusivement $\Gamma \subseteq L$ le sous semi-groupe de $L$ engendré par $a_{1}, \dots, a_{s}$,  $\langle \Gamma \rangle = L$ et $\chi : \langle \Gamma \rangle \rightarrow \R$ le morphisme défini par $\chi(a_{i})= \chi(g_{i})$.

D'après le  \cref{redDelta}, il existe une application continue $\varphi : \supp \nu_{\PP} \rightarrow (\ag^F\times \R)/\Delta_{\mu}$ vérifiant pour tout $g \in \Gamma$, $x \in \supp \nu_{\F}$ : 
\begin{equation*}
(\sigma_{F}(g,x), \chi(g)) \mod \Delta_{\mu} =  \varphi(x) - \varphi(gx) \tag{$\ast$} 
\end{equation*}

En rempla\c cant $x$ par $g^{-1}x $, on obtient 
$$(-\sigma_{F}(g^{-1},x), \chi(g)) \mod \Delta_{\mu} =  \varphi(g^{-1}x) - \varphi(x)$$
puis, $\Delta_{\mu}$ étant stable par passage à l'opposé :
$$(\sigma_{F}(g^{-1},x), \chi(g^{-1})) \mod \Delta_{\mu} =  \varphi(x) -\varphi(g^{-1}x)$$

Ainsi la relation $(\ast)$ est également satisfaite pour $g^{-1}$, puis finalement pour tout $g \in \langle \Gamma \rangle$, $x \in \supp \nu_{\PP}$. Considérons le sous groupe $\Gamma_{0}:=\{\chi = 0\} \subseteq \langle \Gamma \rangle$. Il contient le sous groupe dérivé $[\langle \Gamma \rangle, \langle \Gamma \rangle]$. Sa réalisation dans $G$  (via le morphisme $L \rightarrow G$) a donc pour adhérence de Zariski un sous groupe $G_{0}$ qui contient  $[G, G]$, et est en particulier d'indice fini dans $G$ (qui est semi-simple). On se donne une probabilité $\mu_{0} \in \mathcal{P}(\Gamma_{0})$ à support fini  engendrant un sous groupe Zariski dense dans $G_{0}$, et apériodique dans $F_{0}:= G_{0}/G_{c}$. La marche $(\Gamma_{0}, \mu_{0})$ est ainsi contractante sur $\PP_{0}:= G_{0}/P_{c}$ au dessus de $F_{0}$, d'unique probabilité stationnaire notée $\nu_{\PP_{0}} \in \mathcal{P}(\PP_{0})$.  On note $\Delta_{\mu_{0}} \subseteq \ag^F$ l'image résiduelle du cocycle $(\sigma_{F,\chi})_{|\Gamma_{0}\times \PP_{0}}$.  Comme $\supp \nu_{\PP} \subseteq \PP$ est $\langle \Gamma \rangle$-invariant, il est en particulier $\Gamma_{0}$-invariant, donc $\supp\nu_{\PP_{0}} \subseteq \supp \nu_{\PP}$. La relation $(\ast)$ entraine alors l'inclusion :
$\Delta_{\mu_{0}} \subseteq \Delta_{\mu}$.

Il suffit donc de vérifier que $\Delta_{\mu_{0}} =\ag^F\times \{0\} $ pour terminer la preuve. Comme $\text{Im}(\sigma_{F,\chi})_{|\Gamma_{0}\times \PP_{0}} \subseteq \ag^F\times \{0\}$, on a par  minimalité que $\Delta_{\mu_{0}} \subseteq \ag^F\times \{0\}$, si bien que $\Delta_{\mu_{0}}=\Delta \times \{0\}$ où $\Delta \subseteq \ag^F$ est un sous groupe fermé de $\ag^F$. Mais $\Delta$  contient l'image essentielle de $(\sigma_{F})_{ | \Gamma_{0}\times \PP_{0}}$ pour la $\mu_{0}$-marche, donc $\Delta = \ag^F$ d'après le \cref{apsigmap}, ce qui conclut.

\end{proof}

\bigskip

\begin{cor}
 $E_{\mu}= \ag^F \times \R$. 
\end{cor}

\begin{proof}
Rappelons que $$E_{\mu} :=\text{Vect}_{\R}\{(\sigma_{F}(g,x) - \sigma_{\mu}, \chi(g)), \,\,g \in \text{supp}\,\mu, \,x\in \text{supp}\, \nu_{\PP}\}$$
C'est un sous espace vectoriel  de $\ag^F \times \R$ qui contient $\Delta_{\mu}$ donc $\ag^F \times \{0\}$. Par ailleurs, comme $\chi$ est non trivial par hypothèse, $E_{\mu}$ n'est pas inclus dans  $\ag^F \times \{0\}$. Finalement  $E_{\mu}= \ag^F \times \R$. 
\end{proof}

\bigskip

\begin{cor} \label{Delta2}
 \item Si $\widebar{\chi(\Gamma)}=\R$, alors $\Delta_{\mu}=\ag^F \times\R$. 

\item Si $\chi(\Gamma) \subseteq \R$ est discret, alors $\Delta_{\mu}=\ag^F \times k \chi(\Gamma)$ pour un certain $k \geq 1$. 

\end{cor}

\begin{proof}
$\Delta_{\mu}$ est un sous groupe fermé cocompact de $E_{\mu}=\ag^F \times \R$ et contient $\ag^F \times \{0\}$. Il est donc de la forme  $\Delta_{\mu} =\ag^F \times \R$ ou  $\Delta_{\mu}=\ag^F \times c\Z$ pour un $c>0$. Comme $\Delta_{\mu} \subseteq \ag^F \times \widebar{\chi(\Gamma)}$ par minimalité, on obtient directement l'expression de $\Delta_{\mu}$ lorsque $\chi(\Gamma)$ est discret. On traite maintenant le cas où $\chi(\Gamma)$ est dense dans $\R$ en supposant par l'absurde que $\Delta_{\mu}=\ag^F \times c\Z$ pour un $c>0$. 

Par définition de $\Delta_{\mu}$, il existe une fonction continue $\varphi : \supp \nu_{\PP} \rightarrow \R/c\Z$ telle que pour tout $g \in \Gamma$, $x \in \supp \nu_{\F}$
$$ \chi(g) \mod c\Z  =  \varphi(x) - \varphi(gx) $$ 

On enrichit cette égalité en reprenant le groupe $L$ introduit dans la preuve du \cref{Delta}. Comme montré alors, la  relation précédente est  valable pour tout $g \in L$, $x\in  \supp \nu_{\PP}$. On obtient en particulier  pour tout $g \in [L, L]$, $x \in \supp \nu_{\PP}$, l'égalité $\varphi(x)= \varphi(gx)$. 

On en déduit que  $\varphi$ est constante sur chaque composante connexe de $\PP$.
Considérons le sous groupe $\Gamma'_{c}:= \langle \Gamma \cap G_{c} \rangle$ de $G_{c}$  généré par $\Gamma \cap G_{c}$ et son groupe dérivé $[\Gamma'_{c}, \Gamma'_{c}] \subseteq \Gamma'_{c}$. Le paragraphe précédent implique que $\varphi$ est invariante sous l'action de $[\Gamma'_{c}, \Gamma'_{c}]$. On se ramène ainsi à vérifier que $[\Gamma'_{c}, \Gamma'_{c}]$ agit minimalement sur chaque composante $f \cap \supp \nu_{\PP}$ où  $f \in F$. Notons $\PP_{c}:= G_{c}/P_{c}$ et $\Lambda_{\Gamma'_{c}} \subseteq \PP_{c}$ l'ensemble limite associé au sous groupe Zariski-dense $\Gamma'_{c} \subseteq G_{c}$. On trouve dans \cite{BQII} la description suivante : $ \supp \nu_{\PP} = \amalg_{f \in F} \Lambda_{\Gamma'_{c}}f$.  Par ailleurs $[\Gamma'_{c}, \Gamma'_{c}]$ est aussi Zariski-dense dans $G_{c}$ (en se rappelant que $G$ est semi-simple), et a même ensemble limite que $\Gamma'_{c}$ (étant distingué dans $\Gamma'_{c}$). Il agit donc minimalement sur les composantes $f \cap \supp \nu_{\PP}= \Lambda_{\Gamma'_{c}}f$ ce qui prouve le résultat annoncé.
 
Finalement, on conclut que $\chi$ est à valeurs dans un nombre fini de translatés de $c\Z$ ce qui contredit la densité de $\chi(\Gamma)$ dans $\R$.

\end{proof}

\bigskip

\subsubsection*{Preuve du \cref{TLL}}

L'enjeu de la preuve est de montrer qu'on peut choisir  la fonction $\widetilde{\varphi}_{0}$ qui apparait dans le théorème local limite à valeurs dans $\ag^F \times \widebar{\chi(\Gamma)}$ et de sorte que la mesure $\Pi_{\mu}$ qu'elle induit est une mesure de Haar sur $\ag^F \times \widebar{\chi(\Gamma)}$. On conclut alors par une  application directe du \cref{TLLBQ}, avec $C:= U \times I$ et $v_{n}:= u_{n}+t_{n}$.   

\bigskip

\emph{Supposons $\chi(\Gamma)$  dense dans $\R$}.  D'après le \cref{EDelta}, on peut choisir la fonction $\widetilde{\varphi}_{0}$ identiquement nulle, et la mesure $\Pi_{\mu}$ induite coïncide alors  avec $\pi_{\mu}$, i.e. est la mesure de Haar sur $\ag^F \times \R$ qui donne masse $1$ aux cubes unités de $\Delta_{\mu,0}$. 
\bigskip

\emph{Supposons $\chi(\Gamma)$ discret dans $\R$}. On peut supposer $\chi(\Gamma)=\Z$.  On commence par construire une bonne fonction $\widetilde{\varphi}_{0}$.  Le \cref{EDelta}, combiné au  \cref{redDelta}, donne l'existence d'une fonction continue $\varphi : \supp \nu_{\PP}\rightarrow \R/k\Z$ telle que pour $g \in \supp \Gamma$, $x \in \supp \nu_{\PP}$ on a $\chi(g) \mod k\Z= \varphi(x)-\varphi(gx)$.

En raisonnant comme dans la preuve du \cref{Delta2}, on constate que $\varphi$ est constante sur chaque composante $f \cap \supp \nu_{\PP}$ où  $f \in F$. On peut donc relever $\varphi$ en une fonction  $\widetilde{\varphi}: \PP \rightarrow \R$ localement constante telle que pour tout $g \in \Gamma$, $x \in \PP$, on ait $\chi(g) \mod k\Z= \widetilde{\varphi}(x)-\widetilde{\varphi}(gx) \mod k \Z$. Quitte à imposer $\widetilde{\varphi}_{|\PP_{c}}=0$,  on peut supposer $\widetilde{\varphi}$ à valeurs dans $\chi(\Gamma)=\Z$. On pose $\widetilde{\varphi}_{0}: \PP \rightarrow \ag^F \times \R, \,x \mapsto (0, \widetilde{\varphi}(x))$, fonction localement constante telle que pour tout $g \in G$, $x \in \PP$, $$\sigma_{F, \chi}(g,x) \mod \Delta_{\mu} =  \widetilde{\varphi}_{0}(x)-\widetilde{\varphi}_{0}(gx) \mod \Delta_{\mu}$$ 

Notons $\text{leb}_{\ag^F}$ la mesure de Lebesgue sur $\ag^F$ donnant masse $1$ aux cubes unités de $\Phi_{\mu |\ag^F \times \{0\}}$. On a par définition $\pi_{\mu}:= \leb_{\ag^F} \otimes \sum_{j \in k \Z} \delta_{j}$, puis $\Pi_{\mu} = \leb_{\ag^F} \otimes \frac{1}{|F|} \sum_{f \in F} \sum_{j \in k\Z} \delta_{j-\widetilde{\varphi}(f)}$. 

De plus, pour $g \in \Gamma$, on a
\begin{align*}
\Pi_{\mu}\{. - (0,\chi(g)) \}&= \leb_{\ag^F} \otimes \frac{1}{|F|} \sum_{f \in F} \sum_{j \in k\Z} \delta_{j-\widetilde{\varphi}(f) + \chi(g)}\\
&=\leb_{\ag^F} \otimes \frac{1}{|F|} \sum_{f \in F} \sum_{j \in k\Z} \delta_{j-\widetilde{\varphi}(gf)}\\
&= \Pi_{\mu}
\end{align*}
 La mesure $\Pi_{\mu}$ est donc $\chi(\Gamma)$-invariante, i.e. de la forme $\Pi_{\mu}= c\,\leb_{\ag^F} \otimes \sum_{j \in \Z} \delta_{j}$ pour un $c>0$.

\begin{proof}[]
\end{proof}

%%%%%%%%%%%%%%%%%%%%%%%%%%%%%%%%%%%%%%%%%%%%%%%%%%%%%%%%%%%%%%%%%%%%%%%%%%%%%%%%%%%%%%%%%%%%%%%%%%%%%%%%%%%%%%%%%%%%%%%%%%%%%%%%%%%%%%%%%%%%%%%%%%%%%%%%%%%%%%%%%%%%%%%%%%%%%%%%%%%%%%%%%%%%%%%%%%%%%%%%%%%%%%%%%%%%%%%%%%%%%%%%%%%%%%%%%%%%%%%%%%%%%%%%%%%%%%%%%%%%%%%%%%%%%%%%%%%%%%%%%%%%%%%%%%%%%%%%%%%%%%%%%%%
%%%%%%%%%%%%%%%%%%%%%%%%%%%%%%%%%%%%%%%%%%%%%%%%%%%%%%%%%%%%%%%%%%%%%%%%%%%%%%%%%%%%%%%%%%%%%%%%%%%%%%%%%%%%%%%%%%%%%%%%%%%%%%%%%%%%%%%%%%%%%%%%%%%%%%%%%%%%%%%%%%%%%%%%%%%%%%%%%%%%%%%%%%%%%%%%%%%%%%%%%%%%%%%%%%%%%%%%%%%%%%%%%%%%%%%%%%%%%%%%%%%%%%%%%%%%%%%%%%%%%%%%%%%%%%%%%%%%%%%%%%%%%%%%%%%%%%%%%%%%%%%%%%%

\newpage

\section{Loi des angles et contrôle de la dérive}

Reprenons le schéma de preuve de la dérive exponentielle expliqué à la fin de la section $4$.  On dispose de points $c = (b,z, O, x), c_{p}=(b,z, O, x+ u_{p}) \in B^+$ où $u_{p} \in \R^d$, est très petit, et agit sur la coordonnée en $\T^d$. C'est le terme de dérive.  Pour $n \geq 0$, nous avons vu dans la section précédente que les  $n$-fibres passant par $c$ et $c_{p}$ sont paramétrées par $\Gamma^n$. Etant donné $a \in \Gamma^n$, les points des $n$-fibres correspondants diffèrent seulement d'un nouveau terme de dérive, donné par : $D_{n,c, c_{p}}(a) = a_{1}\dots a_{n} b^{-1}_{n}\dots b^{-1}_{1} u_{p} \in \R^d$. L'objectif de cette section est de contrôler la norme et la direction de la dérive pour $n$ grand,  lorsqu'on ne voit les fibres qu'à travers une fenêtre $W \subseteq B^+$ de mesure finie. On montrera notamment que la norme croît exponentiellement avec $n$, justifiant le nom de dérive exponentielle.

\bigskip

La section (8.1) introduit une notion de points de densité et explique comment ils conditionnent la croissance de la norme et la contraction des directions dans une représentation. 

La section (8.2) décrit la loi asymptotique de ces points de densité. 

La section (8.3) est un préliminaire technique à la section (8.4).

La section (8.4) décrit la loi asymptotique des points de densité le long des morceaux de fibres : c'est la loi des angles.

La section (8.5) décrit la divergence de la projection de Cartan le long des morceaux de fibres.

La section (8.6) conclut en démontrant les théorèmes de contrôle annoncés.

\bigskip
A partir de la section $(8.4)$, les énoncés deviendront plus complexes. Pour les alléger, on utilisera la notation ``\emph{Soit $c \in B^+$ $\beta^+$-typique}'' pour signifier que les propriétés qui suivent sont valables pour $\beta^+$-presque tout point $c\in B^+$.

\bigskip

\subsection{Points de densité}

Soit $G$ un sous groupe algébrique réel semi-simple (implicitement linéaire, on spécifiera plus tard $G= \widebar{\Gamma}^Z \subseteq SL_{d}(\R)$). On se donne pour $G$ des objets $K$, $\mathfrak{a}$, $A$, $\Sigma$, $\Pi$, $P$, $\PP=G/P_{c}$, etc. comme dans les rappels algébriques de la section $4$. On réalise $G$ via un plongement $G \hookrightarrow SL(V_{0})$, où $V_{0}$ est un $\R$-espace vectoriel de dimension finie.  On peut munir $V_{0}$ d'un produit scalaire $\langle .,.\rangle$ tel que $K \subseteq O(V_{0})$, $A \subseteq \text{Sym}(V_{0})$ (\cite{BQRW}, $8.9$). En considérant la décomposition de Cartan, on constate que le groupe $G$ est alors stable par transposition : $G =\,^tG$.

Rappelons que la décomposition de Cartan introduite dans la section $4$ est donnée par l'égalité $G= K \exp(\ag^+)K_{c}$. Cependant, nous aurions très bien pu  considérer la décomposition duale $G= K_{c}\exp(\ag^+) K$. Cette dernière fournit une notion de projection de Cartan $\kappa'$ liée à $\kappa$ via  l'action de $F$ : si $g \in G$, $\kappa'(g)= f_{g}.\kappa(g)$ où $G \rightarrow F, f \mapsto f_{g}$ est la projection canonique. Nous allons définir des points de densité pour ces deux décompositions et voir comment ils sont liés par l'action de $F$.

Soit $g \in G$. La décomposition de Cartan permet d'écrire $g$ sous la forme $$g= k^c_{g}m_{g}\exp(x_{g})l^c_{g}$$ où $k^c_{g}, l^c_{g}\in K_{c}$, $m_{g}\in M=K\cap P$, $x_{g} \in \ag^+$. 

On a alors automatiquement une décomposition duale de $g$ et des décompositions de $^tg$ via les égalités :
\begin{align*}
&g=k^c_{g}  \exp(f_{g}.x_{g}) m_{g}l^c_{g}  \\ & ^tg= (l^c_{g})^{-1} m^{-1}_{g} \exp(f_{g}.x_{g}) (k^c_{g})^{-1}  \\ & ^tg= (l^c_{g})^{-1}  \exp(x_{g}) m^{-1}_{g} (k^c_{g})^{-1}
\end{align*}

On choisit alors pour tout $g \in G$, une décomposition $g= k^c_{g}m_{g}\exp(x_{g})l^c_{g}$ comme ci dessus de sorte que la fonction $g \mapsto (k^c_{g},m_{g}, x_{g}, l^c_{g})$ soit mesurable et compatible avec la transposition : $ (k^c_{ ^tg},m_{ ^tg}, x_{^tg},l^c_{^tg}) =(l^c_{g})^{-1}, m^{-1}_{g} , f_{g} .x_{g}, (k^c_{g})^{-1})$. On définit alors les \emph{points de densité} annoncés en notant $\xi_{0}:= P_{c}/P_{c} \in \PP$ le drapeau standard et en posant : 

$$\xi^+_{g}:= k^c_{g}m_{g} \xi_{0},\,\,\,\, \,\,\,\,\xi^-_{g}:= (m_{g}l^c_{g})^{-1}\xi_{0},\,\,  \,\,\,\,\,\, \eta^+_{g}:= k^c_{g}\xi_{0}, \,\,\,\,\,\,\,\, \eta^{-}_{g}:= (l^c_{g})^{-1} \xi_{0}$$

\bigskip

On a $\xi^+_{g}, \xi^-_{g} \in \PP, \eta^+_{g}, \eta^-_{g} \in \PP_{c}= G_{c}/P_{c}$ et les relations $\xi^+_{g}= \eta^+_{g}f_{g}$, $\xi^-_{g}= \eta^-_{g} f^{-1}_{g}$. De plus, la compatibilité des décompositions choisies entraine que $\xi^+_{^tg}=\xi^-_{g}$, $\eta^+_{^tg}= \eta^-_{g}$.

\bigskip

Soit $(V, \rho)$ une représentation algébrique irréductible  de $G_{c}$, $r :=\dim_{\text{prox}}\rho(G_{c})$ sa dimension proximale, $\omega  \in \ag^\star$ son plus haut poids. On note $(W, \rho):= \text{Ind}^G_{G_{c}}(V, \rho)$ sa représentation induite (voir $(4.1)$). On munit $W$  d'un ``bon'' produit scalaire i.e. choisi tel que $\rho(K) \subseteq O(W)$, $\rho(A) \subseteq \text{Sym}(W)$ et les sous espaces $(V^f)_{f \in F}$ décomposant $W$ sont deux à eux orthogonaux. On se donne $g \in G$, $v \in V$ et on veut estimer la norme et la direction $\rho(g)v \in V$. Cela est possible lorsque $v$ est suffisament éloigné d'un sous-espace vectoriel $V^-_{g} \subseteq V$. Plus précisément, on note $V^U:=\{v \in V, \rho(U)v=v\}$ (sous espace propre du plus haut poids $\omega$ de $\rho$, de dimension $r$), et on pose : 
$$W^+_{g}:= \rho(k^c_{g}m_{g}) V^U \in \mathbb{G}_{r}(W),\,\,\,\,\,\,\,\,\,\,\,\,\,\, V^-_{g}:=((l^c_{g})^{-1}V^U)^\perp \cap V \in \mathbb{G}_{n-r}(V)$$

On peut les appeler respectivement \emph{espace attracteur} et \emph{espace répulsif} de $\rho(g)$ dans $W$. Ils sont liés au points de densités $\xi^+_{g}$ et $\eta^-_{g}$ : en notant $\xi \mapsto V_{\xi}$ le morphisme $G$-équivariant de $\PP$ dans $\mathbb{G}_{r}(W)$ vérifiant $V_{\xi_{0}}= V^U$ (voir $4.1$), on a $ W^+_{g}= V_{\xi^+_{g}}$, $V^{-}_{g}= V_{\eta^-_{g}}^\perp \cap V$. On remarquera que les intersections avec $V$ dans les définitions de $V^-_{g}$ sont là pour garantir que $V^-_{g}$ ne rencontre aucun autre sous-espace $V^f\subseteq W$ que $V$, autrement dit on considère l'orthogonal au sein de $V$ et non de $W$.  

\newpage

\begin{lemme}[Croissance et contraction]\label{LA1}
Supposons la représentaiton $\rho(G_{c})$ non bornée dans $SL(V)$ (ou de manière équivalente $r < \dim V$). Alors pour, $g \in G$, $v \in V-\{0\}$, on a : 
\begin{itemize}
\item $ e^{\omega (\kappa(g))} ||v|| d(\R v, V^-_{g})  \leq  ||\rho(g)v|| \leq e^{\omega (\kappa(g))} ||v||$
\item  $d(\rho(g) \R v, W^+_{g}) \leq \frac{1}{d(\R v, V^-_{g})} \max_{\alpha \in \Pi} e^{-\alpha ( \kappa(g))}$
\end{itemize}
\end{lemme}

\bigskip

\begin{rem.}
$1)$ Le terme $e^{\omega (\kappa(g))}$ n'est autre que $||\rho(g)_{|V}||$ (cf. \cref{rep}). Le premier point compare donc $||\rho(g)v||$ et $||\rho(g)_{|V}||\,||v||$.

$2)$ On pourrait formuler un lemme similaire contrôlant $\rho(g)w$ pour $w \in V^{f^{-1}_{g}}$, à l'aide de $V_{\eta^+_{g}}$, $V_{\xi^{-}_{g}}$.

$3)$ La distance sur l'espace projectif $\mathbb{P}(W)$ est donnée par $d(\R w, \R w') := \frac{|| w \wedge w'||}{||w|| ||w'||}$, et si $E \subseteq W$ est un sous-espace vectoriel, on définit $d(\R w, E):= \inf_{\R w' \subseteq E} d(\R w, \R w')$. 

$4)$ Le \cref{LA1} reste vrai lorsqu'on remplace $V$ et $W$ par une même représentation fortement irréductible $V$ de $G$, munie d'un produit scalaire $K$-invariant tel que les éléments de  $A$ sont auto-adjoints. Les espaces attractif et répulsif de $g \in G$ sont définis de même par $V^+_{g}:= V_{\xi^+_{g}} (= V_{\eta^+_{g}})$, $V^{-}_{g}:= V_{\eta^-_{g}}^\perp (=V_{\xi^-_{g}}^\perp)$. La version fortement irréductible se déduit du lemme précédent en utilisant l'application $\Psi$ du \cref{redind}.
\end{rem.}

\bigskip

\begin{proof}

En remarquant que $k_{g}m_{g}$ n'intervient pas dans les inégalités et quitte à remplacer $v$ par $(l^c_{g})^{-1}v$, on pourra supposer que $g = a \in \exp(\ag^+)$. Par abus de notation, on écrira  $V_{\xi_{0}}^\perp$ pour désigner $V_{\xi_{0}}^\perp \cap V$. On décomposera $v \in V$ en $v=v_{0}+v_{0}^\perp$ où $v_{0}\in V_{\xi_{0}}$, $v_{0}^\perp \in V_{\xi_{0}}^\perp $.

Le premier point demande de vérifier que $||\rho(a)v|| \geq e^{\omega(\kappa(a))} \,||v|| d(\R v, V_{\xi_{0}}^\perp)$.  On a $d(\R v, V_{\xi_{0}}^\perp) = \frac{||v_{0}||}{||v||}$. Par ailleurs, en remarquant que  $\rho(a)$ stabilise $V_{\xi_{0}}$ et $V_{\xi_{0}}^\perp$, on obtient $||\rho(a)v|| \geq ||\rho(a)v_{0}|| = e^{\omega(\kappa(a))} ||v_{0}|| = e^{\omega(\kappa(a))} ||v|| d(\R v, V_{\xi_{0}}^\perp )$. Cela prouve le premier point. 

\bigskip

Le deuxième point demande de vérifier que $d(\rho(a) \R v, V_{\xi_{0}}) \leq \frac{1}{d(\R v, V_{\xi_{0}}^\perp)} \max_{\alpha \in \Pi} e^{-\alpha ( \kappa(a))}$. On a $d(\R v, V_{\xi_{0}}^\perp) = \frac{||v_{0}||}{||v||}$
De même, comme  $\rho(a)$ stabilise $V_{\xi_{0}}$ et $V_{\xi_{0}}^\perp$,  on peut écrire   $d(\rho(a) \R v, V_{\xi_{0}})  = \frac{|| \rho(a) v_{0}^\perp||}{||\rho(a)v||} $.  On veut donc montrer que 
$$\frac{|| \rho(a) v_{0}^\perp||}{||\rho(a)v||} \frac{||v_{0}||}{||v||} \leq  \max_{\alpha \in \Pi} e^{-\alpha ( \kappa(a))}$$
Mais $||\rho(a)v|| \geq ||\rho(a)v_{0}|| = e^{\omega(\kappa(a))} ||v_{0}||$ donc $\frac{||v_{0}||}{||\rho(a)v||} \leq e^{-\omega(\kappa(a))}$. Notons $\omega_{i} \in \mathfrak{a}^\star, i \in I$ les poids de la représentation $(V,\rho_{|G_{c}})$ autres que $\omega$. Comme $\rho(G_{c})$ n'est pas borné, on a $r :=\dim_{\text{prox}}\rho(G_{c})<\dim V$ donc $I \neq \emptyset$. Par définition du plus haut poids, on a pour tout $i \in I$, l'écriture $\omega-\omega_{i} =\sum_{\alpha \in \Pi} k_{\alpha,i} \alpha$ où $k_{\alpha,i} \in \N$. Décomposons $v_{0}^\perp$ en vecteurs propres $v_{0}^\perp = \sum_{i \in I} v_{0,i}^\perp$ avec $\rho(a)v_{0,i}^\perp = e^{\omega_{i}(\kappa(a))} v_{0,i}^\perp$ et les $v_{0,i}$ deux à deux orthogonaux (quitte à en choisir certains nuls pour qu'il n'y en ait qu'un par valeur propre).  Alors 
$$||e^{-\omega(\kappa(a) ) }\rho(a)v_{0}^\perp||^2 = \sum_{i\in I} e^{-2(\omega -\omega_{i})(\kappa(a))}||v_{0,i}^\perp ||^2 \leq  \max_{\alpha \in \Pi} e^{-2\alpha ( \kappa(a))} \sum_{i\in I}||v_{0,i}^\perp||^2$$ car $I \neq \emptyset$ et $\kappa(a) \in \ag^+$. Comme $\sum_{i\in I}||v_{0,i}^\perp||^2\leq ||v||^2$, on obtient finalement que $\frac{|| \rho(a) v_{0}^\perp||}{||\rho(a)v||} \frac{||v_{0}||}{||v||} \leq  \max_{\alpha \in \Pi} e^{-\alpha ( \kappa(a))}$, prouvant le point $2$.
\end{proof}

\bigskip

\bigskip

\emph{Dans la suite on revient au cadre de la marche sur $\T^d \times \R$} correspondant à $V_{0}= \R^d$, $G:= \widebar{\Gamma}^Z \subseteq SL(V_{0})$, groupe semi-simple, non borné,  fortement irréductible sur $V_{0}$. Le dernier lemme montre que la dérive $a_{1}\dots a_{n} b^{-1}_{n}\dots b^{-1}_{1} u_{p} \in V_{0}$ est contrôlée par la position de l'espace répulsif $(V_{0})^-_{a_{1}\dots a_{n}}$ par rapport à $b^{-1}_{n}\dots b^{-1}_{1} u_{p}$, et par les évaluations $\alpha(\kappa(a_{1} \dots a_{n})) \in \R_{+}$. La section est essentiellement consacrée à décrire la distribution des espaces répulsifs $(V_{0})^-_{a_{1}\dots a_{n}}$. Cet objectif sera atteint à travers la \emph{loi des angles}. On montrera aussi que les  $\alpha(\kappa(a_{1} \dots a_{n}))$ tendent vers $+\infty$, puis les théorèmes de contrôle annoncés.

\bigskip

\subsection{Convergence des points de densité}

Les notations sont celles des sous-sections $(4.1)$ et $(8.1)$. On reprend le groupe algébrique réel semi-simple $G$ de $(8.1)$, on se donne $\mu \in \mathcal{P}(G)$ une probabilité sur $G$ à support fini et \emph{apériodique} dans $F:=G/G_{c}$. Cette dernière hypothèse signifie que $\supp \mu$ engendre un sous-groupe Zariski dense de $G$ et que son image dans $F$ par la projection naturelle n'est pas inclus dans le translaté d'un sous groupe propre cocyclique de $F$. Elle est par exemple vérifiée si $\mu(e)>0$ (ce que l'on peut supposer pour prouver le \cref{TH0}). On note $B:= G^{\N^\star}$, $\beta:=\mu^{\otimes \N^\star} \in \mathcal{P}(B)$, $T:B \rightarrow B, (b_{i})_{i \geq 1}\mapsto (b_{i+1})_{i \geq 1}$ le shift unilatère. On rappelle que la variété drapeau $\PP:= G/P_{c}$ admet une unique probabilité $\mu$-stationnaire, notée $\nu_{\PP} \in \mathcal{P}(\PP)$, et que ses mesures limites s'écrivent $\beta$-presque sûrement $\nu_{\PP, b}= \frac{1}{|F|}\sum_{f \in F}\delta_{\xi_{b}.f}$ avec $\xi_{b} \in \PP_{c}:= G_{c}/P_{c}$. Nous allons voir que les points de densité $\xi^+_{b_{1}\dots b_{n}}$ approximent bien les drapeaux limites $\xi_{b}.f_{b_{1}\dots b_{n}}$ si $n$ est assez grand. On notera $\xi_{b,f}:=\xi_{b}.f$. 

\begin{lemme} 
\label{convrep} Soit $(V, \rho)$ une représentation irréductible proximale  de $G_{c}$, $(W, \rho):= \emph{Ind}^G_{G_{c}}(V, \rho)$ sa représentation induite. 
Il existe une constante $\varepsilon >0$ et un rang $n_{0}\geq 0$ tels que pour tout $n \geq n_{0}$, on a : $$ \beta(b \in B, \,\, d(V_{\xi^+_{b_{1}\dots b_{n}}}, \,\,\rho(b_{1}\dots b_{n})V_{\xi_{T^nb}}) \leq e^{-\varepsilon n} ) \geq 1- e^{-\varepsilon n}$$
\end{lemme}

\begin{proof}
On peut supposer $\dim V \geq 2$. L'hypothèse de proximalité implique alors que l'image  $\rho(G_{c})$ n'est pas bornée. On peut donc appliquer le \cref{LA1} qui donne la majoration 
$$d(V_{\xi^+_{b_{1}\dots b_{n}}}, \,\,\rho(b_{1}\dots b_{n})V_{\xi_{T^nb}}) \leq \frac{1}{d(V_{\xi_{T^nb}}, V^-_{b_{1}\dots b_{n}})} \max_{\alpha \in \Pi} e^{-\alpha(\kappa(b_{1}\dots b_{n}))}$$

La régularité H\"older des mesures stationnaires sur les espaces projectifs (\cite{BQRW}, section $14.1$, appliquée à $(G_{c}, \mu_{G_{c}})$) affirme que la loi image de $\beta$ par l'application $B \rightarrow \mathbb{P}(V), b \mapsto V_{\xi_{b}}$  est H\"older régulière. Cela signifie qu' il existe des constantes $C,s>0$   telles que pour tout $r>0$, pour tout hyperplan $H \subseteq V$, 
$$\beta(b \in B, \,\,d(V_{\xi_{b}}, H) \leq r) \leq Cr^{s}$$
 En particulier, pour tout $r>0$, on a $$\beta(b \in B, \,\,d(V_{\xi_{T^nb}}, V^-_{b_{1}\dots b_{n}}) > e^{-rn}) \geq 1-Ce^{-rsn} $$

Par ailleurs, le principe des grandes déviations (\cite{BQRW}, $13.6$) affirme que pour $t>0$, il existe un rang $n_{0}\geq 0$ et une constante $\delta >0$ tels que pour $n \geq n_{0}$: 

$$\beta(b \in B, \,\frac{||\kappa(b_{1}\dots b_{n})-n \sigma_{\mu} ||}{n} \leq t)  \geq 1-e^{-\delta n}$$

Rappelons que le vecteur de Lyapunov $\sigma_{\mu}$ est dans l'intérieur de la chambre de Weyl $\ag^+$, autrement dit que $m:=\min_{\alpha\in \Pi} \alpha(\sigma_{\mu})>0$ (\cite{BQRW}, $10.4$). On peut donc choisir le paramètre $r < \frac{1}{3}m$,  et le paramètre  $t>0$ de sorte que pour tout $x \in \ag$ tel que $||x || \leq t$, on a $|\alpha(x)| \leq \frac{1}{3}m$. On déduit alors de  ce qui précède que pour tout $n \geq n_{0}$, 
$$ \beta(b \in B, \,\, d(V_{\xi^+_{b_{1}\dots b_{n}}}, \,\,\rho(b_{1}\dots b_{n})V_{\xi_{T^nb}}) \leq e^{-n \frac{m}{3} } ) \geq 1- Ce^{-rsn}-e^{-\delta n}$$

Cela conduit au résultat en choisissant $\varepsilon>0$ assez petit puis $n_{0} \geq 0$ assez grand.
\end{proof}

\bigskip

\bigskip

On déduit du lemme précédent  l'approximation des drapeaux limites $(\xi_{b,f})_{b \in B, f \in F}$ par les points de densités.

\begin{cor} \label{convdens}
Il existe une constante $\varepsilon >0$ et un rang $n_{0}\geq 0$ tels que pour tout $n \geq n_{0}$, on a :
$$ \beta(b \in B, \,\, \forall p \geq n, \,  d(\xi^+_{b_{1}\dots b_{p}}, \xi_{b, f_{b_{1}\dots b_{p}}}) \leq e^{-\varepsilon p} ) \geq 1- e^{-\varepsilon n}$$
\end{cor}

\begin{proof}
On a vu dans $(4.1)$ que la variété drapeau admet un plongement  
$$\PP \hookrightarrow \amalg_{f \in F} \Pi_{\alpha \in \Pi} \mathbb{P}(V^f_{\alpha}),\, \xi \mapsto (V_{\alpha, \xi})$$
 où les $V_{\alpha}$ sont des représentations proximales irréductibles de $G_{c}$, et que la distance sur $\PP$ est définie via ce plongement. En appliquant le lemme précédent aux représentations $V_{\alpha}$,  et en se rappelant que $b_{1}\dots b_{n}\xi_{T^nb}=\xi_{b, f_{b_{1}\dots b_{n}}}$ $\beta$-ps,  on obtient qu'il existe une constante $\varepsilon >0$ et un rang $n_{0}\geq 0$ tels que pour tout $n \geq n_{0}$, on a : 
 $$ \beta(b \in B, \,\, d(\xi^+_{b_{1}\dots b_{n}}, \,\,\xi_{b, f_{b_{1}\dots b_{n}}}) \leq e^{-\varepsilon n} ) \geq 1- e^{-\varepsilon n}$$
 
 Le corollaire en découle directement : pour $n \geq n_{0}$,
 \begin{align*}
 \beta(b \in B, \,\, \exists p \geq n, \,  d(\xi^+_{b_{1}\dots b_{p}}, \xi_{b, f_{b_{1}\dots b_{p}}}) > e^{-\varepsilon p} ) &\leq \sum_{p \geq n} e^{-\varepsilon p }\\
 & = \frac{e^{-\varepsilon n}}{1- e^{-\varepsilon}}
 \end{align*}
d'où résultat en passant au complémentaire (quitte à diminuer la valeur de $\varepsilon>0$, augmenter le rang $n_{0} \geq 0$).

\end{proof}

On en déduit la loi des points de densité $(\xi^+_{b_{1}\dots b_{n}})$ pour $n$ grand.

\begin{cor} \label{convdensmes} Soit $\varphi \in C^0(\PP)$. Alors 

$$\int_{B}\varphi(\xi^+_{b_{1}\dots b_{n}}) d\beta(b) \rightarrow \nu_{\PP}(\varphi)$$

\end{cor}

\bigskip

\begin{rem.}
La relation $\eta^+_{g}f_{g}= \xi^+_{g}$ entraine que $\int_{B} \varphi(\eta^+_{b_{1}\dots b_{n}})d\beta(b)  \rightarrow \nu_{\PP_{c}}(\varphi)$.
\end{rem.}

\begin{proof}
D'après le  \cref{convdens}, il suffit de noter $\nu_{n} \in \mathcal{P}(\PP)$ la loi de $\xi_{b, f_{b_{1}\dots b_{n}}}$ quand $b$ varie suivant $\beta$ et de montrer que $\nu_{n}(\varphi) \rightarrow \nu_{\PP}(\varphi)$. Or on a l'égalité $\beta$-presque sûre $ \xi_{b, f_{b_{1}\dots b_{n}}}= b_{1}\dots b_{n}\xi_{T^nb}$. Cela entraîne $\nu_{n}=\mu^{\star n}\star \nu_{\PP_{c}}$. Comme l'action de $\Gamma$ sur $\PP$ est $\mu$-contractante au dessus de $F$, et comme $\mu$ est apériodique dans $F$, on a pour tout $\xi  \in \PP$ la convergence faible-$\star$  $\mu^{\star n} \star \delta_{\xi} \rightarrow \nu_{\PP}$, ce qui conclut la preuve. Remarquons que  dans le cas où $F$ est trivial, cette dernière assertion se démontre en utilisant le lemme $11.5$ $(b)$ de \cite{BQRW} qui donne immédiatement que pour toute fonction continue $\varphi : \PP\rightarrow \R$, on a $\mu^{\star n}\star \delta_{\xi}( \varphi)-\nu_{\PP}(\varphi)=\mu^{\star n}\star \delta_{\xi}( \varphi)-\mu^{\star n}\star \nu_{\PP}(\varphi)\rightarrow 0$. Dans le cas général, on démontre de la même fa\c con que ${\mu^{\star n}\star \delta_{\xi}}_{|\PP_{c}}\rightarrow \nu_{\PP|\PP_{c}}$ puis on utilise le corollaire $11.7$ de \cite{BQRW} (requérant l'apériodicité) pour conclure.

\end{proof}

\subsection{Contrôle de $\theta_{n}(b)$}

On a défini dans la section $4$ le projeté du cocycle d'Iwasawa $\sigma_{F}= \frac{1}{|F|}\sum_{f \in F}f.\sigma : G \times \PP \rightarrow \ag^F$, la  fonction $\theta : B \rightarrow \ag^F, b \mapsto \sigma_{F}(b_{1},\xi_{Tb})$, et posé $\theta_{n}(b) := \sum_{k=1}^n\theta(b_{1}\dots b_{k}, T^kb)= \sigma_{F}(b_{1}\dots b_{n}, \xi_{T^nb})$ (où la dernière égalité est $\beta$-presque sûre). Nous allons estimer le comportement asymptotique de $\theta_{n}(b)$. Cela s'avérera utile pour vérifier les conditions de majoration dans le théorème local limite.

\begin{lemme} \label{thetan0}
Pour $\beta$-presque tout $b \in B$, la suite $$(\frac{\sigma(b_{1}\dots b_{n}, \xi_{T^nb}) - n \sigma_{\mu}}{\sqrt{n \log \log n}})_{n \geq 3} $$
est bornée dans $\ag$. 
\end{lemme}

\begin{proof}
La loi du logarithme itéré (\cite{BQRW} $13.6$) implique que pour $\beta$-presque tout $b \in B$,  la suite $(\frac{\kappa(b_{1}\dots b_{n}) - n \sigma_{\mu}}{\sqrt{n \log \log n}})_{n \geq 3} $ est bornée dans $\ag$. Il suffit donc de prouver que pour $\beta$-presque tout $b \in B$,  la suite $(\frac{\sigma(b_{1}\dots b_{n}, \xi_{T^nb})- \kappa(b_{1}\dots b_{n}) }{\log n})_{n \geq 2}$ est bornée dans $\ag$. Pour cela, on reprend les représentations $(V_{\alpha}, \rho_{\alpha})_{\alpha \in \Pi}$ introduites dans la section $(4.1)$, ainsi que leurs plus hauts poids $(\omega_{\alpha})_{\alpha \in \Pi}$. Commes les $\omega_{\alpha}$ forment une base de $\ag^\star$, il suffit de prouver que  pour tout $\alpha \in \Pi$,  $\beta$-presque tout $b \in B$,  la suite $(\frac{\omega_{\alpha}(\sigma(b_{1}\dots b_{n}, \xi_{T^nb})- \kappa(b_{1}\dots b_{n}) )}{\log n})_{n \geq 2}$ est bornée dans $\R$.
Fixons $\alpha \in \Pi$, notons $(V, \rho, \omega)=(V_{\alpha}, \rho_{\alpha}, \omega_{\alpha})$.  On se donne une famille mesurable $(v_{b})_{b \in B} \in V^B$ telle que pour tout $b \in B$,  $v_{b} \in V_{\xi_{b}} -\{0\}$. En choisissant un bon produit scalaire sur la représentation induite de $V$ et en accord avec le  \cref{rep}, on peut écrire :
$\omega(\sigma(b_{1}\dots b_{n}, \xi_{T^nb})) = \log \frac{||\rho(b_{1}\dots b_{n})v_{T^nb}||}{||v_{T^nb}||}$, $\kappa(b_{1}\dots b_{n})= \log ||\rho(b_{1}\dots b_{n})_{|V}||$. On cherche donc à montrer que pour $\beta$-presque tout $b \in B$, il existe $C>0$ tel que pour $n \geq 0$ assez grand, 

$$\frac{||\rho(b_{1}\dots b_{n}) v_{T^nb}||}{||\rho(b_{1}\dots b_{n})_{|V}|| \,||v_{T^nb}||} \geq n^{-C}$$

D'après le \cref{LA1}, cette condition est réalisée si 
$$d(V_{\xi_{T^nb}}, V^-_{b_{1}\dots b_{n}}) \geq n^{-C}$$

Fixons $C>0$. Notons $E_{n}=\{b \in B, d(V_{\xi_{T^nb}}, V^-_{b_{1}\dots b_{n}}) < n^{-C}\}$. Alors $\beta(E_{n}) \leq M n^{-Cs}$ pour certains $M, s>0$ indépendants de $n$, par régularité H\"older des mesures stationnaires sur les espaces projectifs (voir preuve du \cref{convrep}). En choisissant $C>s^{-1}$, on obtient que $\sum_{n \geq 0} \beta(E_{n}) < \infty$. Le lemme de Borel-Cantelli donne alors le résultat.  

\end{proof}

\begin{cor} \label{thetan}
Pour $\beta$-presque tout $b \in B$, la suite $$(\frac{\theta_{n}(b) - n \sigma_{\mu}}{\sqrt{n \log \log n}})_{n \geq 3} $$
est bornée dans $\ag^F$. 
\end{cor}

\begin{proof}

Cela découle directement du lemme précédent et l'invariance de $\sigma_{\mu}$ sous l'action de $F$ sur $\ag$.   
\end{proof}

\bigskip

\subsection{Loi des angles}

Considérons le système dynamique fibré $(B^+, \beta^+, T^+)$ des sections (4.2) et (6.1), avec les hypothèses $\mu(e)>0$ et $\chi$ non trivial (permettant d'appliquer le théorème local limite de la section précédente, et garantissant l'apériodicité de $\mu$ dans $F$). On fixe une fenêtre  $W \subseteq B^+$  de la forme 
$$W = B \times U \times O_{r}(\R) \times  \T^d \times I$$
 où $U \subseteq \ag^F$ $I\subseteq \R $ sont des parties convexes bornés ouvertes non vides de $\ag^F$ et $\R$, avec $I$ assez grand pour que tout translaté $I+t$ de $I$ rencontre $\chi(\Gamma)$. On rappelle que la $n$-fibre de $(B^+, T^+)$ passant par un point $c \in B^+$ est $F_{n,c} := (T^+)^{-n}(T^+)^n(c)$ et qu'on appelle morceau de fibre l'intersection  $F_{n,c} \cap W$. Nous avons expliqué comment désintégrer $\beta^+_{|W}$ par rapport aux morceaux de fibres $(F_{n,c} \cap W)_{c \in W}$, et obtenu ainsi des probabilités  conditionnelles $(\beta^+_{W, n, c})_{c \in W}$, concentrées sur ces morceaux de fibres. Rappelons aussi que si  $c =(b,z,O, x)\in W$, tout point $c'$ de la $n$-fibre passant par $c$ est de la forme $c'=(b'_{1}\dots b_{n}'T^nb, z', O', x')$. 

La loi des angles décrit le comportement des points de densité $\xi^+_{b'_{1}\dots b'_{n}}, \xi^-_{b'_{1}\dots b'_{n}} \in \PP$, $\eta^+_{b'_{1}\dots b'_{n}}, \eta^-_{b'_{1}\dots b'_{n}} \in \PP_{c}$ le long des morceaux de fibres. On en déduit le distribution des espaces répulsifs dans n'importe quelle représentation irréductible $(V, \rho)$ de $G_{c}$ grâce à la relation $V^-_{b'_{1}\dots b'_{n}} = (V_{\xi^-_{b'_{1}\dots b'_{n}}})^\perp \cap V$ (l'orthogonal étant pris par rapport à un bon produit scalaire sur la représentation induite). On appelle  $\widecheck{\mu} \in \mathcal{P}(G)$ le poussé en avant de la probabilité $\mu$ par la transposition, et $\widecheck{\nu}_{\PP} \in \mathcal{P}(\PP)$  l'unique probabilité $\widecheck{\mu}$-stationnaire sur $\PP$ (voir $4.1$) et $\widecheck{\nu}_{\PP_{c}}$ sa restriction à $\PP_{c}$.  

\bigskip

\begin{th.}[Loi des angles]
Soit $\varphi : \PP \rightarrow \R$ une fonction continue. Pour $\beta^+_{W}$-presque tout $c=(b,z, O, x) \in W$, on a : 
$$\int_{W} \varphi(\xi^{+}_{b'_{n}\dots b'_{1}})  \,\,d \beta^+_{W, n, c}(c')\, \underset{n \to +\infty}{\rightarrow} \,\nu_{\PP}(\varphi)$$

$$\int_{W} \varphi(\xi^{-}_{b'_{1}\dots b'_{n}})  \,\,d \beta^+_{W, n, c}(c')\, \underset{n \to +\infty}{\rightarrow} \,\widecheck{\nu}_{\PP}(\varphi)$$

\end{th.}

\bigskip

\begin{rem.} Les relation $\eta^+_{g}f_{g}= \xi^+_{g}$, $\eta^-_{g}f^{-1}_{g}= \xi^-_{g}$ entrainent  les lois des angles pour $\eta^+$, $\eta^-$ : 
$$\int_{W} \varphi(\eta^{+}_{b'_{n}\dots b'_{1}})  \,\,d \beta^+_{W, n, c}(c')\, \underset{n \to +\infty}{\rightarrow} \,\nu_{\PP_{c}}(\varphi)
\,\,\,\,\,\,\,\,\,\,\,\,\,\,
 \int_{W} \varphi(\eta^{-}_{b'_{1}\dots b'_{n}})  \,\,d \beta^+_{W, n, c}(c')\, \underset{n \to +\infty}{\rightarrow} \,\widecheck{\nu}_{\PP_{c}}(\varphi)$$
\end{rem.}

\bigskip

Exposons tout de suite le corollaire que nous utiliserons.

\begin{cor}\label{LAcor1}
Soit $(V, \rho)$ une représentation algébrique irréductible non bornée de $G_{c}$. Soit $c=(b,z,O, x) \in W$ un point $\beta^+_{W}$-typique. Alors pour tout $\varepsilon>0$, il existe une constante $r >0$ et un rang $n_{0}\geq 0$ tels que pour tout $n \geq n_{0}$, tout $v \in V-\{0\}$, on ait :

$$\beta^+_{W, n, c}\{c' \in W, \,\, d( \R v ,V^{-}_{b'_{1}\dots b'_{n}}) \geq r   \} > 1- \varepsilon$$

\end{cor}

\begin{proof}
La définition de  $V^{-}_{b'_{1}\dots b'_{n}}$ suppose implicitement que $V$ est muni d'un produit scalaire, induit par un bon produit scalaire sur $\text{Ind}^G_{G_{c}}(V, \rho)$. Notons $r_{0}:= \dim_{\text{prox}} \rho(G_{c})$ la dimension proximal de la représentation. Pour $\R v\in \mathbb{P}(V)$, $r>0$, on pose $B(\R v, r) := \{E \in \mathbb{G}_{r_{0}}(V), \,\,d(\R v, E^\perp) < r\}$ (bien défini car $r_{0}< \dim V$ par l'hypothèse de  non compacité). Soit $m \in \mathcal{P}(\mathbb{G}_{r_{0}}(V))$ la probabilité image de $\widecheck{\nu}_{\PP_{c}}$ par le morphisme $G_{c}$-équivariant $\PP_{c} \rightarrow \mathbb{G}_{r_{0}}(V), \xi_{0}\mapsto V^U$. Il est prouvé dans (\cite{BQRW}, 14.5) que $m$ est H\"older-régulière au sens où il existe $C,s>0$ tel que pour tout $\R v \in \Pbb(V) $, pour tout $r>0$, $m(B(\R v, r))\leq Cr^s$. En particulier, on peut choisir $r>0$ assez petit pour que $m(B(\R v , 2r))< \frac{1}{2}\varepsilon$ pour tout $\R v \in \Pbb(V)$.

\bigskip

Notons $m_{n}$ la loi de $V_{\eta^{-}_{b'_{1}\dots b'_{n}}} \in \mathbb{G}_{r_{0}}(V)$ quand $c'=(b',z',O', x') \in W$ varie suivant $\beta^+_{W, n, c}$. La loi des angles affirme que $m_{n} \rightarrow m$ quand $n \to \infty$. Il existe donc un rang $n_{0}\geq 0$ à partir duquel $m_{n}(B(\R v , r))< \varepsilon$ pour tout $\R v\in \Pbb(V)$ (raisonner par l'absurde et utiliser la compacité de $\Pbb(V)$). Comme $m_{n}(B(\R v , r)) = \beta^+_{W, n, c}\{c' \in W, \,\, d( \xi ,V^{-}_{b'_{1}\dots b'_{n}}) < r   \}$, on obtient le résultat voulu. 
\end{proof}

\bigskip

\subsubsection*{Preuve de la loi des angles}

Nous avons choisi les décompositions de Cartan des éléments de $G$ de sorte que pour tout $g \in G$ on ait  $\xi^{+}_{^tg}= \xi^{-}_{g}$. Il suffit donc de vérifier que pour $\beta^+_{W}$-presque tout $c=(b,z, O, x) \in W$, 
$$\int_{W} \varphi(\xi^{+}_{b'_{n}\dots b'_{1}})  \,\,d \beta^+_{W, n, c}(c')\, \underset{n \to +\infty}{\rightarrow} {\nu}_{\PP}(\varphi)$$

Introduisons les suites $q_{n}:= \lfloor\text{log}n\rfloor^2$, $p_{n}:= n-q_{n}$. On va montrer que

\begin{itemize}
\item[a)] $\int_{W} | \varphi(\xi^{+}_{b'_{n}\dots b'_{1}})  - \varphi(\xi^{+}_{b'_{n}\dots b'_{p_{n} +1}})| \,\,d \beta^+_{W, n, c}(c')\, \underset{n \to +\infty}{\rightarrow} 0$ 
\item[b)] $\int_{W} \varphi(\xi^{+}_{b'_{n}\dots b'_{p_{n}+1}})  \,\,d \beta^+_{W, n, c}(c')\, \underset{n \to +\infty}{\rightarrow} \,\nu_{\PP}(\varphi)$
\end{itemize}

Nous utiliserons la description des fibres du système dynamique $(B^+, \beta^+, T^+)$ faite précédemment ainsi que le théorème local limite de la section $7.2$. Quelques rappels. La $n$-fibre passant par un point $c=(b,z, O, x) \in W$ est par définition $F_{n,c}:= (T^+)^{-n}(T^+)^n(c)$ et est paramétrée par $\Gamma^n$ via une application notée $h_{n,c} : \Gamma^n \rightarrow F_{n,c}$. Le morceau de fibre $F_{n,c} \cap W$ est alors paramétré par un sous ensemble de $\Gamma^n$ que l'on note 
$$Q_{n,c} := \{a \in \Gamma^n, h_{n,c}(a) \in W\}= \{a \in \Gamma^n, \,\, \theta_{n}(aT^nb) \in U -z + \theta_{n}(b), \,\,\chi(a_{1}\dots a_{n}) \in I -t + \chi(b_{1}\dots b_{n})) \}$$
où $t$ représente la coordonnée réelle de $x \in \T^d \times \R$. Via ces identifications, la mesure conditionnelle $\beta^+_{W, n, c} \in \mathcal{M}(F_{n,c} \cap W)$ correspond  à la probabilité $\frac{1}{\mu^{\otimes n}(Q_{n,c})} \mu^{\otimes n}_{|Q_{n,c}} \in \mathcal{P}(Q_{n,c})$.  

\bigskip
Soit $c=(b,z, O, x)  \in W$ $\beta^+_{W}$-typique. 

\bigskip

\emph{Preuve de $a)$ :} Il suffit de montrer que pour tout $\varepsilon > 0$, la suite 
$$ \beta^+_{W, n, c}\{c' \in W, \,\,d(\xi^{+}_{b'_{n}\dots b'_{1}}, \xi^{+}_{b'_{n}\dots b'_{p_{n}+1}})\geq  \varepsilon\}  \underset{n \to +\infty}{\rightarrow} 0$$
 
Mais le terme général de cette suite est plus petit que $$\frac{1}{\mu^{\otimes n}(Q_{n,c})} \mu^{\otimes n}\{a \in \Gamma^n, d(\xi^+_{a_{1}\dots a_{n}}, \xi^+_{a_{1}\dots a_{q_{n}}})\geq \varepsilon \}$$

Le \cref{TLL}, combiné avec le \cref{thetan} et le fait que $\chi^\star \mu$ est centrée,  assure que le dénominateur décroît moins vite qu'une puissance de $n$ : il existe $R>0$ tel que pour $n$ assez grand, on a $\mu^{\otimes n} (Q_{n,c}) \geq n^{-R}$. Mais le \cref{convdens} sur la convergence des points de densités  implique que le numérateur décroît  infiniment plus vite : il existe $\alpha >0$ tel que $\mu^{\otimes n}\{a \in \Gamma^n, d(\xi^+_{a_{1}\dots a_{n}}, \xi^+_{a_{1}\dots a_{q_{n}}})\geq \varepsilon \} \leq e^{-\varepsilon q_{n}}$ à partir d'un certain rang. Finalement, le rapport tend vers $0$ ce qui justifie le point $a)$.

\bigskip

\emph{Preuve de $b)$ :}

\begin{align*}
&\int_{W} \varphi(\xi^{+}_{b'_{n}\dots b'_{p_{n}+1}})  \,\,d \beta^+_{W, n, c}(c')\\
=& \frac{1}{\mu^{\otimes n} (Q_{n,c})} \int_{\Gamma^n}\varphi(\xi^{+}_{a_{n}\dots a_{p_{n}+1}}) 1_{U}(z + \theta_{n}(aT^nb)-\theta_{n}(b)) 1_{I}(t+\chi(a_{1}\dots a_{n})-\chi(b_{1}\dots b_{n}))  d\mu^{\otimes n}(a)\\
=&\int_{\Gamma^{q_{n}}} \varphi(\xi^{+}_{a_{1}\dots a_{q_{n}}}) L_{n,c}(a)d\beta(a) \\
\end{align*}

où la dernière égalité s'obtient en renversant l'ordre des $a_{i}$ et en utilisant le théorème de Fubini,  la fonction $L_{n,c} : B \rightarrow [0, +\infty]$ étant définie par : 
\small{
$$L_{n,c}(a) := \frac{1}{\mu^{\otimes n} (Q_{n,c})}  \int_{\Gamma^{p_{n}}}1_{U}(z + \theta_{n}(h_{1}\dots h_{p_{n}} a_{q_{n}}\dots a_{1}T^nb)-\theta_{n}(b)) \,1_{I}(t+\chi(h_{1}\dots h_{p_{n}}a_{q_{n}} \dots a_{1})-\chi(b_{1}\dots b_{n})) \,\, d\mu^{\otimes p_{n}}(h)$$}

\large
Admettons provisoirement que la suite de fonctions $(L_{n,c})_{n \geq 1}$ converge vers $1$ dans $L^{1}(B,\beta)$ lorsque $n \to \infty$. Comme $\varphi$ est bornée, on a alors 

$$\int_{\Gamma^{q_{n}}} |\varphi(\xi^{+}_{a_{1}\dots a_{q_{n}}}) L_{n,c}(a) -  \varphi(\xi^{+}_{a_{1}\dots a_{q_{n}}})| d\beta(a) \rightarrow 0$$

Or $\int_{\Gamma^{q_{n}}}  \varphi(\xi^{+}_{a_{1}\dots a_{q_{n}}}) d\beta(a) \rightarrow \nu_{\PP}(\varphi)$ d'après le \cref{convdensmes} sur la convergence des points de densité. 

Cela permet de conclure que  
$$\int_{W} \varphi(\xi^{+}_{b'_{n}\dots b'_{p_{n}+1}})  \,\,d \beta^+_{W, n, c}(c') \rightarrow \nu_{\PP}(\varphi)$$

Il reste à prouver le lemme technique suivant :
\begin{lemme*}
La suite de fonctions $(L_{n,c})_{n \geq 1}$ converge vers $1$ dans $L^{1}(B,\beta)$ lorsque $n \to \infty$.
\end{lemme*}

\begin{proof}
On peut réécrire la formule pour $L_{n,c}(a)$ sous la forme :

\begin{align*}
L_{n,c}(a)=\frac{1}{\mu^{\otimes n} (Q_{n,c})}  \mu^{\otimes p_{n}}\,\{\,h \in \Gamma^{p_{n}},\,\,& \sigma_{F}(h_{1}\dots h_{p_{n}}a_{q_{n}}\dots a_{1}, \xi_{T^nb}) \in U -z +\theta_{n}(b),\\& \chi(h_{1}\dots h_{p_{n}}a_{q_{n}} \dots a_{1}) \in I-t +\chi(b_{1}\dots b_{n})\,\,\} 
\end{align*}
puis en utilisant la propriété de cocycle et l'équivariance des $(\xi_{b})_{b \in B}$,
\begin{align*}
L_{n,c}(a)=\frac{1}{\mu^{\otimes n} (Q_{n,c})}  \mu^{\otimes p_{n}}\,\{\,h \in \Gamma^{p_{n}},\,\,& \sigma_{F}(h_{1}\dots h_{p_{n}}, \xi_{a_{q_{n}}\dots a_{1}T^nb}) \in U -z +\theta_{n}(b) -\sigma_{F}(a_{q_{n}}\dots a_{1}, \xi_{T^nb}),\\& \chi(h_{1}\dots h_{p_{n}})\in I-t +\chi(b_{1}\dots b_{n})-\chi(a_{q_{n}} \dots a_{1}) )\,\,\} 
\end{align*}

Le théorème local limite de la section précédente permet d'estimer de fa\c con précise  les numérateurs et dénominateurs de $L_{n,c}(a)$. On a besoin toutefois de vérifier que les perturbations $|\theta_{n}(b)-n \sigma_{\mu} -\sigma_{F}(a_{q_{n}}\dots a_{1}, \xi_{T^nb})+ q_{n}\sigma_{\mu}|$ et $|\chi(b_{1}\dots b_{n})-\chi(a_{q_{n}} \dots a_{1})|$ ne croissent pas trop vite avec $n$,  plus précisément qu'elles sont toutes deux inférieures à $R \sqrt{n \text{log}n}$ pour un certain $R>0$ indépendant de $a$ (mais pouvant dépendre de $b$). On a d'emblée les majorations suivantes  (pour $n \geq 3$):

 $\bullet$ $|\chi(b_{1}\dots b_{n})| \leq R_{1}  \sqrt{n \,\text{log}\text{log}n}\,\,$ en utilisant le fait que  $\chi_{\star} \mu$ est centrée à support fini.   

  $\bullet$ $|\chi(a_{q_{n}} \dots a_{1})| \leq R_{2}q_{n}\,\,$ car $\chi_{\star} \mu$ est à support fini.

 $\bullet$ $|\theta_{n}(b)-n \sigma{\mu}| \leq R_{3} \sqrt{n \,\text{loglog}n}\,\,$ d'après le  \cref{thetan}.

 Pour la majoration de $|\sigma_{F}(a_{q_{n}}\dots a_{1}, \xi_{T^nb})- q_{n}\sigma_{\mu}|$, on restreint la variable $a \in B$ à un sous ensemble quasi-plein de $B$. Posons pour cela :
 $$E_{n}:= \{a \in B, \,\,\forall p \geq n, ||\frac{1}{q_{n}}\sigma_{F}(a_{q_{n}}\dots a_{1}, \xi_{T^nb}) - \sigma_{\mu}|| \leq 1 \}$$
 Le principe des grandes déviations (\cite{BQRW},  $13.4$) combiné à la $F$-invariance de $\sigma_{\mu}$  implique que  $\beta(E_{n}) \geq 1- e^{-\alpha n}$ pour $\alpha>0$ assez petit, $n \geq 0$ assez grand. 
 
 On a alors  $\int_{B}|L_{n,c}-1| \,d \beta \leq \int_{E_{n}}|L_{n,c}-1| \,d \beta  +  \int_{E^c_{n}}L_{n,c}\,d \beta +  \beta(E^c_{n}) $ et les termes $ \int_{E^c_{n}}L_{n,c}\,d \beta $ et $\beta(E^c_{n})$ tendent vers $0$ (pour le premier, remarquer que  $\int_{E^c_{n}}L_{n,c}\,d \beta  \leq \frac{\beta(E^c_{n})}{\mu^{\otimes n} (Q_{n,c} )}$ avec $\mu^{\otimes n} (Q_{n,c})$ décroissant moins vite qu'un $n^{-R'}$ via le théorème local limite).
 %%%%%%%%%%%%%%%%%%%%%%%%%%%%%%%%%%??????????
 Pour conclure la preuve, il suffit de montrer la convergence uniforme $|| 1_{E_{n}}L_{n,c} -1_{E_{n}}||_{\infty} \rightarrow 0$. 
 
 Notons $N_{n,c}(a)$ et $D_{n,c}$ les numérateurs et dénominateurs de $L_{n,c}(a)$. Le \cref{TLL} entraine qu'uniformément pour $a \in E_{n}$, on a 

$$N_{n,c}(a)\, 2\pi p_{n}\, e^{\frac{1}{2p_{n}} ||\theta_{n}(b)-n \sigma_{\mu} -\sigma_{F}(a_{q_{n}}\dots a_{1}, \xi_{T^nb})+ q_{n}\sigma_{\mu},\,  \chi(b_{1}\dots b_{n})-\chi(a_{q_{n}} \dots a_{1}) ||^2_{\Phi_{\mu}}}\,\,   \rightarrow \,\,\Pi_{\mu}( U\times (I-t))$$
\large
Par ailleurs, 
 $$D_{n,c}\,2\pi n \,e^{\frac{1}{2n} ||\theta_{n}(b)- n \sigma_{\mu},\chi(b_{1}\dots b_{n})||^2_{\Phi_{\mu}} }\,\,   \rightarrow \,\,\Pi_{\mu}( U\times (I-t))$$

Remarquons que $\Pi_{\mu}( U-z\times (I-t)) >0$ par définition de $\Pi_{\mu}$ et par choix de $I$ assez grand. On est donc ramené à montrer que, uniformément en $a \in E_{n}$,  on a :

{\normalsize \begin{align*}
&\frac{1}{2n} ||\theta_{n}(b)- n \sigma_{\mu},\chi(b_{1}\dots b_{n})||^2_{\Phi_{\mu}} 
-\frac{1}{2p_{n}} ||\theta_{n}(b)-n \sigma_{\mu} -\sigma_{F}(a_{q_{n}}\dots a_{1}, \xi_{T^nb})+ q_{n}\sigma_{\mu},\,  \chi(b_{1}\dots b_{n})-\chi(a_{q_{n}} \dots a_{1}) ||^2_{\Phi_{\mu}} \rightarrow 0\\
\end{align*}}
Notons pour cela 
\begin{align*}
&r_{n}:=(\theta_{n}(b)-n \sigma_{\mu},0)\,\,\,\,\,\,\,\,\,\,\,\,\,\,\, s_{n} :=(\sigma_{F}(a_{q_{n}}\dots a_{1}, \xi_{T^nb})+ q_{n}\sigma_{\mu},0)\\
&t_{n} := (0,\chi(b_{1}\dots b_{n}))\,\,\,\,\,\,\,\,\,\,\,\,\,\,\,\,\,\,u_{n} := (0,\chi(a_{q_{n}} \dots a_{1}))\\
 \end{align*}
  On a 
  $$||\theta_{n}(b)- n \sigma_{\mu},\chi(b_{1}\dots b_{n})||^2_{\Phi_{\mu}} 
= ||r_{n}||^2_{\Phi_{\mu}} + ||t_{n}||^2_{\Phi_{\mu}} + 2 \langle r_{n}, t_{n} \rangle_{\Phi_{\mu}}$$

Par ailleurs, 
\begin{align*}
&||\theta_{n}(b)-n \sigma_{\mu} -\sigma_{F}(a_{q_{n}}\dots a_{1}, \xi_{T^nb})+ q_{n}\sigma_{\mu},\,  \chi(b_{1}\dots b_{n})-\chi(a_{q_{n}} \dots a_{1}) ||^2_{\Phi_{\mu}}\\
 &= ||r_{n}+s_{n}+t_{n}+u_{n}||^2_{\Phi_{\mu}}\\
& =  \sum_{x, y\in \{r,s,t,u\}} \langle x_{n}, y_{n}\rangle\\
& = ||r_{n}||^2_{\Phi_{\mu}} + ||t_{n}||^2_{\Phi_{\mu}} + 2 \langle r_{n}, t_{n} \rangle_{\Phi_{\mu}} + q_{n}\sqrt{n \,\text{log}n }\, \psi(n,a)\\
\end{align*}
 où $\psi : \N \times B \rightarrow \R$ est bornée uniformément pour tout $n$ et $\beta$-presque tout $a$ (utiliser l'inégalité de Cauchy-Schwartz, l'équivalence des normes en dimension finie et les majorations précédentes de $r_{n}, s_{n}, t_{n}, u_{n})$. 

On en déduit donc la convergence souhaitée, concluant la preuve du lemme et de la loi des angles. 

 \end{proof}

\subsection{Divergence de la projection de Cartan}

Le \cref{LA1} qui contrôle la norme et la direction de $\rho(g)v$ fait intervenir la projection de Cartan $\kappa$, lue par le plus haut poids $\omega$ de $G_{c}$ sur $V$,  ainsi que par les racines simples $\alpha \in \Pi$ du système de racines de $\ag$. Pour  contrôler la dérive $b'_{1}\dots b'_{n} b^{-1}_{n}\dots b^{-1}_{1} u_{p}$, il convient donc de  préciser  le comportement asymptotique de $\kappa(b'_{1}\dots b'_{n})$ quand $c'=(b',z',O',x') \in W$ décrit un morceau de $n$-fibre, du point de vue de $\omega$ et des $\alpha \in \Pi$.

 \begin{lemme}\label{LA4}
 Soit $c=(b,z,O,x) \in W$ un point $\beta^+_{W}$-typique. Pour tout $\varepsilon>0$, il existe une constante $R>0$ et un rang $n_{0}\geq 0$ tels que pour tous $n \geq n_{0}$, $\xi \in \PP_{c}$, 
 $$ \beta^+_{W,n,c}\{c' \in W, ||\kappa(b'_{1}\dots b'_{n})-\sigma(b'_{1}\dots b'_{n}, \xi)|| \leq R\} \geq 1-\varepsilon$$
 \end{lemme}
 
 \begin{proof}
 Reprenons les représentations $(\rho_{\alpha}, V_{\alpha})_{\alpha \in \Pi}$ de la section $4.1$. Leurs plus hauts poids $(\omega_{\alpha})_{\alpha \in \Pi}$ forment une base de $\ag^\star$. Il suffit donc de prouver le résultat énoncé après avoir remplacé $||\kappa(b'_{1}\dots b'_{n})-\sigma(b'_{1}\dots b'_{n}, \xi)||$ par $|\omega_{\alpha}(\kappa(b'_{1}\dots b'_{n})-\sigma(b'_{1}\dots b'_{n}, \xi))|$, et ce quelque soit $\alpha \in \Pi$. On rappelle que les $V_{\alpha}$ sont munis d'un produit scalaire, restriction d'un bon produit scalaire sur leurs représentations induites.
 
 Fixons $\alpha \in \Pi$, notons $(V, \rho, \omega)=(V_{\alpha}, \rho_{\alpha}, \omega_{\alpha})$. Le \cref{rep} assure que pour $g \in G$, $\xi \in \PP_{c}$, $v\in V_{\xi}-\{0\}$, on a  $\omega(\sigma(g,\xi)) = \text{log}\frac{||\rho(g)v||}{||v||}$. Le premier point du \cref{LA1} entraine alors que : 
$$ 0\leq \omega(\kappa(g)-\sigma(g,\xi)) \leq |\text{log} \,d(V_{\xi}, V^-_{g})|$$

Le \cref{LAcor1} de la loi des angles implique  alors que pour tout $\varepsilon>0$, il existe une constante $R >0$ et un rang $n_{0}\geq 0$ tels que pour tout $n \geq n_{0}$, tout $\xi \in \PP_{c}$, on ait :
$$\beta^+_{W, n, c}\{c' \in W, \,\, |\omega(\kappa({b'_{1}\dots b'_{n}}) - \sigma(b'_{1}\dots b'_{n},\xi)| \leq R   \} > 1- \varepsilon$$
Cela conclut la preuve.

 \end{proof}

On déduit une estimation de $\omega(\kappa(b'_{1}\dots b'_{n}))$ le long des morceaux de fibres. 
 
 \begin{cor}\label{LAcor2}
 Soit $(V, \rho)$ une représentation fortement irréductible de $G$, de plus haut poids $\omega \in \ag^\star$.  Soit $c=(b,z,O,x) \in W$ un point $\beta^+_{W}$-typique. Pour tout $\varepsilon>0$, il existe une constante $R>0$ et un rang $n_{0}\geq 0$ tels que pour tous $n \geq n_{0}$, 
 $$ \beta^+_{W,n,c}\{c' \in W, |\omega(\kappa(b'_{1}\dots b'_{n})- \theta_{n}(b))| \leq R\} \geq 1-\varepsilon$$
 \end{cor}

\begin{proof}
Cela provient directement  du \cref{LA4} appliqué avec $\xi= \xi_{T^nb}$, et de la $F$-invariance du plus haut poids $\omega$.
\end{proof}

Le lemme suivant permettra de minorer les valeurs des $\alpha(\kappa(b'_{1} \dots b'_{n})), \alpha \in \Pi$, le long des morceaux de fibres.

\begin{lemme} \label{alphacart}
Soit $c=(b,z,O,x) \in W$ un point $\beta^+_{W}$-typique. Pour tout $t_{0}, \varepsilon>0$, il existe un rang $n_{0}\geq 0$ tels que pour tous $n \geq n_{0}$, 
 $$ \beta^+_{W,n,c}\{c' \in W, ||\kappa(b'_{1}\dots b'_{n})- n \sigma_{\mu}|| \leq nt_{0} \} \geq 1-\varepsilon$$

\end{lemme}

\begin{proof}
On note 
$$E_{n}:= \{g \in \Gamma,\, \, \sigma_{F}(g,\xi_{T^nb} ) \in U-z+ \theta_{n}(b), \, \chi(g) \in I-t + \chi(b_{1}\dots b_{n})\}$$
 $E := \{g \in \Gamma, \,||\kappa(g)-n \sigma_{\mu}||\leq nt_{0}\}$. D'après la description des mesures conditionnelles ($6.1$), pour $n \geq 0$, la probabilité à estimer est $ \frac{\mu^{\star n}(E_{n} \cap E)}{\mu^{\star n}(E_{n})}$. Le principe des grandes déviations pour la projection de Cartan (\cite{BQRW}, 13.6) affirme qu'il existe $c>0$ tel que pour $n$ assez grand, on a $\mu^{\star n}(E) \geq 1- e^{-cn}$. Par ailleurs, d'après le \cref{TLL}, il existe $R>0$ tel que pour $n$ assez grand, on a $\mu^{\star n} (E_{n}) \geq n^{-R}$. Finalement, à partir d'un certain rang, on a $ \frac{\mu^{\star n}(E_{n} \cap E)}{\mu^{\star n}(E_{n})} = 1- \frac{\mu^{\star n}( E_{n}\cap E^c)}{\mu^{\star n}(E_{n})} \geq 1- \frac{e^{-cn}}{n^{-R}} \underset{n \to +\infty}{\longrightarrow} 1$. D'où le résultat. 

\end{proof}

\subsection{Contrôle de la dérive}

On conclut la section par les théorèmes de contrôle annoncés, valables pour des représentations très générales. 

% La loi des angles implique que la norme et la direction du terme de dérive $F_{n,c, c_{p}}(a) = a_{1}\dots a_{n} b^{-1}_{n}\dots b^{-1}_{1} u_{p} \in \R^d$ discuté plus haut, sont contrôlées par la norme du produit $a_{1}\dots a_{n}$. Cette norme est elle même conditionnée par l'appartenance à la fenêtre $W$,  et ainsi  liée au produit $b_{1}\dots b_{n}$.  Plus précisément, en notant $\theta_{n}(b)= \sigma(b_{1}\dots b_{n}, \xi_{T^n b})$ où $\sigma$ est le cocyle d'Iwasawa (cf. section $4$), on a que $||a_{1}\dots a_{n}||$ croît au moins en $e^{\theta_{n}(b)}$. Le lemme suivant est du à Benoist-Quint (\cite{BQII}, corollary $4.12$) et décrit le comportement de $\theta_{n}(b)$. 
 
Soit $(V, \rho)$ une représentation algébrique fortement irréductible non bornée de $G$, et $\omega \in \ag^\star$ le plus haut poids de  $(V, \rho_{|G_{c}})$. On suppose $V$ muni d'un produit scalaire tel que $\rho(K) \subseteq O(V)$, $\rho(A) \subseteq \text{Sym}(V)$.

  \bigskip

\begin{itemize}

\item {\bf Contrôle de la norme.}
Soit $c=(b,z,O, x) \in W$ un point $\beta^+_{W}$-typique. Alors  pour tout $\varepsilon>0$,  il existe une constante $C>0$ et un rang $n_{0}\geq 0$ tel que pour tous $n \geq n_{0}$, $v \in V$, on ait :

$$\beta^+_{W, n, c}\{c' \in W, \,\,C^{-1} e^{\omega(\theta_{n}(b))} ||v||  \leq ||\rho(b'_{1}\dots b'_{n}) v|| \leq C e^{\omega(\theta_{n}(b))} ||v|| \} > 1- \varepsilon$$

\bigskip

\item {\bf Contrôle de la direction.} Soit $c=(b,z,O, x) \in W$ un point $\beta^+_{W}$-typique. Alors pour tout  $\varepsilon >0$ il existe une constante $\delta>0$ et un rang $n_{0} \geq 0$ tels que  pour tous $n \geq n_{0}$, $v \in V-\{0\}$, on ait : 

$$\beta^+_{W, n, c}\{c' \in W, \,\,d(\rho(b'_{1}\dots b'_{n})\R v, \,V^+_{b'_{1}\dots b'_{n}}) \leq e^{-\delta n }\} > 1- \varepsilon$$

\end{itemize}

\bigskip

\begin{proof}

Pour la norme : Soit $\varepsilon>0$. D'après le \cref{LAcor1} de la loi des angles et le \cref{LAcor2}, il existe des constantes $r>0, R>0$ et un rang $n_{0} \geq 0$ tel que pour tout $n \geq n_{0}$, $v \in V-\{0\}$ on a 
$$\beta^+_{W, n, c}\{c' \in W, \,\, d( \R v ,V^{-}_{b'_{1}\dots b'_{n}}) \geq r \,\text{ et }\, ||\omega(\kappa(b'_{1}\dots b'_{n})- \theta_{n}(b))|| \leq R   \} > 1- \varepsilon$$
Le résultat est alors une conséquence directe du premier point du  \cref{LA1}.

%Pour la norme : D'après le point $1$ du  \cref{LA1}, pour tout $c'= (b',z', \epsilon', x') \in B^+$, $n \geq 0$, $v \in \R^d$, on a $||b'_{1}\dots b'_{n} v|| \geq d(\R v, V^-_{b'_{1} \dots b'_{n}}) ||b'_{1}\dots b'_{n}|| \,||v||$. Posons $C = e^{\inf U -z}$ (où $U$ est défini juste avant la loi des angles). On en déduit que pour tout $c'=(b',z', \epsilon', x') \in W$, on a $||b'_{1}\dots b'_{n} v|| \geq C \,d(\R v, V^-_{b'_{1} \dots b'_{n}}) e^{\theta_{n}(b)} \,||v||$. 

%On utilise le corollaire de la loi des angles : Il existe $\eta >0$, $n_{0} \geq 0$ tels que pour $n \geq n_{0}$, $v \in \R^{d}-\{0\}$, on a $$\beta^+_{W, n, c}\{c' \in W, \,\, d( \R v ,V^{-}_{b'_{1}\dots b'_{n}}) \geq \eta   \} > 1- \alpha$$

%Finalement,  $\beta^+_{W, n, c}\{c' \in W, \,\, ||b'_{1}\dots b'_{n} v|| \geq C \eta e^{\theta_{n}(b)} \,||v||) >1- \alpha$. Mais d'après le dernier lemme, on a $\theta_{n}(b) \sim n\lambda_{\mu}$. D'où le résultat en choisissant $n_{0}$ assez grand. 

\bigskip

Pour la direction : 
Soit $\varepsilon >0$. D'après le \cref{LAcor1} et le \cref{alphacart}, étant donné $t_{0}>0$,  il existe des constantes $r>0, R>0$ et un rang $n_{0} \geq 3$ tel que pour tout $n \geq n_{0}$, $v \in V-\{0\}$ on a 

$$\beta^+_{W, n, c}\{c' \in W, \,\, d( \R v ,V^{-}_{b'_{1}\dots b'_{n}}) \geq r \,\text{ et }\,  ||\kappa(b'_{1}\dots b'_{n})-n \sigma_{\mu}|| \leq n t_{0}   \} > 1- \varepsilon$$

Soit $c' \in W$ dans cet ensemble. En utilisant le point $2$ du  \cref{LA1}, on obtient que 
\begin{align*}
d(\rho(b'_{1}\dots b'_{n}) \R v, V^+_{b'_{1}\dots b'_{n}}) &\leq \frac{1}{d(\R v, V^-_{b'_{1}\dots b'_{n}})} \max_{\alpha \in \Pi} e^{-\alpha ( \kappa(b'_{1}\dots b'_{n}))}\\
&\leq r^{-1}\max_{\alpha \in \Pi} e^{-n(\alpha (\sigma_{\mu}) + \alpha(x_{n,c'}))}\\
\end{align*}
 où on a noté $x_{n,c'}:=\frac{1}{n}\kappa(b'_{1}\dots b'_{n})- \sigma_{\mu}$, élément de $\ag$ de norme au plus $t_{0}$. Il est prouvé dans \cite{BQRW} (section $10.4$) que le vecteur de Lyapunov $\sigma_{\mu}$ est dans l'intérieur de la chambre de Weyl $\ag^+$, autrement dit que pour tout $\alpha \in \Pi$, on a $\alpha(\sigma_{\mu})>0$. On peut donc choisir $t_{0}>0$ tel que $||\alpha|| \leq \frac{1}{2t_{0}}\alpha(\sigma_{\mu})$ pour tout $\alpha \in \Pi$.  On a alors  
$$d(\rho(b'_{1}\dots b'_{n}) \R v, V^+_{b'_{1}\dots b'_{n}}) \leq r^{-1}e^{- \frac{n}{2}\min_{\alpha \in \Pi} \alpha(\sigma_{\mu})}$$

Cela  donne le résultat annoncé.
\end{proof}

\bigskip

\begin{rem.}
Le contrôle de la direction reste vrai pour des représentations induites à $G$ à partir de représentations irréductibles non bornées de $G_{c}$, et munies d'un bon produit scalaire. L'hypothèse de forte irréductibilité est seulement utilisée dans le contrôle de la norme, pour lequel on utilise la $F$-invariance du plus haut poids à travers le \cref{LAcor2}. 
\end{rem.}

On termine la section par un dernier corollaire qui permettra de préciser le contrôle directionnel dans la preuve de la dérive. 

\begin{cor}\label{LAcor6}
Soit $c=(b,z,O, x)\in W$   un point $\beta^+_{W}$-typique. Pour tout $\varepsilon>0$, il existe une constante $\delta>0$ et un rang $n_{0}\geq 0$ tels que, pour $n\geq n_{0}$, on a 
$$\beta^+_{W,n,c}\{c' \in W, \,\, d(b'_{1} \dots b'_{n} \xi_{T^nb}, \,\xi^+_{b'_{1} \dots b'_{n}}) < e^{-\delta n}\} > 1- \varepsilon$$
\end{cor}

\begin{proof}
En raisonnant à l'aide des $(V_{\alpha}, \rho_{\alpha})_{ \alpha \in \Pi}$ comme dans la preuve du \cref{convdens}, on voit qu'il suffit de prouver le résultat suivant  : Soit $(V, \rho)$ une représentation irréductible \emph{proximale} non bornée de $G_{c}$, et $\langle .,.\rangle$ un bon produit scalaire sur la représentation induite à $G$.  Alors pour $\beta^+_{W}$-presque tout $c=(b,z,O,x)\in W$, pour tout $\varepsilon>0$, il existe une constante $\delta>0$ et un rang $n_{0}\geq 0$ tels que, pour $n\geq n_{0}$, on a 
$$\beta^+_{W,n,c}\{c' \in W, \,\, d(\rho(b'_{1} \dots b'_{n})V_{\xi_{T^nb}}, \,\,V_{\xi^+_{b'_{1} \dots b'_{n}}}) < e^{-\delta n}\} > 1- \varepsilon$$
Comme $(V, \rho)$ est proximale,  les $V_{\xi} \in \Pbb(V)$ sont des droites. On conclut directement à partir du contrôle de la direction démontré plus haut (avec $\R v =  V_{\xi_{T^nb}}$) combiné  à la dernière remarque.
\end{proof}

%%%%%%%%%%%%%%%%%%%%%%%%%%%%%%%%%%%%%%%%%%%%%%%%%%%%%%%%%%%%%%%%%%%%%%%%%%%%%%%%%%%%%%%%%%%%%%%%%%%%%%%%%%%%%%%%%%%%%%%%%%%%%%%%%%%%%%%%%%%%%%%%%%%%%%%%%%%%%%%%%%%%%%%%%%%%%%%%%%%%%%%%%%%%%%%%%%%%%%%%%%%%%%%%%%%%%%%%%%%%%%%%%%%%%%%%%%%%%%%%%%%%%%%%%%%%%%%%%%%%%%%%%%%%%%%%%%%%%%%%%%%%%%%%%%%%%%%%%%%%%%%%%%%
%%%%%%%%%%%%%%%%%%%%%%%%%%%%%%%%%%%%%%%%%%%%%%%%%%%%%%%%%%%%%%%%%%%%%%%%%%%%%%%%%%%%%%%%%%%%%%%%%%%%%%%%%%%%%%%%%%%%%%%%%%%%%%%%%%%%%%%%%%%%%%%%%%%%%%%%%%%%%%%%%%%%%%%%%%%%%%%%%%%%%%%%%%%%%%%%%%%%%%%%%%%%%%%%%%%%%%%%%%%%%%%%%%%%%%%%%%%%%%%%%%%%%%%%%%%%%%%%%%%%%%%%%%%%%%%%%%%%%%%%%%%%%%%%%%%%%%%%%%%%%%%%%%%

\newpage

\section{Preuve de la dérive exponentielle}

Nous allons mettre ensemble les résultats des sections précédentes pour prouver le théorème de dérive exponentielle. On rappelle quelques hypothèses : $\mu \in \mathcal{P}(SL_{d}(\Z))$ est une probabilité sur $SL_{d}(\Z)$ à support fini engendrant un semigroupe $\Gamma$ fortement irréductible (donc d'adhérence de Zariski $G \subseteq SL_{d}(\R)$ semi-simple, non bornée);  $\chi : \Gamma \rightarrow \R$ est un morphisme de groupes tel que $\chi_{\star}\mu \in \mathcal{P}(\R)$ est centrée. On fait de plus les hypothèses $\mu(e)>0$ et $\chi$ non trivial pour pouvoir appliquer le théorème local limite et les théorèmes de contrôle démontrés dans les sections $7.2$ et $8.6$. La notation $\nu \in \mathcal{M}^{Rad}(\T^d \times \R)$ désigne une mesure de Radon stationnaire ergodique pour la marche induite par ($\mu$, $\chi$) sur $\T^d \times \R$, dont la projection sur $\T^d$ est sans atome. Ces données ont permis de construire un système dynamique fibré $(B^+, T^+, \beta^+)$ dans la section $4.2$. On note $\mathcal{Q}^+_{\infty}$ sa tribu queue. 

\begin{th.}[Dérive exponentielle]\label{derivexp}
Soit $Y$ un espace mesurable standard, $\sigma : B^+ \rightarrow Y$ une application $\mathcal{Q}^+_{\infty}$-mesurable, $E \subseteq B^+$ une partie $\beta^+$-pleine. 
Alors pour presque tout $c \in E$, pour tout $\varepsilon>0$, il existe $c', c'' \in E$, $t \in ]-\varepsilon, \varepsilon[^r$ non nul,  tels que $c''= \phi_{t}(c')$ et $\sigma(c)= \sigma(c')=\sigma(c'')$. 
\end{th.}

\begin{proof}
Soit $U \subseteq \ag^F$, $I \subseteq \R$ des  ouverts convexes bornés non vides, avec $I$ assez grand pour que tous ses translatés $I+t$ intersectent $\chi(\Gamma)$. Posons $W := B \times U \times O_{r}(\R)\times (\T^d \times I)\subseteq B^+$. On va montrer le résultat pour $\beta^+$-presque tout $c \in W$, ce qui est suffisant car $U$ et $I$ peuvent être choisis arbitrairement grands. Les points $c'$, $c''$ seront construits dans $W$ également. L'espace d'arrivée $Y$ étant standard, on peut supposer $Y=[0,1]$. Quitte à multiplier $\nu$ par une constante, on peut supposer $\beta^+(W)=1$. Enfin, comme  la composante sur $\ag$ est uniformément bornée dans la fenêtre $W$, il suffit de prouver le \cref{derivexp} en rempla\c cant le flot $\phi_{t}$ par le flot $\phi'_{t}(b, z, O, x)=(b, z, O, x+ \sum_{i=1}^rt_{i}e_{b,i})$ (voir section $4$ pour la définition des $e_{b,i}$). 

Dans un premier temps, introduisons des compacts $K_{0}, K_{1} \subseteq W$ vérifiant de bonnes propriétés. Soit $\alpha>0$ (petit). D'après le théorème de Lusin, il existe un compact $K_{0} \subseteq W$ de mesure $\beta^+_{W}(K)> 1- \frac{\alpha^2}{2}$ sur lequel les restrictions $\sigma_{|K_{0}} : K_{0} \rightarrow Y$, $\xi_{|K_{0}} : K_{0} \rightarrow \mathscr{P}, \,c=(b,z,O, x) \mapsto \xi_{b}$ sont continues. D'après le théorème d'équirépartition des morceaux de fibres (plus précisément le \cref{nonacc}), il existe un compact $K_{1} \subseteq W$ de mesure $\beta^+_{W}(K_{1}) > 1- \alpha$, et un rang $N_{0}\geq 0$ tels que pour tout  $c\in K_{1}$, $n \geq N_{0}$, on a 
 $$ \beta^+_{W,n,c}(K_{0}) > 1- \alpha$$
 
D'après le théorème de Lusin, quitte à réduire $K_{1}$, on peut supposer que l'application $\sigma_{|K_{1}} : K_{1} \rightarrow Y$ est continue. On montre de plus que la non-dégénérescence des mesures limites $\nu_{b}$ (\cref{cornondeg}) assure que $\beta^+_{W}$-presque  tout $c=(b,z,O, x)\in K_{1}$ est limite d'une suite $(c_{p})_{p \geq 1}$ de points de $K_{1}$ de la forme $c_{p}=(b,z,O, x +(u_{p},t_{p}))$ où $u_{p}\in \R^d$, $t_{p}\in \R$, $u_{p}$ agit sur la coordonnée en $\T^d$, a sa projection sur $\T^d$ en dehors de l'ensemble stable $W_{b}(0)$, et $u_{p}, t_{p}$ tendent vers $0$ quand $p$ tend vers l'infini. En effet,  remarquons que la mesure $\beta^+_{W |K_{1}}$ s'écrit 
$$\beta^+_{W|K_{1}}=\int_{B\times \ag^F \times O_{r}(\R)} \delta_{(b,z,O)}\otimes \nu_{b|K^{(b,z,O)}_{1}} \,d\beta(b)d\leb_{\ag^F}(z)dh(O)$$
 où $K^{(b,z,O)}_{1} :=\{x\in X, (b,z,O,x)\in K_{1}\}$. Ainsi, pour $\beta^+_{W}$-presque  tout $c=(b,z,O, x)\in K_{1}$, on a $x$ dans le support de $\nu_{b|K^{(b,z,O)}_{1}}$, autrement dit, pour tout voisinage $V$ de $(0,0)$ dans $\T^d\times \R$, on a $\nu_{b|K^{(b,z,O)}_{1}}(x+V)>0$.  En remarquant que $W_{b}(x)=x+W_{b}(0)$, le \cref{cornondeg} implique que 
 $$\nu_{b|K^{(b,z,O)}_{1}}(x+(V- W_{b}(0)\times \R))>0$$
  On obtient donc la suite $(c_{p})_{p\geq 1}$ annoncée en se donnant une base décroissante de voisinages $(V_{p})_{p\geq 1}$ de $(0,0)$ dans $\T^d\times \R$, en choisissant $x_{p}$ dans $K^{(b,z,O)}_{1}\cap x+(V_{p}- W_{b}(0)\times \R)$ puis en posant $c_{p}:=(b,z,O, x_{p})$. 

Fixons un point $c \in K_{1}$ admettant une telle approximation $(c_{p}) \in K^{\N^\star}_{1}$. On va construire les points $c', c''$ du théorème de dérive exponentielle comme limites de points pris dans les fibres passant par $c$ et les $c_{p}$. Pour $c'=(b',z', O', x') \in W$, notons $$D_{n}(c') = b'_{1}\dots b'_{n} b^{-1}_{n}\dots b^{-1}_{1} \in SL_{d}(\Z)$$ Le calcul explicite des mesures conditionnelles par rapport aux morceaux de fibres montre que lorsque $c'=(b',z', O', x')$ varie dans $W$ selon la loi $\beta^+_{W,n,c}$, le translaté $c''= c' + (D_{n}(c')u_{p},t_{p}):= (b',z', O', x'+(D_{n}(c')u_{p},t_{p}))$ varie selon la loi $\beta^+_{W,n,c_{p}}$. En particulier, on a 
$$ \beta^+_{W,n,c}\{c' \in W, \,\,c' + (D_{n}(c')u_{p},t_{p}) \in K_{0} \} >1- \alpha$$

On applique maintenant les théorèmes de contrôle donnés par la loi des angles pour la représentation canonique de $G$ sur $V =\R^d$. Quitte à réduire $K_{1}$ dès le début, cela donne :

\begin{lemme*}
 
\begin{itemize}
\item \emph{(Contrôle de la norme) }Soit $\varepsilon_{2}>0$. Il existe $\varepsilon_{1} \in ]0, \varepsilon_{2}[$, $p_{0} \geq 0$ tel que pour tout $p \geq p_{0}$, il existe un entier $n_{p} \geq 1$ pour lequel 
$$\beta^+_{W,n_{p},c}\{c' \in W, \,\,||D_{n_{p}}(c')u_{p}|| \in [\varepsilon_{1}, \varepsilon_{2}]\} >1- \alpha$$

\item \emph{(Contrôle de la direction) }Pour toute constante $\delta>0$, il existe un rang $N_{1}\geq 0$ tel que pour $n \geq N_{1}$, pour tout $p \geq 0$, on a 
$$\beta^+_{W,n,c}\{c' \in W, \,\, d(D_{n}(c') \R u_{p}, V_{b'_{1} \dots b'_{n} T^nb}) < \delta\} > 1- \alpha$$
\end{itemize}

\end{lemme*}

\begin{proof}
Reportée plus bas.
\end{proof}

Concluons la preuve. On se donne $\varepsilon_{2}> \varepsilon_{1}$ et $p_{0}\geq 0$ comme dans le lemme précédent. Quitte à tronquer  la suite $(c_{p})$, on peut supposer que $p_{0}=0$. Par ailleurs, la suite $n_{p}$ tend vers $+\infty$. Quitte à tronquer davantage, on peut donc supposer $n_{p} \geq N_{0}$ pour tout $p\geq 1$. On peut de plus choisir $\alpha< 1/4$ dès le début de la preuve. D'après ce qui précède, on peut alors se donner pour tout $p \geq 1$, un point $c'_{p} \in K_{0}$ de la $n_{p}$ fibre passant par $c$ tel que :
\begin{itemize}
\item[-] $c''_{p}:= c'_{p} + (D_{n_{p}}(c'_{p})u_{p},t_{p}) \in K_{0} \cap F_{n_{p},c_{p}}$
\item[-] $||D_{n_{p}}(c'_{p})u_{p}|| \in [\varepsilon_{1}, \varepsilon_{2}]$
\item[-] $d(D_{n_{p}}(c'_{p}) \R u_{p}, V_{b'^{(p)}_{1} \dots b'^{(p)}_{n_{p}} T^{n_{p}}b}) < \delta_{p}$
 \text{ avec $\delta_{p} \rightarrow 0$} (en notant $c'_{p}= (b'^{(p)}_{1} \dots b'^{(p)}_{n_{p}} T^{n_{p}}b, \dots)$)
\end{itemize} 
Par compacité de $K_{0}$ et quitte à extraire, on peut supposer qu'on a les convergence   $(c'_{p}) \rightarrow c' $ et $(c''_{p}) \rightarrow c''$ avec $c', c'' \in K_{0}$.  

On a les égalités  $\sigma(c)=\sigma(c')= \sigma(c'')$. En effet, comme l'application $\sigma$ est $\mathcal{Q}^+_{\infty}$-mesurable, elle est en particulier constante sur les fibres. On a ainsi $\sigma(c)= \sigma(c'_{p})$ et $\sigma(c_{p})= \sigma(c''_{p})$ pour $p\geq 1$. Comme l'application $\sigma$ est continue sur $K_{0}$ et $K_{1}$, on a finalement  $\sigma(c')= \lim \sigma(c'_{p}) = \sigma(c)$, et $\sigma(c'')= \lim \sigma(c''_{p}) = \lim \sigma(c_{p})= \sigma(c)$. 

Enfin, quitte à extraire on peut supposer que la dérive $D_{n_{p}}(c'_{p})u_{p}$ converge vers un vecteur $v \in \R^d$ de norme $||v|| \in [\varepsilon_{1}, \varepsilon_{2}]$. Vérifions qu'alors $c'$ et $c''$ sont liés par l'action du flot. Ecrivons $c'= (b',z', O', x')$. La convergence de $c'_{p}$ vers $c'$ implique que  $b'^{(p)}_{1} \dots b'^{(p)}_{n_{p}} T^{n_{p}}b \rightarrow b'$. La continuité de l'application $\xi$ sur $K_{0}$ et le  contrôle directionnel de la dérive entrainent alors que $v \in V_{b'}$. Finalement, en utilisant que la suite $(t_{p})_{p\geq 1}$ tend vers $0$, on obtient  $c''= c'+(v,0)$ avec $v \in V_{b'}$, de norme dans $[\varepsilon_{1}, \varepsilon_{2}]$. Comme $\varepsilon_{2}$ est arbitrairement petit, cela prouve la dérive exponentielle pour presque tout point de $K_{1}$. Comme $\alpha$ peut être choisi arbitrairement petit, la preuve est terminée.
\end{proof}

\bigskip

\bigskip

Il reste à démontrer le lemme utilisé dans la preuve de la dérive exponentielle.

\begin{proof}[Démonstration du lemme]
Vérifions le premier point. Notons $\omega \in \ag^\star$ le plus haut poids de $G_{c}$ sur $V= \R^d$. D'après le contrôle de la norme, il existe une constante $C>0$ et un rang $N'_{0}\geq 0$ tel que pour tous $n \geq N'_{0}$, $p\geq 1$ on a :
$$\beta^+_{W, n, c}\{c' \in W, \,\,C^{-1} e^{\omega(\theta_{n}(b))} ||b^{-1}_{n}\dots b^{-1}_{1}u_{p}||  \leq ||D_{n}(c')u_{p}|| \leq C e^{\omega(\theta_{n}(b))} ||b^{-1}_{n}\dots b^{-1}_{1}u_{p}|| \} > 1- \alpha$$
Il suffit donc de montrer que si $\varepsilon_{1} \in ]0, \varepsilon_{2}[$ est assez petit, et $p\geq 0$ assez grand, alors il existe $n \geq N'_{0}$ tel que l'intervalle $[s_{n}, t_{n}]:=e^{\omega(\theta_{n}(b))} ||b^{-1}_{n}\dots b^{-1}_{1}u_{p}||\,[C^{-1}, C]$  est inclus dans $[\varepsilon_{1}, \varepsilon_{2}]$. On note $R= \max\{||g||+ ||g^{-1}||, \,\,g \in \text{supp}\mu\}$. On choisit $\varepsilon_{1}>0$ tel que $\varepsilon_{2}/\varepsilon_{1}> C^2 R^2$. On se donne $p_{0} \geq 0$ tel que pour $p \geq p_{0}$ on a $||u_{p}|| < R^{-2N'_{0}} \varepsilon_{1}$. Soit $p\geq p_{0}$, et $n_{p}$ le premier entier tel que $s_{n_{p}} > \varepsilon_{1}$. Un tel entier existe bien car la suite $(s_{n})$ n'est pas bornée (par la condition $u_{p} \notin W_{b}(0)$ dans $\T^d$ et la positivité du premier Lyapunov sur $V$ qui donne $\omega(\theta_{n}(b)) \rightarrow +\infty$ via le \cref{thetan}). Comme $s_{n+1}/s_{n} \leq R^2$ pour tout $n \geq 1$, on a de plus $n_{p} \geq N'_{0}$ et $s_{n_{p}} \leq \varepsilon_{1} R^2$. Comme $t_{n}/s_{n} \leq C^2$ pour tout $n \geq 1$, on a $t_{n} \leq \varepsilon_{1} R^2 C^2 \leq \varepsilon_{2}$ d'où l'inclusion $[s_{n_{p}}, t_{n_{p}}] \subseteq [\varepsilon_{1}, \varepsilon_{2}]$.  D'où le premier point.

\bigskip

Vérifions le second point. Soit $\delta>0$. D'après le contrôle de la direction donné par la loi des angles (voir $8.6$), il existe un rang $N \geq 0$ tel que pour $n \geq N$, pour tout $p \geq 0$, on a 
$$\beta^+_{W,n,c}\{c' \in W, \,\, d(D_{n}(c') \R u_{p}, V_{\xi^+_{b'_{1} \dots b'_{n}}}) < \delta/2\} > 1- \frac{\alpha}{2}$$
Il reste à vérifier que $ V_{\xi^+_{b'_{1} \dots b'_{n}}}$ est proche de $V_{b'_{1} \dots b'_{n} T^nb}$ pour $n$ grand. Soit $r:= \dim_{\text{prox}}G$ la dimension proximal de $G \subseteq SL_{d}(\R)$. On voit $\mathbb{G}_{r}(V)$ comme un sous-ensemble des parties fermées de $\Pbb(V)$ et on le munit de la distance de Hausdorff. On munit $\mathscr{P}$ de la distance introduite dans la section $4$. L'application $\mathscr{P} \rightarrow \mathbb{G}_{r}(V), \xi \mapsto V_{\xi}$ est continue sur $\PP$ compact, donc uniformément continue. Il existe ainsi $\delta'>0$ tel que pour tous $\xi_{1}, \xi_{2} \in \mathscr{P}$ à distance $d(\xi_{1}, \xi_{2})< \delta'$, on a $d(V_{\xi_{1}}, V_{\xi_{2}})<\delta/2$. Il suffit donc de montrer qu'il existe un rang $N'\geq 0$ tel que pour $n \geq N'$, on a 
$$\beta^+_{W,n,c}\{c' \in W, \,\, d(\xi^+_{b'_{1} \dots b'_{n}}, \xi_{b'_{1} \dots b'_{n} T^nb}) < \delta'\} > 1- \frac{\alpha}{2}$$
C'est une application directe du \cref{LAcor6} du contrôle de la direction dans la section précédente.

%Mais cette mesure n'est autre que $$\frac{\beta\{b' \in B, \,\, (b'_{1}, \dots, b'_{n}) \in Q_{n,c}, \,d(\xi^+_{b'_{1} \dots b'_{n}}, \xi_{b'_{1} \dots b'_{n} T^nb}) < \delta'\}}{\beta\{b' \in B, \,\, (b'_{1}, \dots, b'_{n}) \in Q_{n,c}\}}$$

%Elle est minorée par 
%$$1 -\frac{\beta\{b' \in B, \, \,d(\xi^+_{b'_{1} \dots b'_{n}}, \xi_{b'}) \geq \delta'/2\}+\beta\{b' \in B, \, \,d(\xi_{b'}, \xi_{b'_{1} \dots b'_{n} T^nb}) \geq \delta'/2\}}{\beta\{b' \in B, \,\, (b'_{1}, \dots, b'_{n}) \in Q_{n,c}\}}$$
%Le théorème local limite assure que $d_{n} :=\beta\{b' \in B, \,\, (b'_{1}, \dots, b'_{n}) \in Q_{n,c}\} \geq  n^{-R}$ pour un certain $R>0$, et $n$ assez grand. Le \cref{LA2}  assure que $s_{n} := \beta\{b' \in B, \, \,d(\xi^+_{b'_{1} \dots b'_{n}}, \xi_{b'}) \geq \delta'/2\}$ décroît à vitesse exponentielle avec $n$.  De même pour $t_{n} := \beta\{b' \in B, \, \,d(\xi_{b'}, \xi_{b'_{1} \dots b'_{n} T^nb}) \geq \delta'/2\}$ en utilisant le \cref{LA4} et l'équivariance des drapeaux limites. Finalement le rapport $\frac{s_{n}+t_{n}}{d_{n}}$ tend vers $0$ quand $n$ tend vers $+\infty$, ce qui prouve le second point. 
\end{proof}

%%%%%%%%%%%%%%%%%%%%%%%%%%%%%%%%%%%%%%%%%%%%%%%%%%%%%%%%%%%%%%%%%%%%%%%%%%%%%%%%%%%%%%%%%%%%%%%%%%%%%%%%%%%%%%%%%%%%%%%%%%%%%%%%%%%%%%%%%%%%%%%%%%%%%%%%%%%%%%%%%%%%%%%%%%%%%%%%%%%%%%%%%%%%%%%%%%%%%%%%%%%%%%%%%%%%%%%%%%%%%%%%%%%%%%%%%%%%%%%%%%%%%%%%%%%%%%%%%%%%%%%%%%%%%%%%%%%%%%%%%%%%%%%%%%%%%%%%%%%%%%%%%%%
%%%%%%%%%%%%%%%%%%%%%%%%%%%%%%%%%%%%%%%%%%%%%%%%%%%%%%%%%%%%%%%%%%%%%%%%%%%%%%%%%%%%%%%%%%%%%%%%%%%%%%%%%%%%%%%%%%%%%%%%%%%%%%%%%%%%%%%%%%%%%%%%%%%%%%%%%%%%%%%%%%%%%%%%%%%%%%%%%%%%%%%%%%%%%%%%%%%%%%%%%%%%%%%%%%%%%%%%%%%%%%%%%%%%%%%%%%%%%%%%%%%%%%%%%%%%%%%%%%%%%%%%%%%%%%%%%%%%%%%%%%%%%%%%%%%%%%%%%%%%%%%%%%%

\newpage

\section{Conclusion}

On termine la preuve du \cref{TH} en prouvant le \cref{THnonat} énoncé dans la section $4$ :

\begin{th.*}

Soit  $\mu \in \mathcal{P}(SL_{d}(\Z))$ une probabilité à support fini engendrant un sous semi-groupe $\Gamma \subseteq SL_{d}(\Z)$ fortement irréductible, et soit $\chi : \Gamma \rightarrow \R$ un morphisme de semi-groupes tel que la probabilité $\chi_{\star}\mu \in \mathcal{P}(\R)$ est centrée. Alors toute mesure de Radon $\mu$-stationnaire ergodique $\nu$ sur $\T^d \times \R$ dont la projection $\nu(. \times \R) \in \mathcal{M}(\T^d)$ est sans atome est (à translation près) une mesure de Haar sur $\T^d \times \widebar{\chi(\Gamma)}$.

\end{th.*}

Notons $X:= \T^d \times \R$. On se donne une décomposition  de $\nu$ en mesures limites $(\nu_{b})_{b \in B} \in \mathcal{M}^{Rad}(X)^B$ (cf. section $2$). On a en particulier l'égalité  $\nu = \int_{B}\nu_{b}\,d\beta(b)$. Nous allons montrer que les $\nu_{b}$ sont invariants par translation de leur coordonnée en $\T^d$. C'est alors aussi le cas de $\nu$. On conclura ensuite en utilisant que la projection de $\nu$ sur $\R$ est une mesure de Radon $\chi(\Gamma)$-invariante. On note toujours $r:= \dim_{\text{prox}} \Gamma \in \llbracket 1, d-1 \rrbracket$ la dimension proximale de $\Gamma \subseteq SL_{d}(\R)$.

Soit $b \in B$, $V_{b} \subseteq \R^d$ le $r$-plan limite associé (cf. section $4$). On désintègre la mesure de Radon $\nu_{b}\in \mathcal{M}^{Rad}(X)$ sous l'action de $V_{b}$ par translation de la coordonnée en $\T^d$. On obtient une application mesurable $\sigma_{b} : X \rightarrow \mathcal{M}^{Rad}_{1}(V_{b}),\, x \mapsto \sigma_{b,x}$. On voudrait montrer que les $\sigma_{b,x}$ sont  invariants par une même droite de $V_{b}$. On aurait alors que la mesure $\nu_{b}$ est invariante sous l'action de cette droite. La forte irréductibilité de $\Gamma$ sur $\R^d$ permettrait alors de conclure que $\nu_{b}$ est $\T^d$-invariante pour presque tout $b$ (voir plus bas). Pour  se ramener à ce cas là, on note pour $x\in X$, $V_{b,x}:= [\text{Stab}_{V_{b}}\sigma_{b,x}]_{0}$  la composante neutre du stabilisateur de $\sigma_{b,x}$ dans $V_{b}$ (agissant par translation sur les mesures projectives de $\mathcal{M}^{Rad}_{1}(V_{b})$). C'est un élément de $\mathbb{G}(V_{b})$ ensemble des sous-espaces vectoriels de $V_{b}$. On désintègre $\nu_{b}$ par rapport à l'application $\pi_{b} : X \rightarrow \mathbb{G}(V_{b}), x \mapsto V_{b,x}$ et on obtient une famille mesurable $(\nu_{b,x})_{x \in X} \in \mathcal{M}^{Rad}(X)^X$ $\nu_{b}$-presque partout caractérisée par :
\begin{itemize}
\item $\nu_{b} = \int_{X}\nu_{b,x}\,\,d\nu_{b}(x)$
\item Pour tout $x \in X$, $\nu_{b,x}$ est concentré sur $\pi^{-1}_{b}(V_{b,x}) \subseteq X$.
\end{itemize}
 
\begin{lemme} \label{Conc1}
Soit $b \in B$. Pour $\nu_{b}$-presque tout $x\in X$, la mesure $\nu_{b,x}$ est $V_{b,x}$-invariante.  
\end{lemme}

\begin{proof} Remarquons que l'action de $V_{b}$ sur $X= \T^d\times \R$ ne modifie pas la coordonnée réelle. Il suffit donc de montrer que pour tout intervalle borné $I \subseteq \R$,  pour $\nu_{b}$-presque tout $x\in X$, la mesure finie $\nu_{b,x|\T^d \times I}$ est $V_{b,x}$-invariante.  
 Cela est démontré dans (\cite{BQI}, proposition $4.3$). 
\end{proof}

Le théorème de dérive exponentielle est utilisé pour obtenir le résultat suivant.
\begin{lemme}\label{Conc2}
Pour $\beta$-presque tout $b \in B$, $\nu_{b}$-presque tout $x \in X$, on a $V_{b,x}\neq \{0\}$.
\end{lemme}

Rappelons d'abord un lemme général  les mesures conditionnelles. On dira qu'une action d'un groupe topologique $G$ sur un ensemble $Z$ est à stabilisateurs discrets si pour tout $z\in Z$, le fixateur $G_{z}:=\{g\in G, g.z=z\}$ est discret dans $G$.

\begin{lemme}\label{Conc3}
Soit $Z$ un espace borélien standard, $m$ une mesure $\sigma$-finie sur $Z$, $G$ un groupe localement compact à base dénombrable muni d'une action  sur $Z$ qui est mesurable à stabilisateurs  discrets. De même soit $(Z', m', G')$.  On se donne $f : Z \rightarrow Z'$ une bijection bimesurable telle que $f_{\star}m=m'$ et $\phi : G \rightarrow G'$ un isomorphisme de groupes  topologiques tel que pour tout $z \in Z$, $g\in G$, on a $f(g.z)= \phi(g).f(z)$. 

Alors,  si $\sigma : Z \rightarrow \mathcal{M}_{1}(G)$,  $\sigma' : Z' \rightarrow \mathcal{M}_{1}(G)'$  sont des décomposition de $m$ et $m'$ en mesures conditionnelles pour les actions de $G$ et $G'$, on a pour $m$-presque tout $z \in Z$ que 
$$ \sigma'(f(z)) = \phi_{\star} \sigma(z) $$
\end{lemme}

\begin{proof}
Cela découle de l'unicité des mesures conditionnelles. Voir \cite{Duf} pour plus de détails.
\end{proof}

\begin{proof}[Démonstration du \cref{Conc2}]
On reprend les notations de la section $4$. Pour $\beta \otimes \text{leb}_{\ag}\otimes h$-presque tout $(b,z,O) \in B\times \ag^F\times O_{r}(\R)$,  on a que $(\star)$ la famille de mesures projectives $(\sigma(b,z,O, x))_{x \in X} \in \mathcal{M}^{Rad}_{1}(\R^r)$ est une décomposition de $\delta_{(b,z,O)}\otimes \nu_{b} \in \mathcal{M}^{Rad}((b,z,O)\times X)$ par rapport à l'action de $\R^r$-flot $(\phi_{t})_{t \in \R^r}$. Soit $b \in B$ tel qu'il existe $(z,O)$ vérifiant $(\star)$. Notons $A$ l'isomorphisme vectoriel $A : \R^r \rightarrow V_{b}, t \mapsto  \langle t, e^{\omega(z)}O e_{b} \rangle'$. On a ainsi $\phi_{t}(b,z,O, x) = (b,z,O, x+A(t))$. On  déduit de \cref{Conc3} que pour $\nu_{b}$-presque tout $x \in X$, on a $A_{\star} \sigma(b,z,O, x) = \sigma_{b,x}$. Quitte à bien choisir $(b,z,O)$ dès le départ, le théorème de dérive exponentielle donne que pour $\nu_{b}$-presque tout $x \in X$,   la mesure projective $\sigma(b,z,O, x)$ est invariante selon une droite de $\R^r$ (\cref{corderiv}). On en déduit que $\sigma_{b,x}$ est invariante par une droite $V_{b}$ , i.e. $V_{b,x}\neq \{0\}$.
\end{proof}

\begin{cor} \label{Conccor}
Pour $\beta$-presque tout $b \in B$, $\nu_{b}$-presque tout $x \in X$, on a $V_{b,x}= \T^d$.
\end{cor}

Pour prouver ce corollaire, nous utiliserons le lemme général suivant :
\begin{lemme}\label{Conc4}
Soit $\mu \in \mathcal{P}(SL_{d}(\Z))$ une probabilité dont le support engendre un semi-groupe $\Gamma$ fortement irréductible sur $\R^d$. Soit $\mathcal{F}$ l'ensemble (dénombrable) des sous-groupes fermés connexes non nuls du tore $\T^d$.  Alors $\mathcal{F}$ n'admet qu'une seule probabilité $\mu$-stationnaire : la masse de Dirac $\delta_{\T^d}$.
\end{lemme}

\begin{proof}
On raisonne par l'absurde en supposant qu'il existe $m \in \mathcal{P}(\mathcal{F})$ une probabilité $\mu$-stationnaire sur $\mathcal{F}$ telle que $m(\{\T^d\})< 1$. Quitte à restreindre $m$ à $\mathcal{F}-\{\T^d\}$ ($\Gamma$-stable) puis normaliser, on peut même supposer $m(\{\T^d\})=0$. On note $\alpha = \max_{u \in \mathcal{F}}(m(u))$ et $\mathcal{F}' = \{u \in \mathcal{F}, m(u)=\alpha\}$. Pour $u \in \mathcal{F}'$, on pose $V_{u} \subseteq \R^d$ le sous espace vectoriel se projetant sur $u$. Alors  $\{V_{u}, u \in \mathcal{F}'\}$ est un ensemble fini $\Gamma$-invariant de sous-espaces vectoriels stricts non triviaux de $\R^d$. Cela contredit la forte irréductibilité de $\Gamma$ sur $\R^d$.
\end{proof}

\begin{proof}[Démonstration du \cref{Conccor}]

Notons $B^X:= B \times X$, $\beta^X:= \int_{B}\delta_{b}\otimes \nu_{b}\,d\beta(b) \in \mathcal{M}^{Rad}(B^X)$. 

Remarquons que pour $\beta^X$-presque tout $(b,x) \in B^X$, $\mu$-presque tout $g \in \Gamma$, on a $g V_{b,x}= V_{gb,gx}$. En effet, pour $\beta$-presque tout $b \in B$, $\mu$-presque tout $g \in \Gamma$, on a  $g_{\star}\nu_{b}= \nu_{gb}$ et $gV_{b} = V_{gb}$. Le \cref{Conc3} implique alors que pour $\nu_{b}$-presque tout $x \in X$, on a  $g_{\star} \sigma_{b,x}=  \sigma_{gb,gx}$, ce qui prouve l'affirmation.     

Considérons l'application $\Phi : B^X \rightarrow \mathcal{F}, (b,x) \mapsto V_{b,x}$. Elle est presque sûrement $\Gamma$-équivariante d'après ce qui précède.  La mesure de Radon $\mu$-stationnaire $\beta^X$ a donc pour image une mesure $\mu$-stationnaire $\Phi_{\star} \beta^X \in \mathcal{M}(\mathcal{F})$. On montre que cette mesure est concentrée sur $\T^d$. 

\bigskip

\emph{$1$er cas :  $\chi(\Gamma) \subseteq \R$ est discret}. C'est alors un sous groupe de $\R$ (cf. \cref{ref0}).  L'ergodicité de $\nu$ implique que $\nu$, et donc les $\nu_{b}$, sont portés par un ensemble de la forme $\T^d \times (r + \chi(\Gamma))$ où $r \in \R$. On peut supposer que c'est $\T^d \times \Z$.  Notons $\tau : B \rightarrow \N \cup \{\infty\}, b \mapsto \inf\{n \geq 1, \chi(b_{1}\dots b_{n})=0\}$ (fini $\beta$-ps) le temps de retour au bloc de départ, et $\mu_{\tau}\in \mathcal{P}(\Gamma)$ la loi de $b_{1}\dots b_{\tau(b)}$ quand $b$ varie selon $\beta$. D'après la section $2$, la mesure $\beta^X$ est $\mu_{\tau}$-stationnaire. Soit $N \geq 0$ un entier. Comme la $\mu_{\tau}$-marche stabilise les blocs $(\T^d \times \{k\})_{k \in \Z}$, la mesure finie  $\beta^X_{|\T^d \times [-N,N]}$ est également $\mu_{\tau}$-stationnaire sur $B^X$. Par ailleurs,  le semi-groupe $\Gamma_{0}$ engendré par le support de $\mu_{\tau}$ est fortement irréductible sur $\R^d$ (\cref{0irr}). Le   \cref{Conc4} implique donc que la mesure image $\Phi_{\star} \beta^X_{|\T^d \times [-N,N]}$  sur $\mathcal{F}$ est concentrée sur $\T^d$. Comme $N$ est arbitraire, c'est finalement le cas de $\beta^X$. 

\bigskip

\emph{$2$-ème cas : $\chi(\Gamma) \subseteq \R$ est dense}. La preuve est essentiellement la même que celle donnée dans la sous-section $(6.2)$ où on montre la non-dégénérescence des mesures limites. On rappelle les grandes lignes. On fixe $(\beta^X_{b})_{b \in B}$ une décomposition de $\beta^X$ en mesures limites ainsi qu'une constante $r>0$, et pour tout $I \subseteq \R$  intervalle de longueur $r$,  et on note $\beta^X_{b,I} := \beta^X_{b | B \times \T^d \times I}$. On montre que les mesures finies $(\Phi_{\star}(\beta^X_{b,I}))_{b,I} \in \mathcal{M}^f(\mathcal{F})$ ont presque-toutes les mêmes poids (comptés avec multiplicités, voir \cref{ref12}). On en déduit que pour presque tout $b \in B$, la mesure $\Phi_{\star}(\beta^X_{b,I})$ ne dépend pas de $I$ à translation près, puis que la mesure finie $m :=\Phi_{\star}\beta^X_{|B \times \T^d \times I} =\int_{B}\Phi_{\star}(\beta^X_{b,I}) \,d\beta(b)$ est $\mu$-stationnaire. Par le \cref{Conc4}, elle donc concentrée sur $\T^d$ ce qui conclut étant donné que $I$ est de taille arbitraire. 
\end{proof}

\bigskip

\begin{cor*}
Pour $\beta$-presque tout $b \in B$, $\nu_{b}$ est $\T^d$-invariante.
\end{cor*}

\begin{proof}
Cela découle du \cref{Conc1}, du \cref{Conccor} et de l'expression de $\nu_{b}$ comme moyenne intégrale des $\nu_{b,x}$.
\end{proof}

\begin{proof}[Preuve du \cref{THnonat}]
D'après le corollaire précédent, les $\nu_{b}$ sont invariants par translation quelconque de la coordonée en $\T^d$.  C'est donc aussi le cas de $\nu$ car $\nu = \int_{B} \nu_{b}\,d\beta(b)$.

\bigskip

Supposons $\chi(\Gamma) \subseteq \R$ discret. C'est alors un sous groupe de $\R$ (cf. \cref{ref0}).  L'ergodicité de $\nu$ implique que $\nu$ est portée par un ensemble de la forme $\T^d \times (r + \chi(\Gamma))$ où $r \in \R$. On peut supposer que c'est $\T^d \times \Z$. La mesure $\nu$ se décompose sur les blocs $\nu= \sum_{k \in \Z} \nu_{|\T^d\times \{k\}}$ et chaque mesure $\nu_{|\T^d\times \{k\}}$ est $\T^d$-invariante donc une mesure de Haar.  D'après le deuxième point du \cref{premcor},  les masses totales des $\nu_{|\T^d\times \{k\}}$ sont les mêmes, ce qui conclut. 

\bigskip

Supposons que  $\chi(\Gamma) \subseteq \R$ est dense. Pour $I \subseteq \R$ mesurable, la mesure $\nu(. \times I) \in \mathcal{M}^f(\T^d)$ est $\T^d$-invariante donc $SL_{d}(\Z)$-invariante. La relation de $\mu$-stationnarité  donne alors que pour $A \subseteq \T^d, I \subseteq \R$ mesurables, on a : $\nu(A \times I) = \int_{G}\nu(A \times [I- \chi(g)]) \, d\mu(g)$. La mesure $\nu(A \times .) \in \mathcal{M}^{Rad}(\R)$ est donc $\chi_{\star} \mu$-stationnaire, i.e.  un multiple de la mesure de lebesgue (cf. \cref{ref0}). Finalement, $\nu(A+z, I+t)= \nu(A \times I)$ pour tout $z \in \T^d$, $t \in \R$,  d'où le résultat.

\end{proof}

\bigskip

\bigskip

\appendix
\newpage

\section{Mesures de Radon limites}\label{AnnexeA}

On définit les mesures limites associées à une mesure de Radon stationnaire et on étudie quelques unes de leurs propriétés. 

\bigskip

Soit $X$ un espace topologique localement compact à base dénombrable, $G$ un groupe localement compact à base dénombrable agissant continument sur $X$, $\mu \in \mathcal{P}(G)$ une probabilité sur $G$, $\nu \in \mathcal{M}^{\text{Rad}}(X)$ une mesure de Radon $\mu$-stationnaire sur $X$. On note $B := G^{\N^\star}$, $\mathcal{B}$ sa tribu produit,  $\beta := \mu^{\otimes \N^\star}$, et pour $n \geq 0$, $\mathcal{B}_{n} \subseteq \mathcal{B}$ la sous-tribu des $n$-premières coordonnées.

\subsection{Existence de mesures limites}

Le résultat suivant définit les mesures limites associées à $\nu$.
\begin{lemme}\label{meslim}
Il existe une application mesurable $B \rightarrow \mathcal{M}^{\text{Rad}}(X), b \mapsto \nu_{b}$  telle que pour $\beta$-presque tout $b \in B$, on ait la convergence faible-$\star$ $(b_{1}\dots b_{n})_{\star} \nu \underset{n \to +\infty}{\rightarrow} \nu_{b}$ . On a de plus :

\begin{itemize}

\item
Pour $\beta$-presque tout $b \in B$, ${b_{1}}_{\star}\nu_{Tb} = \nu_{b}$ où $T$ dénote le shift sur $B$

\item
$\nu \geq \int_{B} \nu_{b} \,\,d\beta (b)$ 
 
 \end{itemize}
 
 %En particulier, si $\nu$ est $\mu$-ergodique, alors il existe $c \in [1, +\infty[$ tel que $\nu = c \int_{B} \nu_{b} \,\,d\beta (b)$.
\end{lemme}

\begin{rem*}
L'inégalité dans le point 2 est optimale au sens où il peut arriver qu'une mesure stationnaire non nulle aient ses mesures limites $\beta$-presque toutes nulles. C'est par exemple le cas pour la mesure de Babillot-Bougerol-Elie pour une marche affine sur $\R$ à Lypaunov nul sans point fixe et admettant un moment d'ordre $2+\varepsilon$. Une preuve sera donnée dans \cite{TBthese}.
\end{rem*}

\begin{proof}[Preuve du \cref{meslim}]
On note $C_{c}(X)$ l'ensemble des fonctions réelles continues  sur $X$ à support compact. Soit $f \in C_{c}(X)$, et pour $n\geq 0$, posons $\varphi_{n} : B \rightarrow \mathbb{R}, \,b \mapsto (b_{1}\dots b_{n})_{\star}\nu\,(f)$. La stationnarité de $\nu$ entraine que la suite de fonctions $(\varphi_{n})_{n \geq 0}$ est une martingale par rapport à la filtration $(\mathcal{B}_{n})_{n \geq 0}$. Elle est de plus uniformément bornée dans $L^1(B,\beta)$ (avec $\int_{B} | \varphi_{n}| \,d\beta \leq \nu(|f|)$). Il existe donc $\varphi_{\infty} \in L^1(B,\beta)$ tel que $\varphi_{n} \underset{ \to +\infty}{\rightarrow} \varphi_{\infty}$ $\beta$-ps. 
 
 On se donne $(K_{n})_{n \geq 0}$ une suite exhaustive de compacts de $X$ et $D \subseteq C_{c}(X)$ un $\Q$-sous espace vectoriel dénombrable tel que pour chaque $n \geq 0$, on ait $\{f \in D, \,\, \text{supp}(f) \subseteq K_{n} \}$ dense dans $C_{K_{n}}(X):=\{f \in C_{c}(X), \,\, \text{supp}(f) \subseteq K_{n} \}$. On demande aussi que $D$ vérifie ces propriétés de densité quand on se restreint à considérer les fonctions non négatives.  D'après le premier paragraphe, il existe $B' \subseteq B$ mesurable plein tel que pour tout $b \in B'$, $f \in D$, la limite de la suite $((b_{1}\dots b_{n})_{\star}\nu\,(f))_{n \geq 0}$ soit définie, on la note $\nu_{b}(f)$.

 \bigskip
 
 Soit $b \in B'$. La fonction $\nu_{b} : D \rightarrow \R$  est une application $\Q$-linéaire non négative sur $D$. On vérifie qu'on peut la prolonger en une forme linéaire non négative $ \nu_{b} : C_{c}(X) \rightarrow \R$ (non nécessairement continue). Pour cela, il suffit de vérifier que pour tout $n \geq 0$, la restriction $\nu_{b} : D\cap C_{K_{n}}(X) \rightarrow \R$ se prolonge en une forme linéaire non négative continue sur $\nu_{b} : C_{K_{n}}(X) \rightarrow \R$. Les hypothèses de densités sur $D$ font que les différents prolongements coïncident et donnent le résultat. 
 
On se donne donc $n \geq 0$ et on vérifie que $\nu_{b |D\cap C_{K_{n}}(X)}$  se prolonge en une forme linéaire non négative continue  $\nu_{b} : C_{K_{n}}(X) \rightarrow \R$. D'après les hypothèses de densité sur $D$, il suffit de vérifier que l'application $\Q$-linéaire $\nu_{b |D\cap C_{K_{n}}(X)}$ est continue.  Soit $n \geq 0$ et $(f_{k})_{k \geq 0} \in (D \cap C_{K_{n}}(X))^\N$ une suite de fonctions de $D$ à support dans $K_{n}$ telle que $||f_{k}||_{\infty} \rightarrow 0$ quand $k \rightarrow +\infty$. On montre que $\lim_{k\to+\infty}\nu_{b}(f_{k}) = 0$. Pour cela, on se donne une fonction non négative $p \in D$ tel que  $p_{|K_{n}} \geq 1$.  On a alors 
 \begin{align*}
  \nu_{b}(f_{k})&= \lim_{n} (b_{1} \dots b_{n})_{\star} \nu (f)\\
  & \leq  \lim_{n} (b_{1} \dots b_{n})_{\star} \nu (||f_{k}||_{\infty} p)\\
  &=||f_{k}||_{\infty} \nu_{b}(p)\\
 \end{align*}
entrainant   $\nu_{b}(f_{k}) \underset{k\to +\infty}{\rightarrow} 0$

 On peut donc  prolonger $\nu_{b}$ en une forme linéaire non négative  $ \nu_{b} : C_{c}(X) \rightarrow \R$. D'après le théorème de représentation de Riesz, cette forme linéaire provient d'une mesure de Radon, que l'on note également $\nu_{b}$. Remarquons que cette mesure est indépendante du prolongement choisi  car déterminée par ses valeurs sur $D$.  
 
 On a ainsi défini une application $B' \rightarrow \mathcal{M}^{Rad}(X), b \mapsto \nu_{b}$. Cette application est bien mesurable (car si $f \in D$, $b \mapsto \nu_{b}(f)$ est mesurable comme limite, puis cela est vrai pour tout $f \in C_{c}(X)$ par passage à la limite). On la prolonge mesurablement à $B$ de fa\c con quelconque. 
 
 \bigskip
 
 Vérifions que l'application $B \rightarrow \mathcal{M}^{\text{Rad}}(X), b \mapsto \nu_{b}$ convient.
 
 Soit $b \in B'$,  $f_{0} \in C_{c}(X)$ on montre que $(b_{1}\dots b_{n})_{\star} \nu \,(f_{0}) \underset{n \to +\infty}{\rightarrow} \nu_{b}(f_{0})$. On se donne $n_{0} \geq 0$ tel que $K_{n_{0}}$ contienne le support de $f_{0}$ et $p \in D$ tel que $p\geq 0$, $p_{|K_{n_{0}}} \geq 1$.  Soit $\epsilon >0$. Soit $f \in D \cap C(K_{n_{0}})$ telle que $||f_{0} - f||_{\infty} <  \epsilon/(3 \nu_{b}(p))$. On a pour tout $n \geq 0$ que :
 
 \begin{align*}
 |\nu_{b}(f_{0}) - (b_{1}\dots b_{n})_{\star} \nu \,(f_{0}) | & \leq |\nu_{b}(f_{0}-f)| + |\nu_{b}(f ) - (b_{1}\dots b_{n})_{\star} \nu \,(f)| + |(b_{1}\dots b_{n})_{\star} \nu \,(f-f_{0})| \\
        & \leq  \epsilon/3 + |\nu_{b}(f ) - (b_{1}\dots b_{n})_{\star} \nu \,(f)| + ||f_{0} - f||_{\infty} \,(b_{1}\dots b_{n})_{\star} \nu \,(p)\\
\end{align*}

On obtient donc que pour $n \geq 0$ assez grand, on a $|\nu_{b}(f_{0}) - (b_{1}\dots b_{n})_{\star} \nu \,(f_{0}) | \leq \epsilon$, d'où la convergence.

Vérifions maintenant les deux points du lemme :

\begin{itemize}

\item Soit $b \in B'$ tel que $Tb \in B'$. Montrons que $(b_{1})_{\star}\nu_{Tb} = \nu_{b}$. Soit $f \in C_{c}(X)$. D'après le paragraphe précédent, on a 
$$(b_{1})_{\star}\nu_{Tb}(f) = \lim (b_{2}\dots b_{n})_{\star} \nu \,(f(b_{1}.)) = \lim (b_{1}\dots b_{n})_{\star} \nu \,(f) = \nu_{b}(f)$$
 D'où l'équivariance.

\item Notons $\nu' := \int_{B} \nu_{b} \,\,d\beta (b)$. C'est une mesure positive sur $X$. Soit $f \in C_{c}(X)$. On a 
 \begin{align*}
 \nu'(f) &= \int_{B} \nu_{b}(f) \,\,d\beta (b)\\
 & = \int_{B} \lim (b_{1}\dots b_{n})_{\star} \nu\,(f) \,\,d\beta (b)\\
 & \leq \liminf \int_{B}  (b_{1}\dots b_{n})_{\star} \nu\,(f) \,\,d\beta (b)\tag{lemme de Fatou }\\
 & = \nu(f)   \tag{stationnarité de $\nu$}\\
 \end{align*} 
On en déduit que $\nu'$ est de Radon, puis que $\nu' \leq \nu$. 
\end{itemize}

\end{proof}

\subsection{Conditionner la marche par un temps d'arrêt}

Ce paragraphe affirme qu'une mesure de Radon $\mu$-stationnaire dont les mesures limites ne sont pas toutes nulles et également $\mu_{\tau}$-stationnaire où $\mu_{\tau}$ désigne la mesure $\mu$ conditionnée par un temps d'arrêt $\tau$.

Plus précisément, soit  $\tau : B \rightarrow \N\cup\{\infty\}$ un temps d'arrêt pour la filtration $(\mathcal{B}_{n})_{n \geq 0}$ des premières coordonnées. On suppose $\tau$ fini $\beta$-ps. On note $\mu_{\tau} \in \mathcal{P}(\Gamma)$ la loi de $b_{1}\dots b_{\tau(b)}$ quand $b$ varie selon $\beta$. Attention, $\mu_{\tau}$ ne correspond pas toujours à la loi de $b_{\tau(b)}\dots b_{1}$, et une $\mu_{\tau}$-trajectoire n'est pas une $\mu$-sous trajectoire à priori.

\begin{lemme}\label{cond}
Si on a l'égalité $\nu = \int_{B}\nu_{b}\,d\beta(b)$, alors la mesure $\nu$ est $\mu_{\tau}$-stationnaire.
\end{lemme}

\begin{proof}
Pour $i \geq 1$, on note $\tau^1= \tau$ et $\tau^{i+1}(b):= \tau^{i}(b) + \tau(T^{\tau^{i}(b)}(b))$. On pose ensuite $\rho_{\tau} : B \rightarrow B, b \mapsto (b_{1}\dots b_{\tau(b)}, b_{\tau(b)+1}\dots b_{\tau^2(b)}, \dots)$. La loi image $\beta_{\tau}:= (\rho_{\tau})_{\star}\beta = (\mu_{\tau})^{\N^\star}$. Pour $\beta_{\tau}$ presque tout $a \in B$, la suite $(a_{1}\dots a_{n})_{\star} \nu$ converge vers une mesure $\nu_{a}$ (qui est une des mesures limites associées à $\nu$). 

Vérifions la $\mu_{\tau}$-stationnarité. On a $\nu = \int_{B} \nu_{b} \,\, d\beta(b)= \int_{B} \nu_{\rho_{\tau}(b)} \,\, d\beta(b)$ car $\nu_{b}=\nu_{\rho_{\tau}(b)}$ $\beta$-ps, donc 
\begin{align*}
\int_{\Gamma} g_{\star}\nu \,\,d\mu_{\tau}(g) &= \int_{B}\int_{B} b'_{1}\dots b'_{\tau(b)}\nu_{b} \,\,d\beta(b)\,d\beta(b') \\
&= \int_{B}\int_{B} \nu_{b'_{1}\dots b'_{\tau(b)} b} \,\,d\beta(b)\,d\beta(b') \\
&= \int_{B}\int_{B} \nu_{\rho_{\tau}(b')_{1} \rho_{\tau}(b)} \,\,d\beta(b)\,d\beta(b') \\
&= \int_{B} \nu_{\rho_{\tau}(b)} \,\,d\beta(b) \\
&=\nu
\end{align*}
\end{proof}

\newpage

\bibliographystyle{abbrv}

\bibliography{bibliographie}

\nocite{BFLM}

\end{document}